\documentclass[12pt]{amsart}
\usepackage[utf8]{inputenc}
\usepackage{lmodern}

\usepackage{mathtools, amsmath, amsthm, amssymb, bm, centernot, cancel, scalerel}
\usepackage{enumerate}
\usepackage{stmaryrd, algorithm, algpseudocode}
\usepackage{enumitem, hyperref, caption, accents, soul}
\usepackage{multirow, multicol, tabularx}
\usepackage{subfigure}
\usepackage{tikz, pgfplots, pgfplotstable}
\usetikzlibrary{arrows, arrows.meta}
\usetikzlibrary{angles, quotes}

    \usepackage[left=0.75in, right=0.75in, top=1in, bottom=1in]{geometry}
    \usepackage{enumitem}
    \usepackage{lipsum}
    \setlist[itemize]{leftmargin=*}

    \usepackage[
    backend=biber,
    style=alphabetic,
    sorting=nyt,
    giveninits=true
    ]{biblatex}
\AtEveryBibitem{
    \clearfield{urlyear}
    \clearfield{urlmonth}
    \clearfield{url}
    \clearfield{doi}
    \clearfield{issn}
    \clearfield{isbn}
    \clearfield{note}
    \clearfield{eprint}
}

    \hypersetup{
        colorlinks=true, 
        linkcolor=blue,
        citecolor=red,
    }

    \newtheorem{theorem}{Theorem}[section]  
    \newtheorem{definition}[theorem]{Definition}
    \newtheorem{corollary}[theorem]{Corollary}
    \newtheorem{lemma}[theorem]{Lemma}
    \newtheorem{remark}[theorem]{Remark}
    \newtheorem{proposition}[theorem]{Proposition}
    \newtheorem{assump}[theorem]{Assumption}

    \numberwithin{equation}{section}

    \let\pa\partial

    \newcommand{\R}{\mathbb R}
    \newcommand{\bA}{\mathbf A}
    \newcommand{\bB}{\mathbf B}

    \newcommand{\bE}{\mathbf E}

    \newcommand{\bM}{\mathbf M}

    \newcommand{\bS}{\mathbf S}

    \newcommand{\bn}{\mathbf n}
    
    \newcommand{\bp}{\mathbf p}

    \newcommand{\bx}{\mathbf x}

    \newcommand{\G}{\Gamma}

    \newcommand{\T}{\mathcal T}

    \newcommand{\nablaG}{\nabla_{\Gamma}}

    \newcommand{\cT}{\mathcal T}
    \newcommand{\cP}{\mathcal P}

    \newcommand{\cL}{\mathcal L}

    \newcommand{\rr}{\mathbb{R}}

    \DeclareGraphicsExtensions{.pdf,.eps,.ps,.eps.gz,.ps.gz,.eps.Y}

    \newcommand{\mV}{\mathbb{V}}
    \newcommand{\mU}{\mathbb{U}}

    \newcommand{\Cone}{C_{\mathrm{stab},1}}
    
    \newcommand{\Cneg}{C_{\mathrm{stab},-1}}

    \newcommand{\Cinvh}{{C_{\mathrm{inv}, h}}}

    \newcommand{\cbm}{c_b^-}
    \newcommand{\Cbp}{C_b^+}
    \newcommand{\Cbh}{C_{b,h}}
    \newcommand{\cbh}{c_{b,h}}

    \usepackage{mathtools}

    \newcommand\wwidehat[1]{\arraycolsep=0pt\relax
    \begin{array}{c}
    \stretchto{
      \scaleto{
        \scalerel*[\widthof{\ensuremath{#1}}]{\kern-.5pt\bigwedge\kern-.5pt}
        {\rule[-\textheight/2]{1ex}{\textheight}} 
      }{\textheight}  
    }{0.5ex}\\
    #1\\
    \rule{*ex}{0ex}
    \end{array}
    }

    \def\XXint#1#2#3{{\setbox0=\hbox{$#1{#2#3}{\int}$ }
    \vcenter{\hbox{$#2#3$ }}\kern-.6\wd0}}

    \pgfplotsset{compat=1.18}

\usepackage[normalem]{ulem}

\title{Inf-Sup Stability of Parabolic TraceFEM}
\author{L. Bouck, R. H. Nochetto, M. Shakipov, V. Yushutin}
\date{\today}

\begin{document}

\begin{abstract}
    We develop a parabolic inf-sup theory for a modified TraceFEM semi-discretization in space of the heat equation posed on a stationary surface embedded in $\R^n$. We consider the normal derivative volume stabilization and add an $L^2$-type stabilization to the time derivative. We assume that the representation of and the integration over the surface are exact, however, all our results are independent of how the surface cuts the bulk mesh. For any mesh for which the method is well-defined, we establish necessary and sufficient conditions for inf-sup stability of the proposed TraceFEM in terms of $H^1$-stability of a stabilized $L^2$-projection and of an inverse inequality constant that accounts for the lack of conformity of TraceFEM. Furthermore, we prove that the latter two quantities are bounded uniformly for a sequence of shape-regular and quasi-uniform bulk meshes. We derive several consequences of uniform discrete inf-sup stability, namely uniform well-posedness, discrete maximal parabolic regularity, parabolic quasi-best approximation, convergence to minimal regularity solutions, and optimal order-regularity energy and $L^2 L^2$ error estimates. We show that the additional stabilization of the time derivative restores optimal conditioning of time-discrete TraceFEM typical of fitted discretizations.
\end{abstract}

\maketitle


\setcounter{tocdepth}{2} 

%%%%%%%%%%%%%%%%%%%%%%%%%%%%%%%%%%%%%%%%%%%%%%%%%%%%%%%%%%%%%%%%%%%%%%%%%%%%%%%%%%%%%%%
\section{Introduction}\label{sec:introduction}
%%%%%%%%%%%%%%%%%%%%%%%%%%%%%%%%%%%%%%%%%%%%%%%%%%%%%%%%%%%%%%%%%%%%%%%%%%%%%%%%%%%%%%%

    We consider the heat equation $\pa_t u - \Delta_\G u = f, u(0) = u_0$, posed on a closed $C^2$-manifold $\G$ embedded in $\mathbb{R}^n$ of codimension 1, as a prototype of a linear surface parabolic problem. We study the space semi-discretization induced by a TraceFEM with normal derivative stabilization, augmented with an $L^2$-type stabilization of the time derivative. The latter is intimately related to a stabilized $L^2$-projection operator $P_h$. We establish necessary and sufficient conditions for parabolic inf-sup stability in terms of $H^1$-stability of $P_h$ and an inverse inequality constant that accounts for the lack of conformity of TraceFEM. We also derive several novel consequences of uniform inf-sup stability including convergence to the solution $u$ under minimal regularity of $(f,u_0)$, energy and $L^2L^2$ optimal order-regularity error estimates as well as an optimal condition number.

    This introduction motivates this endeavor, discusses the TraceFEM approximation of the surface heat equation with additional $L^2$ stabilization of the time derivative, defines the associated operator $P_h$, and presents a comprehensive summary of our main results.
    
    %----------------------------------------------------------------------------------
    \subsection{TraceFEM}
    %----------------------------------------------------------------------------------
        Trace Finite Element Method or TraceFEM is an Eulerian method of discretizing PDEs posed on a surface $\G$ embedded in $\rr^n$ represented by a level set function $\phi$. The method was proposed in \cite{olshanskii_finite_2009} for the two-phase flow and coupled bulk-interface problems. In this approach, instead of meshing the surface directly, the volumetric ambient space $\Omega_h$ that $\G$ is embedded into is discretized instead. The discrete problem is then solved on the surface $\G$ tested against discrete functions defined in the bulk $\Omega_h$.

        Due to the disparity between the continuous and discrete problem, the method in its original form  proposed in \cite{olshanskii_finite_2009} is not, in general, well-posed. In attempts to improve the stability of the method, several stabilizations were introduced; 
        see \cite[Section~3]{olshanskii_trace_2017} for a list of references. In this work, we shall focus on the \textit{normal derivative volume stabilization} introduced in \cite{burman_cut_2018}, and later proven to be amenable to higher-order elements and to admit optimal condition number for the discretization of elliptic problems \cite{grande_analysis_2018}.
        
        Higher-order TraceFEM is a nontrivial matter because a practical algorithm requires the entire surface $\G$ to be approximated. Even the property that $P_1$ approximation of $\G$ is sufficient for optimal-order error estimates with $P_1$ elements requires a careful argument, first presented in \cite{dziuk_finite_1988}. Higher-order TraceFEM can be realized via the so-called isoparametric mappings introduced and analyzed in \cite{lehrenfeld_high_2016}. The geometric consistency analysis of TraceFEM is technical but well-understood.

        Therefore, for simplicity, we assume that the surface $\G$ is stationary, there is no geometric error for the quadrature on $\G$ and for the normal to the surface, and there is no time discretization. This paper focuses on the most significant difficulty for parabolic TraceFEM -- stability and convergence under minimal regularity.

    %----------------------------------------------------------------------------------
    \subsection{Motivation}
    %----------------------------------------------------------------------------------
        This study comes from modeling and computational work on the surface Beris-Edwards model \cite{bouck_hydrodynamical_2024}. This model describes the coupling of the $Q$-tensor description of a liquid crystal with an incompressible fluid on a curved thin film. The latter is modeled as a surface embedded in $\mathbb{R}^3$, and the $Q$-tensor is subject to a gradient flow dynamic while being transported by the fluid. The incompressible flow is tangential to the surface, and is governed by the surface Navier-Stokes equations nonlinearly coupled with the $Q$-tensor. Since the full model is thermodynamically consistent, it obeys an energy law that reveals minimal regularity.
        Further regularity of solutions is unlikely due to the strong nonlinear coupling. This makes existence of solutions a challenging problem that is still open, a fact shared by other models of liquid crystals on surfaces \cite{napoli2016hydrodynamic, nestler2022active, nitschke2023beris,pearce2019geometrical}. Such models are crucial for the design and study of novel materials and active matter \cite{keber2014topology}.

        The numerical method used in \cite{bouck_hydrodynamical_2024} to computationally study the model hinges on the parabolic TraceFEM. This method relies on the extrinsic description of the model used in its derivation, and utilizes calculus in a tubular neighborhood. In this context, numerical analysis of TraceFEM requires dealing with minimal regularity, a study missing in the existing literature. Given the complexity of the surface Beris-Edwards model \cite{bouck_hydrodynamical_2024}, we were naturally led to the question of stability and convergence of TraceFEM under minimal regularity for the heat equation. We view this endeavor as a preliminary step to analyzing TraceFEM for the surface Beris-Edwards model.
        
    %----------------------------------------------------------------------------------
    \subsection{The surface heat equation}
    %----------------------------------------------------------------------------------
        To narrow down the regularity and conditioning issues faced when discretizing the surface Beris-Edwards model with TraceFEM, in this paper we instead focus on the simplest diffusion process on a fixed closed surface $\G$ governed by the heat equation: given initial data $u_0$ and forcing $f$, we seek a function $u$ that satisfies 
        \begin{equation}\label{eq:model_problem}
            \pa_t u -\Delta_\G u = f \quad  \text{on} \quad \G \times (0, T), \qquad \text{and} \quad u(0) = u_0 \quad  \text{on} \quad \G.
        \end{equation}
        This problem, and its generalizations that allow $\G$ to evolve in time, already appeared in the literature of TraceFEM \cite{olshanskii_error_2014, hansbo_characteristic_2015, lehrenfeld_stabilized_2018}, of which only the last two include stabilization, and only the last one extends to higher-order. None of them develop an inf-sup theory, and corresponding quasi-best approximation and minimal regularity estimates, which require a uniform bound on the negative
        Sobolev norm $\|\pa_t u_h\|_{H^{-1}_\G}$ for the discrete solution $u_h$. Moreover, the condition
        number estimate for the matrix $\bB$ arising from any practical implicit space-time discretization
        of \eqref{eq:model_problem} with TraceFEM on quasi-uniform meshes reads $\kappa(\bB) \lesssim \frac{h^2}{\Delta t} + \frac{\Delta t}{h^2}$ \cite{lehrenfeld_stabilized_2018}, where $\Delta t$ is the time step, and $h$ is the mesh size. Even though this estimate is sharp, its dependence on $\Delta t$ is problematic for $\Delta t \ll h^2$ as it could happen for high-order schemes in time and space-time adaptive schemes. 

        Investigating these two issues is not a matter of sharper analysis. The essence of the problem is the current formulation of TraceFEM that lacks sufficient stabilization of the time derivative. In contrast, researchers working on CutFEM, a related unfitted FEM \cite{burman_cutfem_2015}, figured out the importance of time derivative stabilization in the context of conditioning for the wave equation \cite{sticko_stabilized_2016}.
        However, neither community seems to have realized the intimate connection between time derivative
        stabilization and $H^1$-stability of a stabilized $L^2$-projection. Revealing this connection, and studying several consequences (both analytical and algebraic) is the main purpose of this paper.

        To highlight the key properties of TraceFEM related to the unfittedness of the discretization in space, we focus on a space {\it semidiscretization} of \eqref{eq:model_problem}. Let $\{\widetilde{\cT}_h\}_h$ be a sequence of meshes surrounding $\Gamma$, and let $T \in \widetilde{\cT}_h$ denote a generic element of diameter $h_T$. Let $\cT_h$ be the union of elements $T\in\widetilde{\cT}_h$ that intersect $\Gamma$, and the actual computational domain $\Omega_h$ be the union of elements from $\cT_h$.
        The finite element space $V_h$, subordinate to $\cT_h$, consists of continuous piecewise polynomial functions over $\Omega_h$. To simplify the analysis, we ignore hereafter the geometric error due to approximating $\G$, an important matter that deserves further investigation. We thus assume that discrete functions $v_h\in V_h$ can be evaluated on $\G$ and that their tangential gradient $\nabla_\G v_h$ and normal derivative $\nabla_\bn v_h$ to $\G$ can be computed exactly. For $\alpha \in \{1, 0, -1\}$ corresponding to the differentiability order of spaces $H^1_\G, L^2_\G$ and $H^{-1}_\G$ (the dual of $H^1_\G)$, we define the \textit{normal derivative volume stabilization} \cite{burman_cut_2018,larson_stabilization_2020} on possibly graded meshes $\cT_h$ as follows:
        \begin{equation}\label{eq:intro:stabilization}
            s_\alpha(v_h, w_h) := \sum_{T \in \cT_h} h_T^{1-2\alpha} \int_T (\nabla_\bn v_h) (\nabla_\bn w_h) \mathrm{d}x, \quad \forall v_h,w_h \in V_h.
        \end{equation}
        This leads to the following stabilized $L^2_*$-inner product and $H^1_*$-semi-inner product for all $v_h,w_h \in V_h$
        \begin{equation}\label{eq:stabilized-scalar-products}
            (v_h,w_h)_* := (v_h,w_h)_\G + s_0(v_h,w_h),
            \qquad
            a_* (v_h,w_h) := (\nabla_\G v_h,\nabla_\G w_h)_\G + s_1(v_h,w_h).
        \end{equation}
        Then our semidiscrete, stabilized TraceFEM for the linear parabolic problem \eqref{eq:model_problem} reads: given $u_0$ and $f$, find $u_h(t) \in V_h$ such that for a.e. $t \in (0,T)$
        \begin{equation} \label{eq:intro:our_method_1}
            (\pa_t u_h, v_h)_* + a_*( u_h, v_h) = \langle f, v_h \rangle_\G, \quad \forall v_h \in V_h,
        \end{equation}
        where $\langle \cdot, \cdot \rangle_\G,$ stands for the duality pairing $H^{-1}_\G-H^1_\G$. We stress that adding $s_0(\pa_t u_h, v_h)$ to the leftmost term of \eqref{eq:intro:our_method_1} is our sole modification to the standard TraceFEM \cite{hansbo_characteristic_2015, lehrenfeld_stabilized_2018}, but it has lasting consequences which ensue from the properties of the operator $P_h$ discussed next.

    %----------------------------------------------------------
    \subsection{Stabilized \texorpdfstring{$L^2$}{L2}-projection}
    %----------------------------------------------------------
    
        We first observe that the stabilized $L^2$-inner product in \eqref{eq:stabilized-scalar-products} induces the following {\it stabilized $L^2$-projection} $P_h: H^{-1}_\G\to V_h$: given $u\in H^{-1}_\G$ let $P_hu\in V_h$ satisfy
        \begin{equation} \label{eq:intro:l2proj}
           (P_h u, v_h)_* = \langle u, v_h \rangle_\G, \quad \forall v_h \in V_h.
        \end{equation}
        Such an operator $P_h$ has been considered earlier in \cite{berre_cut_2023}, along with a face-based stabilization and normal derivative volume stabilization in the context of a cell-by-cell EMI electrophysiology model. However, \cite{berre_cut_2023} does not embark on an analysis of $P_h$.
        Similar constructions to \eqref{eq:intro:l2proj} can also be found in 
        \cite{hansbo_stabilized_2015, gurkan_stabilized_2020, ulfsby_stabilized_2023}, and most notably \cite{burman_stabilized_2015} which proves $H^1$-stability of the ghost penalty-stabilized $L^2$-projection and thereby derives discrete inf-sup stability for the discretization of the stationary three-field Stokes problem. 

        For linear parabolic problems in flat domains, fitted by the mesh, Tantardini and Veeser \cite{tantardini_l2-projection_2016} showed that $H^1$-stability of the $L^2$-projection on the discrete space $V_h$ is both a necessary and sufficient condition for discrete inf-sup stability of Galerkin approximations in space.
        It turns out that {\it uniform} $H^1$-stability of the $L^2$-projection on $V_h$ is an implicit condition on the mesh sequence $\{\cT_h\}_h$. Its validity on quasi-uniform, shape-regular meshes is a consequence of inverse estimates and is considered standard. Its generalization to graded meshes, however, has been the subject of study for decades \cite{crouzeix_stability_1987, bramble_stability_2002, carstensen_merging_2001, bank_h1_2014}, culminating with \cite{diening_sobolev_2021} that establishes such a result on graded meshes generated by newest vertex bisection.

        Uniform discrete inf-sup stability is the ``strongest" stability result one can aim at for this class of problems and a given space discretization scheme. In fact, it readily leads to quasi-best approximation and convergence within the minimal regularity framework $u \in L^2 H^1 \cap H^1 H^{-1}$. A natural but delicate question, addressed in this paper, is whether unfitted TraceFEM possesses such parabolic inf-sup property. Two notable differences with \cite{tantardini_l2-projection_2016} are the presence of stabilization terms and the absence of a boundary.

        We develop an inf-sup theory for general bulk meshes $\{\cT_h\}_h$, without requiring shape-regularity or quasi-uniformity. However, proving {\it uniform} bounds on the inf-sup constant w.r.t $h$ is a delicate matter equivalent to showing uniform $H^1_*$-stability of $P_h$ and boundedness of an inverse constant $C_{\textrm{inv},h}$ that accounts for the lack of conformity of TraceFEM. We establish the latter for shape-regular and quasi-uniform bulk meshes. Since $\G$ is unfitted relative to $\cT_h$, enforcing shape-regularity of $\{\cT_h\}_h$ is not a restrictive assumption. Dealing with graded meshes may be advantageous in enhancing local resolution, but it is beyond the scope of this paper, and so this issue remains open.
      
        A brief literature review is now in order. Continuous inf-sup stability for the evolving surface heat equation is derived in \cite{olshanskii_eulerian_2014}, with the inf-sup constant decaying exponentially with $T$. Our inf-sup stability constant for a fixed surface heat equation decays as $1/T$ for general data $(u_0,f)$ and is insensitive to $T$ for data with vanishing mean value.
        Discrete inf-sup stability is shown in \cite{olshanskii_error_2014} for an evolving-surface advection-diffusion equation. However, the norm of the discrete solution is mesh-dependent, in contrast to our robust inf-sup framework. 
        Finally, spacetime CutFEM for a parabolic problem on an evolving domain is considered in 
        \cite{larson_space-time_2024}, and an inf-sup condition is derived under additional regularity of the solution and test space.
        To the best of our knowledge, we present the first uniform discrete inf-sup stability analysis for stabilized TraceFEM discretizations of parabolic problems.

     %----------------------------------------------------------------------------------------
    \subsection{Our main results}\label{sec:main-results}
     %----------------------------------------------------------------------------------------

        Inspired by \cite{tantardini_l2-projection_2016}, we consider the following operator norms for $P_h$:
        \begin{equation} \label{eq:intro:operator_norms}
            \|P_h\|_{\cL(H^1_\G)} := \sup_{v \in H^1(\G)} \frac{\|P_h v\|_{H^1(\G)}}{\|v\|_{H^1(\G)}}, \quad \|P_h\|_{\cL(H^1_*)} := \sup_{v \in H^1(\G)} \frac{\|P_h v\|_{H^1_*}}{\|v\|_{H^1(\G)}},
        \end{equation}
        where
        \[
        \|P_h v\|_{H^1_*}^2 := a_*(P_h v,P_h v) = \|P_h v\|_{H^1(\G)}^2 + s_1(P_h v, P_h v)
        \]
        is the stabilized $H^1_*$-norm on $V_h$; hence $\|P_h\|_{\cL(H^1_\G)}\le \|P_h\|_{\cL(H^1_*)}.$
        We are able to prove that $P_h$ is uniformly stable in the $H^1_*$-norm for any sequence of quasi-uniform, shape-regular meshes $\{\cT_h\}_h$.

        We prove a discrete parabolic inf-sup condition for the following Petrov-Galerkin formulation of \eqref{eq:intro:our_method_1}. Let data $(u_0, f) \in L^2_\G \times L^2_t H^{-1}_\G$ have minimal regularity and let $\mU_h := H^1_t V_h$ and $\mV_h := V_h \times L^2_t V_h$ be the semidiscrete trial and test spaces. Let $I:=(0,T)$ be a fixed time interval with $T<\infty$, and let integrals over $I$ be written without the differential $dt$. We seek $u_h\in\mU_h$ such that
        \begin{equation} \label{eq:intro:petrov_galerkin_semidiscrete}
            b_h(u_h, v_h) := (u_h(0), v_{h,0})_* + \int_I \langle \pa_t u_h, v_{h,1} \rangle_* + a_*(u_h, v_{h,1}) = (u_0, v_{h,0})_\G + \int_I \langle f, v_{h,1} \rangle_\G
        \end{equation}
        for all $v = (v_{h,0}, v_{h,1}) \in \mV_h$. We define the norms on the trial and test spaces $\mU_h, \mV_h$ to be
        \begin{equation} \label{eq:intro:1*}
            \|u_h\|_{1,*}^2 := \|u_h(0)\|_{L^2_*}^2 + \int_I \|\pa_t u_h\|_{H^{-1}_*}^2 + |u_h|_{H^1_*}^2, \quad             \|v_h\|_{2,*}^2 := \|v_{h,0}\|_{L^2_*}^2 + \int_I \|v_{h,1}\|_{H^1_*}^2
        \end{equation}
        Let $c_* \le C_*$ be the discrete inf-sup and continuity constants for \eqref{eq:intro:petrov_galerkin_semidiscrete}-\eqref{eq:intro:1*}, namely
        \begin{equation}\label{eq:*-constants}
            c_* := \inf_{u_h \in \mU_h} \sup_{v_h \in \mV_h} \frac{b_h(u_h, v_h)}{\|u_h\|_{1,*} \|v_h\|_{2,*}},
            \qquad
            C_* := \sup_{u_h \in \mU_h} \sup_{v_h \in \mV_h} \frac{b_h(u_h, v_h)}{\|u_h\|_{1,*} \|v_h\|_{2,*}}.
        \end{equation}
        In addition, let $c_b \le C_b$ be the continuous stability constants for the Petrov-Galerkin formulation of the problem \eqref{eq:model_problem} on $\Gamma$ which can be seen as the continuous counterparts of $c_*$, $C_*$ upon omitting all discrete stabilizations $s_\alpha$ in \eqref{eq:intro:petrov_galerkin_semidiscrete} and \eqref{eq:intro:1*}. The continuous problem is well-posed so that the bounds 
        \[0 < c_b^- \leq c_b \,,\quad C_b \leq C_b^+\]
        exist for some constants  $c_b^-,C_b^+$.
        
        We now list our inf-sup estimate, several consequences, and briefly comment on them.

        \begin{itemize}
           \item {\it Discrete inf-sup condition}. The inf-sup and continuity constants, 
           $c_*, C_*$ \eqref{eq:*-constants} for the proposed semidiscrete scheme, as well as $c_b^-, C_b^+$ for the continuous problem \eqref{eq:model_problem}, and the operator norms of $P_h$ in \eqref{eq:intro:operator_norms} satisfy the following relations on any fixed mesh $\cT_h$ that resolves the geometry of $\G$ of  class $C^2$: 
            \begin{equation}\label{eq:intro:main_result}
                 \frac{c_b^-}{\|P_h\|_{\cL(H^1_*)} + \Cinvh} \le c_*
                 \le \frac{C_b^+}{\|P_h\|_{\cL(H^1_\G)}}, \qquad C_* \leq C_b^+;
            \end{equation}
            here $\Cinvh$ is the constant \eqref{eq:Cinvh} of an inverse inequality between the dual space $V_h^{-1}$ of $V_h$ with the norm $L^2_*$. Estimate \eqref{eq:intro:main_result} extends a similar one for linear parabolic problems in flat, fitted domains by Tantardini and Veeser \cite{tantardini_l2-projection_2016}, does not require a priori assumptions on the mesh sequence $\{\cT_h\}_h$, and can be interpreted as follows. If $\{\cT_h\}_h$ is well-behaved in the sense that both $\|P_h\|_{\cL(H^1_*)}$ and $\Cinvh$ are uniformly bounded, then $c_*$ is uniformly bounded from below: this gives a {\it sufficient condition} for discrete {\it uniform} inf-sup stability of the stabilized TraceFEM scheme \eqref{eq:intro:petrov_galerkin_semidiscrete}. Complementing this result, we are able to prove that if $\{\cT_h\}_h$ is shape-regular and quasi-uniform, then  such uniform bounds on $\|P_h\|_{\cL(H^1_*)}$ and $\Cinvh$ are valid. On the other hand, if $\|P_h\|_{\cL(H^1_\G)} \to \infty$ for $\{\cT_h\}_h$, then $c_*\to0$  as $h\to0$ and \eqref{eq:intro:petrov_galerkin_semidiscrete} is not inf-sup stable; this provides a {\it necessary condition} for inf-sup stability of \eqref{eq:intro:petrov_galerkin_semidiscrete}. This is
            the content of Theorem \ref{thm:Ph-c*}, whereas Theorem \ref{thm:inf-sup-stability} shows the optimal asymptotic dependence $c_b^- \approx T^{-1}$ for general data $(u_0,f)$ and
            $c_b^- \approx 1$ for data with vanishing meanvalue.
        
            \medskip
            \item {\it Discrete maximal parabolic regularity}. If $u_0 \in L^2_\G$ and $f\in L^2_t H^{-1}_\G$, then
             \[ 
               \|\Delta_h u_h\|_{L^2_t H^{-1}_\G} + \|\pa_t u_h\|_{L^2_t H^{-1}_\G} \leq C (\|f\|_{L^2_t H^{-1}_\G} + \|u_0\|_{L^2_\G}),
            \]
            where $\Delta_h: V_h \to V_h$ is the discrete Laplace-Beltrami operator and $C$ is proportional 
            to the left-hand side of \eqref{eq:intro:main_result}. This is the content of
            Proposition \ref{prop:max-regularity}. 

            \medskip
            \item {\it Quasi-best approximation}. An important but non-obvious issue is how to quantify the accuracy of TraceFEM, given that the continuous and the discrete solutions live in different geometries. We address this upon introducing the following {\it error functionals}:
            for $v:\G\to\R, v_h\in V_h$, let
            \begin{equation}\label{eq:intro:error_fnals}
                \begin{split}
                    \bE_{H^{-1}_*}[v, v_h]^2 & := \|v - v_h\|_{H^{-1}_\G}^2 + s_{-1}(v_h, v_h) \\
                    \bE_{L^2_*}[v, v_h]^2 & := \|v - v_h\|_{L^2_\G}^2 + s_{0}(v_h, v_h) \\
                    \bE_{H^1_*}[v, v_h]^2 & := |v - v_h|_{H^1_\G}^2 + s_{1}(v_h, v_h),
                \end{split}
            \end{equation}
            and for $v \in \mU=L^2_t H^1_\G \cap H^1_t H^{-1}_\G$ and $v_h \in \mU_h=H^1_t V_h$ let the total error functional be defined as
            \begin{equation}\label{eq:intro:energy_error_functional}
                \bE[v, v_h]^2 := \bE_{L^2_*}[v(0), v_h(0)]^2 + \int_I \bE_{H^{-1}_*}[\pa_t v, \pa_t v_h]^2 + \bE_{H^1_*}[v, v_h]^2.
            \end{equation}
            Theorem \ref{theorem:quasi-best-approx} shows that TraceFEM is quasi-optimal for this measure of accuracy, namely
            \begin{equation}\label{eq:intro-quasi-best}
                    \bE[u, u_h] \leq \left( 1 + \frac{C_*}{c_*} \right) \inf_{v_h \in \mU_h} \bE[u, v_h].
            \end{equation}
    
            \medskip
            \item {\it Convergence under minimal regularity}. Under mild assumptions on $\{\cT_h\}_h$, we prove in Propositions \ref{prop:convergence} and \ref{prop:strong-convergence} that if $u_0\in L^2_\G$ and $f\in L^2_t H^{-1}_\G$, then $u_h$ 
            converges strongly in $L^2_t H^1_\G$ and $H^1_t H^{-1}_\G$ to the unique solution $u$ of the continuous counterpart of \eqref{eq:intro:petrov_galerkin_semidiscrete}.
            
            \medskip
            \item {\it Optimal energy error estimates}. If $u_0\in H^1_\G$ and $f\in L^2_t L^2_\G$, and 
            $\{\cT_h\}_h$ are shape-regular and quasi-uniform meshes, then there exists a constant $C$ proportional to the constant in \eqref{eq:intro-quasi-best}
            such that the optimal order-regularity error estimate
            \[
               \bE[u,u_h]^2 \le C h^2 \Big( |u_0|_{H^1_\G}^2 + \int_\G \|f\|_{L^2_\G}^2 \Big)
            \]
            holds. This is established in Corollary \ref{corollary:error_estimate}.

            \medskip
            \item {\it Optimal $L^2_tL^2_*$-error estimate}. If $\{\cT_h\}_h$ are shape-regular and quasi-uniform meshes, then Theorem \ref{theorem:l2l2_error_estimate} proves the optimal order-regularity error estimate
            \begin{equation}\label{eq:a_priori_error_est_l2l2_with_data}
                \int_I \bE_{L^2_*}[u, u_h]^2 \le C h^2 \Big( \|u_0\|_{L^2_\G}^2 + \int_\G \|f\|_{H^{-1}_\G}^2 \Big),
            \end{equation}
            where the constant $C$ is independent of the inf-sup constant $c_*$. This estimate hinges on
            a duality argument from \cite{k_chrysafinos_error_2003} and a direct evaluation of $\int_I \bE_{H^1_*}[u, u_h]^2$ that by-passes
            \eqref{eq:intro-quasi-best}.
    
            \medskip
            \item {\it Optimal condition number}. For a family of shape-regular, quasi-uniform triangulations $\{\cT_h\}_h$, let $\bB_*$ be the matrix associated with the bilinear
            form
            \[
              B_*(u_h,v_h) = \frac{1}{\Delta t} (u_h,v_h)_* + a_*(u_h,v_h).
            \]
            The latter mimics one time step of an implicit method such as backward Euler. 
            Lemma \ref{lemma:conditioning} shows the condition number estimate
            \begin{equation}\label{eq:intro-cond-number}
                \kappa (\bB_*) \lesssim 1 + \frac{\Delta t}{h^2},
            \end{equation}
            which is robust with respect to $\Delta t \ll h^2$ and thus improves upon \cite{lehrenfeld_stabilized_2018}.
            
        \end{itemize}
  
        The outline of the paper is as follows. In Section \ref{sec:prelims}, we introduce a minimal assumption on the local mesh size and key quantities needed for our main result. We establish necessary and sufficient conditions for the inf-sup stability in Section \ref{section:inf-sup} and devote Section \ref{section:conseq} to exploring several consequences of uniform inf-sup property such as uniform well-posedness, maximal parabolic regularity, best approximation property, and convergence under minimal regularity. We provide a comprehensive study of the stabilized $L^2$-projection $P_h$ on quasi-uniform and shape-regular meshes in Section \ref{section:stabilized_projections}. In Section \ref{section:error_est}, we present optimal
        order-regularity error estimates for our semidiscretization. We finish with an estimate of the condition number in Section \ref{section:algebraic_stability} followed by concluding remarks in Section \ref{section:conclusion}.

\section{Preliminaries} \label{sec:prelims}
    %-------------------------------------------------------------------------------
    \subsection{Bulk mesh}
    %-------------------------------------------------------------------------------
        We recall that a sequence of meshes $\{\widetilde{\cT}_h\}_h$ contains $\G$ and that $\cT_h$ is a subset of $\widetilde{\cT}_h$ containing only elements $T\in\widetilde{\cT}_h$ that intersect $\G$. Let $\Omega_h$ be the union of all elements $T \in \cT_h$; this is the actual computational domain. Let $V_h$ be a continuous piecewise polynomial finite element space over $\cT_h$. We define $h$ to be the mesh size of $\cT_h$, i.e. the largest element diameter $h_T$ for $T\in\cT_h$. Hereafter, we assume that $\cT_h$ resolves $\G$.
            
        \begin{assump}[resolution of the geometry] \label{assump:res-of-geom}
            The stationary surface $\Gamma$ is of class $C^2$ and integration and evaluation on $\Gamma$ are exact. Moreover, any partition $\cT_h$ resolves $\G$ in the sense that
            \[h_T \leq C (\max_{i} \|\kappa_i\|_{L^\infty(\G \cap T)})^{-1}
            \qquad\forall T\in\cT_h,\] 
            where $\kappa_i, i=1,...,n-1$ are the principal curvatures of $\G$, and $C$ is a generic constant. 
        \end{assump}
        
        This assumption ensures that the closest point projection $\bp$, and the normal extension $v^e=v\circ\bp$ for any $v\in L^2_\G$ are well-defined in a tubular neighborhood of $\Gamma$, and thus in $\Omega_h$. In particular, $\bn^e=\bn\circ\bp$ is the unit normal extension which, abusing notation, we will denote $\bn$. This, in turn, guarantees that the TraceFEM stabilizations are well-defined. Moreover, Assumption \ref{assump:res-of-geom} explicitly states that we ignore the issue of geometry approximation in favor of clean inf-sup theory. The $C^2$-regularity of $\Gamma$ is minimal for the existence of the closest point projection
        \cite[Section 1.3.3]{bonito_finite_2020}, and is essential for the well-posedness of stabilized TraceFEM, as well as some of the arguments in the forthcoming analysis.
    
    %-------------------------------------------------------------------------------
    \subsection{Stabilization of TraceFEM}
    %-------------------------------------------------------------------------------

        The symbol $(\cdot, \cdot)_\Xi$ denotes the $L^2$-inner products over $\Xi \in \{\G, \Omega_h, T\}$. We denote by $\|\cdot\|_{H^k_\G}$ and $|\cdot|_{H^k_\G}$ the Sobolev $H^k(\G)$-norm and seminorm for $k\in \mathbb{N}$, and by $H^{-1}_\G$ the dual space of $H^{1}_\G$. We point out that $H^{1}_\G$ contains functions with non-vanishing mean and $\G$ is a closed manifold, so $H^{-1}_\G$ is not the usual space $H^{-1}(\G)$.
        We write $\langle \cdot, \cdot \rangle_\G$ for the duality pairing $\langle \cdot, \cdot \rangle_{H^{-1}_\G \times H^1_\G}$, and $a(\cdot, \cdot) := (\nablaG \cdot, \nablaG \cdot)_\G$ for the Dirichlet form on $\G$. Lastly, we indicate with $L^p_t$ the $L^p$-space on a time interval $(0,t)$.
        
        Our analysis is based on the following volumetric stabilization terms in $\R^n$ introduced in \cite{burman_cut_2018}.
        \begin{definition}[normal derivative volume stabilization] \label{def:stabilization}
            For $u_h, v_h \in V_h$, define for $j=-1,0,1$
            \begin{equation}\label{eq:stabilization}
                s_j (u_h, v_h) := \sum_{T \in \cT_h} s_{j,T} (u_h,v_h)\,,\quad s_{j,T}(u_h,v_h):=h_T^{\alpha_j} (\nabla_\bn u_h, \nabla_\bn v_h)_{T}, \quad \alpha_j = 1-2j
            \end{equation}
            where $\bn$ is the unit normal extension to $\Omega_h$ and $\nabla_\bn := \bn \cdot \nabla$ is the normal derivative to $\G$. 
        \end{definition}
        
        The above volumetric stabilizations help us define the following stabilized norms on $V_h$, which are instrumental in the forthcoming analysis.
        
        \begin{definition}[stabilized norms] \label{def:stabilized_norms}
            Define the following norms on $V_h$:
            \begin{equation}\label{eq:stabilized_norms}
                \|\cdot\|_{H^1_*}^2 := \|\cdot\|_{H^1_\G}^2 + s_{1}(\cdot, \cdot), \quad \|\cdot\|_{L^2_*}^2 := \|\cdot\|_{L^2_\G}^2 + s_{0}(\cdot, \cdot), \quad \|\cdot\|_{H^{-1}_*}^2 := \|\cdot\|_{H^{-1}_\G}^2 + s_{-1}(\cdot, \cdot).
            \end{equation}
            The local versions of these norms are denoted $\|\cdot\|_{H^1_*(T)}, \|\cdot\|_{L^2_*(T)}$, and $\|\cdot\|_{H^{-1}_*(T)}$, respectively.
        \end{definition}
        The choice of powers of $h$ in these definitions follows from the analysis of the stabilized $L^2$-projection $P_h$ of Section \ref{section:stabilized_projections}. In particular, we will see in the proof of Lemma \ref{lemma:interpolation} below that if $v\in H^2_\Gamma$ and $\{\cT_h\}_h$ is a shape-regular and quasi-uniform sequence of meshes, then the following error estimate is valid for the Scott-Zhang interpolant $I_h v^e$ to the normal extension $v^e$ of $v$
        $$
        \Vert v - I_h v^e\Vert_{L^2_\Gamma} + h \Vert \nablaG v - \nablaG I_h v^e\Vert_{L^2_\Gamma} \lesssim h^2|v|_{H^2_\Gamma},
        $$
        as well as an estimate on the normal derivatives
        $$
         h^{\alpha_j/2} \Vert \nabla I_h v^e \cdot \bn \Vert_{L^2(\Omega_h)} \lesssim h^{\alpha_j/2} h^{3/2}|v|_{H^2_\Gamma};
        $$
        this follows from the fact that $\nabla v^e \cdot \bn = 0$. For the latter estimate to balance the former, we require $(\alpha_j +3)/2 = 2 - j$ or $\alpha_j = 1-2j$ at least for $j=0,1$.
        
        It turns out that these powers yield stability and optimal error estimates for $P_h$ in the spaces $H^j_\G$ with $j=1,0,-1$. Note that the stabilization $s_{-1}(\cdot, \cdot)$ of the $H^{-1}_\G$-norm does not enter the numerical method \eqref{eq:intro:petrov_galerkin_semidiscrete}, but it is useful for its analysis. We shall also make use of the following discrete inner products and duality pairings:
        \begin{equation} \label{eq:discrete_inner_prod}
            (\cdot, \cdot)_* := (\cdot, \cdot)_\G + s_{0}(\cdot, \cdot), \quad \langle \cdot, \cdot \rangle_* := \langle \cdot, \cdot \rangle_\G + s_{0}(\cdot, \cdot), \quad a_*(\cdot, \cdot) := a(\cdot, \cdot) + s_{1}(\cdot, \cdot).
        \end{equation}
        
        Note that since the FE space lives in the bulk, while the solution $u$ to the continuous problem  \eqref{eq:model_problem} lives on the manifold $\G$, we need a special notation in order to define the distances between the two objects. In principle, for $u \in L^2(\G)$ we can insert the normal extension $u^e$, which is defined a.e. in $\Omega_h$, in the stabilization terms $s_1$ and $s_0$ without introducing any consistency error.
        The problem, however, is that we cannot assume a priori that $\pa_t u \in L^2_t H^{-1}_\G$ admits a normal extension. This motivates the definition \eqref{eq:intro:error_fnals} of {\it error functionals}.

%----------------------------------------------------------------------------
    \subsection{Inverse parameter}
%----------------------------------------------------------------------------

        In addition to the operator norms $\|P_h\|_{\cL(H^1_\G)}$ and $\|P_h\|_{\cL(H^1_*)}$ defined in \eqref{eq:intro:operator_norms}, we need to define the mesh-dependent constant $\Cinvh$ appearing in our main result \eqref{eq:intro:main_result}. To that end, we introduce
        \begin{definition}[discrete dual norm] \label{def:discrete_dual_norm}
            For $v_h \in V_h$, define
            \begin{equation}\label{eq:discrete_dual_norm}
                \|v_h\|_{V^{-1}_h} := \sup_{w_h \in V_h} \frac{(v_h, w_h)_*}{\|w_h\|_{H^1_*}}.
            \end{equation}
            where the inner product and the norm are defined in \eqref{eq:discrete_inner_prod} and \eqref{eq:stabilized_norms}.
        \end{definition}
        \begin{definition}[inverse parameter] \label{def:Cinvh}
        The inverse inequality constant between $V_h^{-1}$ and $L^2_\G$ reads
            \begin{equation} \label{eq:Cinvh}
                \Cinvh := \sup_{v_h \in V_h} \frac{\big( \sum_{T \in \cT_h} h_T^2 \|v_h\|_{L^2_*(T)}^2 \big)^{1/2}}{\|v_h\|_{V_h^{-1}}}.
            \end{equation}
        \end{definition}
        In principle, for a general sequence of meshes, $\Cinvh$ may blow up as $h\to0$. In Section \ref{section:stabilized_projections} we shall show that $\Cinvh$ is uniformly bounded in $h$ provided that $\{\cT_h\}_h$ is shape-regular and quasi-uniform. Since this quantity is a solution to a discrete optimization problem, it can be formulated as a generalized eigenvalue problem and estimated numerically on a fixed mesh $\cT_h.$

    %----------------------------------
    \subsection{Symmetry relation for \texorpdfstring{$P_h$}{Ph}}
    %----------------------------------
        Recall that $P_h$ is well-defined on $H^{-1}(\G)$ as per \eqref{eq:intro:l2proj}. The following {\it symmetry} relation is valid for all $\ell\in H^{-1}_\G$ and $v\in H^1_\G$
        \begin{equation}\label{eq:l2_proj_symmetry}
            \langle \ell, P_h v \rangle_\G = (P_h \ell, P_h v)_* = \langle P_h \ell, v \rangle_\G.
        \end{equation}
        We can endow $P_h: H^{-1}_\G \to V_h$ with the following norm:
        \begin{equation} \label{eq:Ph-op-norm-h-1}
            \|P_h\|_{\cL(H^{-1}_\G)} := \sup_{\ell\in H^{-1}_\G}
                 \frac{\|P_h\ell\|_{H^{-1}_\G}}{\|\ell\|_{H^{-1}_\G}}.
        \end{equation}
        The symmetry relation \eqref{eq:l2_proj_symmetry} and the definition of the operator norm \eqref{eq:Ph-op-norm-h-1} imply
        \begin{equation}\label{eq:Ph-H1-H-1}
            \begin{aligned}
                \|P_h\|_{\cL(H^{1}_\G)} & = \sup_{v\in H^1_\G} 
                \frac{\|P_h v\|_{H^1_\G}}{\|v\|_{H^1_\G}}
                = \sup_{v\in H^1_\G} \sup_{\ell\in H^{-1}_\G}
                \frac{\big| \langle P_h v, \ell\rangle \big|}{\|v\|_{H^1_\G} \|\ell\|_{H^{-1}_\G}}
                \\
                & = 
                \sup_{\ell\in H^{-1}_\G} \sup_{v\in H^1_\G} 
                \frac{\big| \langle v, P_h\ell\rangle \big|}{\|v\|_{H^1_\G} \|\ell\|_{H^{-1}_\G}}
                = \sup_{\ell\in H^{-1}_\G} \frac{\| P_h \ell\|_{H^{-1}_\G}}{\|\ell\|_{H^{1}_\G}}
                = \|P_h\|_{\cL(H^{-1}_\G)} .
            \end{aligned}
        \end{equation}
%----------------------------------------------------------------------------
    \subsection{Relation between the dual norms}
%----------------------------------------------------------------------------
        We have so far two definitions of dual norms in $V_h$, namely Definitions \ref{def:stabilized_norms} (stabilized norms) and \ref{def:discrete_dual_norm} (discrete dual norm). They are respectively
        \begin{equation*}
            \|v_h\|_{H^{-1}_*}^2 = \|v_h\|_{H^{-1}_\G}^2 + s_{-1}(v_h,v_h),
            \quad
            \|v_h\|_{V^{-1}_h} = \sup_{w_h\in V_h} \frac{(v_h,w_h)_*}{\|w_h\|_{H^1_*}},
            \quad\forall v_h\in V_h.
        \end{equation*}
        While $\|v_h\|_{V^{-1}_h}$ is a natural dual norm on the finite element space $V_h$, its dependence on $V_h$ is undesirable. On the other hand, the definition of $\|v_h\|_{H^{-1}_*}$ is independent of the choice of $V_h$, except for the stabilization term $s_{-1}$, which involves integration over $\Omega_h$, and is more appropriate for the parabolic inf-sup property of Section \ref{section:inf-sup}. A natural question is whether these two norms are equivalent. This subsection explores this question.

        We start with a simple observation about duality for the stabilization forms \eqref{eq:stabilization} that hinges on a symmetric choice of scaling of the element diameter $h_T$ for $T\in\cT_h$:
        \[
        s_0(v_h,v_h) \leq s_1(v_h, v_h)^{1/2} s_{-1}(v_h, v_h)^{1/2} 
        \quad\forall v_h\in V_h
        \]
        and, correspondingly, for the stabilized norms \eqref{eq:stabilized_norms}
        \begin{equation}\label{eq:norm-duality}
         \big| (v_h,w_h)_* \big| \le \|v_h\|_{H^1_*} \|w_h\|_{H^{-1}_*}
         \quad\forall v_h,w_h\in V_h.
        \end{equation}
        In fact, this calculation is what justifies the definition of $\|\cdot\|_{H^{-1}_*}$ and $s_{-1}$.
        
        \begin{lemma}[relation between the dual norms] \label{lemma:relationship_bw_the_dual_norms} 
            Let $\cT_h$ satisfy Assumption \ref{assump:res-of-geom} (resolution of the geometry). If $\|P_h\|_{\mathcal{L}(H^1_*)}$ is the operator norm \eqref{eq:intro:operator_norms} of the stabilized $L^2$-projection \eqref{eq:intro:l2proj} and $\Cinvh$ is given by Definition \ref{def:Cinvh} (inverse parameter), then we have
            \begin{equation}\label{eq:neg-norm-equiv}
                \|v_h\|_{V^{-1}_h} \leq \|v_h\|_{H^{-1}_*} \leq \big(\|P_h\|_{\mathcal{L}(H^1_*)} + \Cinvh \big) \|v_h\|_{V^{-1}_h} \quad \forall v_h \in V_h\,.
            \end{equation}
        \end{lemma}
        \begin{proof}
            Let $v_h \in V_h$. For the leftmost inequality we simply use \eqref{eq:norm-duality}
            \begin{equation*}
                \|v_h\|_{V^{-1}_h} = \sup_{w_h \in V_h} \frac{(v_h, w_h)_*}{\|w_h\|_{H^1_*}} \leq \|v_h\|_{H^{-1}_*}.
            \end{equation*}
            For the rightmost inequality, we use the definitions of $\|\cdot\|_{H^{-1}_\G}$ and \eqref{eq:intro:l2proj} of $P_h$ to write
            \[
                \|v_h\|_{H^{-1}_\G} = \sup_{w \in H^1_\G} \frac{(v_h, w)_\G}{\|w\|_{H^1_\G}} = \sup_{w\in H^1_\G} \frac{(v_h, P_h w)_*}{\|w\|_{H^1_\G}} 
            \]
            We now crucially employ the definition of $\|P_h\|_{\cL(H^1_*)}$ given in \eqref{eq:intro:operator_norms}, together with Definition \ref{def:discrete_dual_norm} (discrete dual norm) to arrive at
            \[
                \|v_h\|_{H^{-1}_\G}
                \leq \|P_h\|_{\mathcal{L}(H^1_*)} \sup_{w \in H^1_\G} \frac{(v_h, P_h w)_*}{\|P_h w\|_{H^1_*}} \leq \|P_h\|_{\mathcal{L}(H^1_*)} \sup_{w_h \in V_h} \frac{(v_h, w_h)_*}{\|w_h\|_{H^1_*}} = \|P_h\|_{\mathcal{L}(H^1_*)} \|v_h\|_{V^{-1}_h}.
            \]
            To bound the stabilization term $s_{-1}(v_h,v_h)$, we use Definition \ref{def:Cinvh} to get
            \[
            s_{-1}(v_h,v_h) = \sum_{T\in\cT_h} s_{-1, T}(v_h, v_h) = \sum_{T\in\cT_h} h_T^2 s_{0, T}(v_h, v_h) \le \sum_{T\in\cT_h} h_T^2 \|v_h\|_{L^2_*(T)}^2  
            \leq C_{\textrm{inv},h}^2 \|v_h\|_{V^{-1}_h}^2.
            \]
             Adding gives $\|v_h\|_{H^{-1}_*}^2 \leq \big(\|P_h\|_{\mathcal{L}(H^1_*)}^2 + C_{\textrm{inv},h}^2\big) \|v_h\|_{V^{-1}_h}^2$, hence the desired estimate.
        \end{proof}

        We point out that \eqref{eq:neg-norm-equiv} does not assert equivalence of $\|\cdot\|_{V_h^{-1}}$ and $\|\cdot\|_{H^{-1}_*}$ unless $\Cinvh$ and $\|P_h\|_{\cL(H^1_*)}$ are uniformly bounded. We will prove this to be true for a sequence $\{\cT_h\}_h$ of shape-regular and quasi-uniform meshes in Section~\ref{section:stabilized_projections}.

%%%%%%%%%%%%%%%%%%%%%%%%%%%%%%%%%%%%%%%%%%%%%%%%%%%%%%%%%%%%%%%%%%%%%%%%%%%%%%
\section{Inf-sup stability analysis} \sectionmark{Inf-sup stability} \label{section:inf-sup}
%%%%%%%%%%%%%%%%%%%%%%%%%%%%%%%%%%%%%%%%%%%%%%%%%%%%%%%%%%%%%%%%%%%%%%%%%%%%%%

    We formulate surface parabolic problems, both continuous and semidiscrete, within the framework of Petrov-Galerkin methods with different trial and test spaces. The Banach-Ne\v cas Theorem provides sufficient and necessary conditions for existence, uniqueness, and continuous
    dependence on data for continuous problems. The celebrated Babu\v ska Theorem extends this
    theory to discrete problems. For parabolic problems on flat domains we refer to \cite{Hackbusch_1981,Dupont_1982,MakridakisBabuska_1997,SchwabStevenson_2009,Andreev_2013} and to the book \cite{ErnGuermond_2004}. Our main reference, however, is Tantardini and Veeser \cite{tantardini_l2-projection_2016} for fitted FEM with zero Dirichlet boundary condition. In fact, we develop a similar inf-sup theory to \cite{tantardini_l2-projection_2016} for the special class TraceFEM of unfitted FEMs for parabolic problems on surfaces. We want to emphasize that throughout the section, the sequence $\{\cT_h\}_h$ is not assumed to be quasi-uniform or even shape-regular. We only require that the stabilized TraceFEM is well-defined as per Assumption \ref{assump:res-of-geom} (resolution of the geometry), which, in a sense, is the minimal mesh assumption.
    
    %-----------------------------------------------------------------------------
    \subsection{Continuous setting}
    %-----------------------------------------------------------------------------
    
        Let $I:=(0,T)$. Define the trial, test, and data spaces as follows:
        \begin{equation}\label{eq:coninuous-norms}
        \begin{array}{lll}
            \mathbb{U} := L^2_t H^1_\G \cap H^1_t H^{-1}_{\G} & \quad \|u\|_1^2 := \|u(0)\|_{L^2_\G}^2 + \int_I \|\pa_t u\|_{H^{-1}_\G}^2 + |u|_{H^1_\G}^2 & \quad u \in \mathbb{U} , \\
            \mathbb{V} := L^2_{\G} \times L^2_t H^1_\G & \quad \|v\|_2^2 := \|v_0\|_{L^2_\G}^2 + \int_I \|v_1\|_{H^1_\G}^2 & \quad v = (v_0,v_1) \in \mathbb{V} , \\
            \mathbb{D} := L^2_{\G} \times L^2_t H^{-1}_\G & \quad \|(u_0, f)\|_3^2 := \|u_0\|_{L^2_\G}^2 + \int_I \|f\|_{H^{-1}_\G}^2 & \quad (u_0,f) \in \mathbb{D}.
            \end{array}
        \end{equation}
        Since $\mathbb{U} \subset C^0(I;L^2_\G)$ by Lions-Magenes \cite{ErnGuermond_2004}, 
        the inclusion of $\|u(0)\|_\G$ in the definition of the norm $\|\cdot\|_1$ is justified.
        On the other hand, the usage of the seminorm $|\cdot|_{H^1_\G}$ as opposed to the full counterpart $\|\cdot\|_{H^1_\G}$ is motivated by the bound on the inf-sup constant. We do not
        assume that functions in the spaces $\mathbb{U},\mathbb{V},\mathbb{D}$ have vanishing meanvalue on $\G$.
        
        The weak formulation for the surface heat equation reads: given $(u_0, f) \in \mathbb{D}$, find $u \in \mathbb{U}$ such that \looseness=-1
        \begin{equation} \label{eq:cont_problem}
            b(u, v) = \ell(v) \quad \forall v \in \mathbb{V},
        \end{equation}
        where for all $v = (v_0, v_1) \in \mathbb{V}$
        \[
           b(u, v) := (u(0), v_0)_\G + \int_I \langle \pa_t u, v_1 \rangle_\G + a(u, v_1)
        \]
        and
        \[
           \ell(v) := (u_0, v_0)_\G + \int_I \langle f, v_1 \rangle_\G.
        \]
        We choose the bilinear form $a$ to be the Dirichlet form $a(\cdot, \cdot) := (\nablaG \cdot, \nablaG \cdot)_\G$ over $\G$ for simplicity of presentation.
        This form is continuous on $H^1_\G$ but is not coercive because $a(1,1)=0$. Moreover,
        \begin{equation}\label{eq:norm3}
            \|(u_0,f)\|_3 = \sup_{v\in\mathbb{V}} \frac{\ell(v)}{\|v\|_2}.
        \end{equation}
        
        \begin{proposition}[well-posedness of the continuous problem] \label{proposition:inf-sup_stability}
            The bilinear form $b$ is continuous with constant $C_b$ and satisfies the inf-sup condition with constant $c_b$, namely
            \begin{equation}\label{eq:inf-sup_constants}
                C_b := \sup_{w\in\mathbb{U}} \sup_{v\in\mathbb{V}} \frac{b(w,v)}{\|w\|_1 \|v\|_2},
                \quad
                c_b := \inf_{w\in\mathbb{U}} \sup_{v\in\mathbb{V}} \frac{b(w,v)}{\|w\|_1 \|v\|_2},
            \end{equation}
            with 
            \begin{equation} \label{eq:Cb+cb-}
                C_b \leq \Cbp := \sqrt{2}, \qquad c_b \geq \cbm := \frac{1}{\sqrt8 (1 + 2 T)},
           \end{equation}
           where the asymptotic dependence on the final time $T = \sup_{t \in I} t$ is optimal. Moreover, if functions in the pair of spaces $(\mathbb{U},\mathbb{V})$ have vanishing mean value on $\G$ for a.e. $t\in I$, then $c_b\ge 8^{-1/2}$. Lastly, 
           the kernel of the adjoint problem is trivial, namely, if $v\in V$ satisfies
           \begin{equation}\label{eq:continuous-nondegeneracy}
                b(u, v) = 0 \quad \forall u \in \mU \quad\implies\quad v = 0.
           \end{equation}
            Therefore, there exists a unique solution $u\in\mathbb{U}$ of \eqref{eq:cont_problem},
            which depends continuously on data.
        \end{proposition}
        \begin{proof}
            It follows along the lines of Theorem \ref{thm:inf-sup-stability} ($V_h$-dependent inf-sup stability) and Lemma \ref{lem:adjoint} (adjoint) below and is thus omitted.
        \end{proof}
    
    %--------------------------------------------
    \subsection{Semidiscretization with TraceFEM}
    %--------------------------------------------
        Recall that $\Omega_h$ is the set of all elements intersected by $\G$, and $V_h$ is the finite element space over $\Omega_h$ of continuous piecewise polynomial functions so that $V_h|_\G \subset H^1_\G$; the polynomial degree is not critical for our analysis. Let the space-time semidiscrete spaces be
        \begin{equation}\label{eq:semidiscrete-spaces}
            \mathbb{U}_h := H^1_t V_h, \quad \mathbb{V}_h := V_h \times L^2_t V_h.
        \end{equation}
        
        We define the stabilized bilinear form $b_h: \mathbb{U}_h \times \mathbb{V}_h \to \mathbb{R}$ 
        and linear form $\ell_h:\mV_h\to\R$ to be
        \begin{equation}\label{eq:stab-blinear-form}
            \begin{split}
                b_h(u_h, v_h) & := (u_h(0), v_{h,0})_* + \int_I \langle \pa_t u_h, v_{h,1} \rangle_* + a_*(u_h, v_{h,1}) \quad\forall u_h\in \mathbb{U}_h, v_h \in \mathbb{V}_h,
            \end{split}
        \end{equation}
        and
        \begin{equation}\label{eq:linear-form}
            \ell(v_h) = (u_0, v_{h,0})_\G + \int_I \langle f, v_{h,1} \rangle_\G = (P_h u_0, v_{h,0})_* + \int_I (P_h f, v_{h,1})_* \quad\forall v_h \in \mathbb{V}_h,
        \end{equation}
        where the subscript $_*$ stands for stabilization and has been introduced in \eqref{eq:discrete_inner_prod}. Note that the stabilization of the time derivative $s_{0}(\pa_t u_h, v_{h,0})$ is \textit{weakly} consistent because $\pa_t u \in L^2(I; H^{-1}_\G)$ cannot be inserted into it due to the lack of spatial regularity. 
        We stress once again that adding $s_0(\pa_t u_h, v_{h,1})$ to the middle term of 
        \eqref{eq:stab-blinear-form} is a crucial modification to the standard TraceFEM \cite{hansbo_characteristic_2015, lehrenfeld_stabilized_2018}. We explore in Section \ref{sec:robust-inf-sup} below its relation with the stabilized $L^2$-projection $P_h$.
        On the other hand, the linear form is implicitly stabilized thanks to the definition of $P_h$.

        The semidiscrete problem reads as follows: given $(u_0, f) \in \mathbb{D}$, find $u_h \in \mathbb{U}_h$ such that
        \begin{equation}\label{eq:discrete_problem}
            b_h(u_h, v_h) = \ell(v_h), \quad\forall v_h \in \mathbb{V}_h.
        \end{equation}
        In order to prove semidiscrete inf-sup conditions for $b_h$, we endow the test space $\mathbb{V}_h$ with the stabilized norm
        \begin{equation}\label{eq:2*}
            \|v_h\|_{2,*}^2 := \|v_{h,0}\|_{L^2_*}^2 + \int_I \| v_{h,1} \|_{H^1_*}^2.
        \end{equation}
        The situation with the trial space is not obvious, because we can either measure $\pa_t u_h$ in $\|\cdot\|_{V_h^{-1}}$ defined in \eqref{eq:discrete_dual_norm} or $\|\cdot\|_{H^{-1}_*}$ defined in \eqref{eq:stabilized_norms}. Since $\|\cdot\|_{V_h^{-1}}$ is the natural norm coming from a discrete stability analysis, we shall focus on the following parabolic norm in $\mathbb{U}_h$ in this section:
        \begin{equation}\label{eq:1*-Vh}
            \| u_h \|_{1,*;V_h}^2 := \| u_h(0) \|_{L^2_*}^2 + \int_I \| \partial_t u_h \|_{V_h^{-1}}^2
            + |u_h|_{H^1_*}^2.
        \end{equation}
        We refer to the norm $\|\cdot\|_{1,*;V_h}$, and the stability of $b_h$ associated with it, as {\it $V_h$-dependent} because the definition \eqref{eq:discrete_dual_norm} of $\|\cdot\|_{V_h^{-1}}$ takes the supremum over discrete functions in $V_h$ exclusively. 
        
        We now derive the $V_h$-dependent inf-sup stability analysis of $b_h$. To this end, we let the inf-sup and continuity constants of $b_h$ relative to the norms $\|\cdot\|_{1,*;V_h}$ and $\|\cdot\|_{2,*}$ be defined as follows:
        \begin{equation}\label{eq:inf-sup-continuity}
               \cbh := \inf_{u_h\in\mathbb{U}_h} \sup_{v_h\in\mathbb{V}_h}
               \frac{b_h(u_h,v_h)}{\|u_h\|_{1,*;V_h} \|v_h\|_{2,*}}, \qquad \Cbh := \sup_{u_h\in\mathbb{U}_h} \sup_{v_h\in\mathbb{V}_h}
               \frac{b_h(u_h,v_h)}{\|u_h\|_{1,*;V_h} \|v_h\|_{2,*}},
        \end{equation}
        At first glance, $\cbh$ and $\Cbh$ might be sensitive to the limit $h \to 0$. However, we will see in a moment that the same $h$-independent bounds $\cbh\ge\cbm$ and $\Cbh \le \Cbp$ of the continuous problem are valid.

        The remainder of this section is dedicated to proving $V_h$-dependent inf-sup stability of \eqref{eq:discrete_problem}, namely w.r.t. to the norms $\|\cdot\|_{1,*;V_h}$ in $\mathbb{U}_h$ and $\|\cdot\|_{2,*}$ in $\mathbb{V}_h$. The proof, with minor modifications, also gives the inf-sup stability estimate of Proposition \ref{proposition:inf-sup_stability} (well-posedness of the continuous problem) with the same constants $c_b^-$ and $C_b^+$.
        
        The question of inf-sup stability of $b_h$ w.r.t. to a more robust norm on $\mU_h$ that measures $\pa_t u_h$ in $H^{-1}_*$ instead of $V_h^{-1}$ will be explored in the following Section \ref{sec:robust-inf-sup} (Robust discrete inf-sup stability).
        \begin{theorem}[$V_h$-dependent inf-sup stability]\label{thm:inf-sup-stability}
            Let Assumption \ref{assump:res-of-geom} (resolution of the geometry) be valid. Then,
            the bilinear form $b_h$ in \eqref{eq:discrete_problem} is continuous and inf-sup stable in the discrete spaces $(\mathbb{U}_h,\|\cdot\|_{1,*;V_h})$ and $(\mathbb{V}_h,\|\cdot\|_{2,*})$ with constants defined in \eqref{eq:inf-sup-continuity} given by
            \begin{equation}\label{eq:cont-inf-sup-stability}
                \Cbh \le \Cbp = \sqrt2, \qquad \cbh \ge \cbm = \frac{1}{\sqrt8 (1+2T)},
            \end{equation}
            where $\Cbp, \cbm$ are defined in Proposition \ref{proposition:inf-sup_stability} (well-posedness of the continuous problem), and the asymptotic dependence on the final time $T = \sup_{t \in I} t$ is optimal. Moreover, if functions in the pair of spaces $(\mathbb{U}_h,\mathbb{V}_h)$ have vanishing meanvalue on $\G$ for a.e. $t\in I$, then $\cbm = 8^{-1/2}$.
        \end{theorem}

        We emphasize that no assumption on the underlying mesh $\T_h$ beyond Assumption \ref{assump:res-of-geom} is needed for Theorem \ref{thm:inf-sup-stability} to hold.
        We split the proof of Theorem \ref{thm:inf-sup-stability} into three parts and discuss them over the next three following subsubsections.

        %-----------------------------------------------------------------------------
        \subsubsection{Modified Laplace-Beltrami operator}
        %-----------------------------------------------------------------------------

            The following discrete operator $A_h: V_h \to V_h$ is a substitute for the discrete Laplacian in flat domains. However, since we do not assume data with vanishing meanvalues, we need a zero order term for the problem to be well posed in $V_h$. Given $w_h\in V^{-1}_h$, the auxiliary problem reads as follows: find $\phi_h \in V_h$ such that
            \begin{equation} \label{eq:discr_aux_pb}
                ( A_h\phi_h,v_h )_{\Omega_h} :=
                a_*(\phi_h, v_h) + (\phi_h, v_h)_* = \langle w_h, v_h \rangle_* \quad \forall v_h \in V_h.
            \end{equation}
            We stress that the bilinear form in \eqref{eq:discr_aux_pb} induces the norm $\|\cdot\|_{H^1_*}$ in $V_h.$

            \begin{lemma}[stability of auxiliary problem] \label{eq:discr_aux_pb_stability}
                The solution $\phi_h\in V_h$ of \eqref{eq:discr_aux_pb} satisfies
                \[
                  \|\phi_h\|_{H^1_*} = \|w_h\|_{V^{-1}_h} \quad\forall w_h\in V_h.
                \]
            \end{lemma}
            \begin{proof}
                 Take $v_h=\phi_h\in V_h$ in \eqref{eq:discr_aux_pb} to obtain
                 \[
                    \|\phi_h\|_{H^1_*}^2 = \langle w_h,\phi_h \rangle_* 
                    \le \|w_h\|_{V^{-1}_h} \|\phi_h\|_{H^1_*} \
                    \quad\Rightarrow\quad
                    \|\phi_h\|_{H^1_*} \le \|w_h\|_{V^{-1}_h}.
                 \]
                 On the other hand, compute
                 \[
                    \langle w_h, v_h \rangle_* = a_*(\phi_h,v_h) + (\phi_h, v_h )_* 
                    \le \|\phi_h\|_{H^1_*} \|v_h\|_{H^1_*}.
                    \quad\Rightarrow\quad
                    \|w_h\|_{V^{-1}_h}. = \sup_{v_h\in V_h} \frac{\langle w_h,\phi_h \rangle_*}{\|v_h\|_{H^1_*}} \le \|\phi_h\|_{H^1_*}.
                 \]
                 This concludes the proof.
            \end{proof}

        %-----------------------------------------------------------------------------
        \subsubsection{Control of the \texorpdfstring{$L^\infty_t L^2_*$}{L2t L2*} norm}
        %-----------------------------------------------------------------------------

            The following intermediate bound will be useful in the third step.

            \begin{lemma}[embedding]\label{lem:L2-norm}
                The solution $u_h(\cdot,t)\in V_h$ of the semidiscrete problem \eqref{eq:discrete_problem} satisfies the following estimate for a.e. $t\in (0,T)$
                \begin{equation}\label{eq:L2-norm}
                \|u_h(t)\|_{L^2_*}^2 \leq 2\|u_h(0)\|_{L^2_*}^2 + 2 \int_I |u_h(t)|_{H^1_*}^2 + 
                2(1+2T) \int_I \|\pa_t u_h\|_{V^{-1}_h}^2.
                \end{equation}
            \end{lemma}
            \begin{proof}
               Compute
               \[
                  \|u_h(t)\|_{L^2_*}^2 - \|u_h(0)\|_{L^2_*}^2 = 
                  \int_0^t \frac{d}{dt} \|u_h\|^2_{L^2_*} = 2 \int_0^t \langle \partial_t u_h, u_h \rangle_*
                  \le 2 \int_0^t \|\partial_t u_h\|_{V^{-1}_h} \|u_h\|_{H^1_*},
               \]
               whence
               \begin{align*}
                  \|u_h(t)\|_{L^2_*}^2 &\le \|u_h(0)\|_{L^2_*}^2 + 2 \int_0^t \|\partial_t u_h\|_{V^{-1}_h} \|u_h\|_{H^1_*}
                  \\
                  & \le \|u_h(0)\|_{L^2_*}^2 + 2 \int_0^t \|\partial_t u_h\|_{V^{-1}_h} \|u_h\|_{L^2_*}
                  + \|\partial_t u_h\|_{V^{-1}_h} |u_h|_{H^1_*}.
               \end{align*}
               We note the appearance of $\|u_h\|_{L^2_*}$ on the right-hand side. We thus proceed as
               follows to avoid a Gr\"onwall argument. If
               $
                 U := \max_{t\in I} \|u_h\|_{L^2_*},
               $
               then
               \[
                 U^2 \le 2 U \Big( \int_I \|\partial_t u_h\|_{V^{-1}_h} \Big)^2 
                 + \|u_h(0)\|_{L^2_*}^2  + 2 \int_I \|\partial_t u_h\|_{V^{-1}_h} |u_h|_{H^1_*},
               \]
               whence
               \begin{align*}
                   \|u_h(t)\|_{L^2_*}^2 &\le U^2 \le 4 \Big( \int_I \|\partial_t u_h\|_{V^{-1}_h} \Big)^2 
                 + 2 \|u_h(0)\|_{L^2_*}^2  + 4 \int_I \|\partial_t u_h\|_{V^{-1}_h} |u_h|_{H^1_*}
                 \\ &
                 \le 2 \|u_h(0)\|_{L^2_*}^2 + 2 \int_I |u_h|_{H^1_*}^2 + 
                 \big( 4T+2 \big) \int_I \|\partial_t u_h\|_{V^{-1}_h}^2,
               \end{align*}
               as asserted.
            \end{proof}
        
        %-----------------------------------------------------------------------------        
        \subsubsection{Inf-sup stability: Proof of Theorem \ref{thm:inf-sup-stability}}
        %-----------------------------------------------------------------------------    
            
            Continuity of $b_h$ with constant $\Cbh \le \sqrt{2}$ is straightforward because $a_*(\cdot, \cdot)$ is continuous on $V_h$ with constant $1$. Hence
            \[
              b_h(u_h,v_h) \le \|u_h(0)\|_{L^2_*} \|v_{h,0}\|_{L^2_*}
              + \int_I \|\partial_t u_h\|_{V^{-1}_h} \|v_{h,1}\|_{H^1_*}
              + |u_h|_{H^1_*} |v_h|_{H^1_*} \le \sqrt{2} \|u_h\|_{1,*;V_h} \|v_h\|_{2,*}
            \]

            For the inf-sup condition of $b_h$, we let $v_{0,h} = 2 u_h(0)$ and $v_{1,h} = 2 u_h + 2 \phi_h$.
            If $\phi_h\in V_h$ denotes the solution of \eqref{eq:discr_aux_pb} for $w_h=\partial_t u_h$,
            we then have
            \[
              b_h(u_h, v_h) = 2 \|u_h(0)\|_{L^2_*}^2 + \int_I 2 \langle \pa_t u_h, u_h \rangle_* + 2 \langle \pa_t u_h, \phi_h \rangle_* + 2 a_*(u_h, u_h) + 2 a_*(u_h, \phi_h)
            \]
            By definition \eqref{eq:discr_aux_pb} we see that
            \[
            \langle  \pa_t u_h, \phi_h \rangle_* = \|\phi_h\|_{H^1_*}^2.
            \]
            If we further compute
            \[
              \int_I \langle  \pa_t u_h, u_h \rangle_* = \frac12 \|u_h(t)\|_{L^2_*}^2
              - \frac12 \|u_h(0)\|_{L^2_*}^2
            \]
            and
            \[
              \big| a_*(u_h,\phi_h)  \big| \le \frac12 a_*(u_h, u_h) +
              \frac12 a_*(\phi_h, \phi_h),
            \]
            we end up with
            \[
             \begin{split}
                 b_h(u_h,v_h) & \geq \|u_h(0)\|_{L^2_*}^2 + \int_I 2 \|\phi_h\|_{H^1_{*}}^2 + 2 a_*(u_h, u_h) - a_*(u_h, u_h) - a_*(\phi_h, \phi_h) \\
                 & \geq \|u_h(0)\|_{L^2_*}^2 + \int_I \|\phi_h\|_{H^1_{*}}^2 + a_*(u_h, u_h) \\
                 & = \|u_h(0)\|_{L^2_*}^2 + \int_I \|\pa_t u_h\|_{V_h^{-1}}^2 + a_*(u_h, u_h) = \|u_h\|_{1,*;V_h}^2
             \end{split}
            \]
            where in the last step we used Lemma \ref{eq:discr_aux_pb_stability} (stability of the auxiliary problem). Hence
            \begin{equation}\label{eq:lower-bound-bh}
                b_h(u_h, v_h) \geq \|u_h\|_{1,*;V_h}^2, 
            \end{equation}
            with $v_h = (2 u_h(0), 2 u_h + 2 \phi_h)$.
            Our next task is to estimate $\|v_h\|_{2,*}$, which reads
            \[
              \|v_h\|_{2,*}^2 \le 
              4 \|u_h(0)\|_{L^2_*}^2 + 8 \int_I \|\pa_t u_h\|_{V^{-1}_h}^2 + \|u_h\|_{H^1_*}^2 = 4 \|u_h(0)\|_{L^2_*}^2 + 8 \int_I \|\pa_t u_h\|_{V^{-1}_h}^2 + |u_h|_{H^1_*}^2 + \|u_h\|_{L^2_*}^2
            \]
            again according to Lemma \ref{eq:discr_aux_pb_stability}. The last term is not present in the definition of the norm $\|u_h\|_{1,*;V_h}$, but we can control it as follows using Lemma \ref{lem:L2-norm} (embedding):
                \begin{equation*}
                    \|v_h\|_{2,*}^2 \le (4 + 16T) \|u_h(0)\|_{L^2_*}^2 + (8 + 16T) \int_I |u_h|_{H^1_*}^2 + (8 + 16T + 32T^2) \int_I \|\pa_t u_h\|_{V^{-1}_h}^2,
                \end{equation*}
            whence
            \begin{equation}\label{eq:norm-vh}
                \|v_h\|_{2,*} \le \sqrt{8}(1 + 2 T) \|u_h\|_{1,*;V_h}.
            \end{equation}
            Finally, in view of \eqref{eq:lower-bound-bh} and \eqref{eq:norm-vh}, we are able to estimate the inf-sup constant $c_b$:
            \begin{equation}\label{eq:constant-C*}
            \cbh = \sup_{v_h \in \mV_h} \frac{b_h(u_h, v_h)}{\|u_h\|_{1,*;V_h} \|v_h\|_{2,h}} \geq \frac{\|u_h\|_{1,*;V_h}^2}{\sqrt8 (1 + 2 T) \|u_h\|_{1,*;V_h}^2} = \frac{1}{\sqrt8 (1 + 2 T)} = \cbm.
            \end{equation}

            We now discuss the asymptotic behavior of $\cbh$ in \eqref{eq:constant-C*} with respect to $T$. First of all, we recall that $(\cbh)^{-1}$ can be regarded as a stability constant of the semidiscrete parabolic problem \eqref{eq:linear-form}:
            \[
              \frac{\|u_h\|_{1,*;V_h}}{\|(u_0,f)\|_{3,*}}
              = \frac{\|u_h\|_{1,*;V_h}}{\sup_{v_h\in V_h} \frac{\ell (v_h)}{\|v_h\|_{2,*}}}\le \frac{1}{\cbh}.
            \]
            We claim that, for general data $\ell=(u_0,f)\in\mathbb{D}$, the linear growth in $T$ of $(\cbh)^{-1}$ is optimal because
            \[
              \int_\G u_h(t) = \int_\G P_h u_0 + \int_0^t \langle f,1 \rangle_\G \ne 0.
            \]
            To see this consider the initial value problem $u'=1, u(0)=1$ in $\R$,
            whose solution is $u(t)=1+t$.

            On the other hand, if both $u_0$ and $f(\cdot,t)$ have vanishing mean value for a.e. $t\in I$, then so does $u_h(\cdot,t)$. Therefore, the seminorm $|v_h|_{H^1_*}$ induced by the bilinear form $a_*$ becomes a norm in $V_h$. Moreover, if we adopt $|v_h|_{H^1_*}$ instead of $\|v_h\|_{H^1_*}$ in the definition \eqref{eq:1*-Vh} of $\|v_h\|_{2,*}$, then the same calculation leading to \eqref{eq:norm-vh} gives
            \[
              \| v_h \|_{2,*} \le \sqrt{8} \, \|u_h\|_{1,*:V_h} 
              \quad\Rightarrow\quad
              \cbh \le \frac{1}{\sqrt{8}},
            \]
            In contrast to \eqref{eq:constant-C*}, the estimate of $c_b$ is now independent of $T$. The proof of Theorem \ref{thm:inf-sup-stability} is thus finished.

            The choice of $|v_h|_{H^1_*}$ in the definition \eqref{eq:1*-Vh} of $\|v_h\|_{1,*;V_h}$ is dictated by the bilinear form $a_*$. In principle, we could have used the full norm $\|v_h\|_{H^1_*}$ within $\|v_h\|_{1,*;V_h}$, and repeated the above argument. This, however, would have led to a suboptimal inf-sup constant $\cbh$ that scales as $T^{-2}$.

    %-------------------------------------------------------------------------------
    \subsection{Robust discrete inf-sup stability} \label{sec:robust-inf-sup}
    %-------------------------------------------------------------------------------
    In view of Lemma \ref{lemma:relationship_bw_the_dual_norms} (relation between dual norms), we now replace $\|\pa_t u_h\|_{V_h^{-1}}$, defined in \eqref{eq:discrete_dual_norm}, by the stronger norm $\|\pa_t u_h\|_{H^{-1}_*}$, defined in \eqref{eq:stabilized_norms}, to obtain the following norm on the trial space $\mU_h$:
        \begin{equation}\label{eq:1*}
            \| u_h \|_{1,*}^2 := \| u_h(0)\|_{L^2_*}^2 + \int_I \| \partial_t u_h \|_{H^{-1}_*}^2
            + | u_h |_{H^1_*}^2.
        \end{equation}
        This is thus a discrete version of the continuous norm $\|u_h\|_1$ defined in \eqref{eq:coninuous-norms}, and is stronger than $\|u_h\|_1$ in the sense that $\|u_h\|_1 \le \|u_h\|_{1,*}$ because of the stabilization terms. We refer to this norm {\it robust} because uniform inf-sup stability w.r.t. to this norm implies several key results known for fitted methods, hence it is robust w.r.t. unfittedness. In fact, this norm is sufficiently strong as to enable us to use tools from functional analysis in Section \ref{sec:convergence} and prove quasi-best approximation and convergence to minimal regularity solutions. This is not possible for the weaker $V_h$-dependent norm $\|u_h\|_{1,*;V_h}$ and corresponding inf-sup theory of Theorem \ref{thm:inf-sup-stability} ($V_h$-dependent inf-sup stability).

        Let the robust continuity and inf-sup constants corresponding to the robust $\|\cdot\|_{1,*}$-norm \eqref{eq:1*} on the trial space $\mathbb{U}_h$ be denoted by
        \begin{equation}\label{eq:inf-sup-*}
            c_* := \inf_{u_h\in\mathbb{U}_h} \sup_{v_h\in\mathbb{V}_h}
               \frac{b_h(u_h,v_h)}{\|u_h\|_{1,*} \|v_h\|_{2,*}}, \qquad C_* := \sup_{u_h\in\mathbb{U}_h} \sup_{v_h\in\mathbb{V}_h}
               \frac{b_h(u_h,v_h)}{\|u_h\|_{1,*} \|v_h\|_{2,*}}.
        \end{equation}
        First, we state an immediate observation concerning $C_*$.
        \begin{lemma}[robust continuity constant] \label{lem:C*}
            Let $\cT_h$ satisfy Assumption \ref{assump:res-of-geom}. The continuity constant $C_*$ defined in \eqref{eq:inf-sup-*} is bounded by $C_b^+ = \sqrt2$ from Proposition \ref{proposition:inf-sup_stability} (well-posedness of the continuous problem).
        \end{lemma}
        \begin{proof}
            Follows from Cauchy-Schwarz and duality \eqref{eq:norm-duality}: $|(\pa_t u_h, v_{h,1})_*| \leq \|\pa_t u_h\|_{H^{-1}_*} \|v_{h,1}\|_{H^1_*}$.
        \end{proof}
        The next step is to establish a lower and an upper bound on the robust inf-sup constant $c_*$ in terms of mesh-independent quantities $\cbm, \Cbp$ from Proposition \ref{proposition:inf-sup_stability}, and mesh-dependent quantities $\|P_h\|_{\cL(H^1_\G)}, \|P_h\|_{\cL(H^1_*)}$ defined in \eqref{eq:intro:operator_norms} and $\Cinvh$ defined in \eqref{eq:Cinvh}. This is encapsulated in Theorem \ref{thm:Ph-c*}, which we shall now prove in several steps.
        \begin{lemma}[relation between $\cbm, \Cbp$ and $c_*$] \label{lem:cb-c*}
            If
            \begin{equation} \label{eq:def-Lambda}
                \Lambda_h := \inf_{u_h \in \mU_h} \frac{\|u_h\|_{1,*;V_h}}{\|u_h\|_{1,*}},
            \end{equation}
            then we have
            \begin{equation}\label{eq:cd-c*}
                \Lambda_h \cbm \le c_* \le \Lambda_h \Cbp .
            \end{equation}
        \end{lemma}
        \begin{proof}
            We proceed in two steps. We first write
            \begin{align*}
              \sup_{v_h\in\mathbb{V}_h} \frac{b_h(u_h,v_h)}{\|u_h\|_{1,*} \|v_h\|_{2,*}}
              = \frac{\|u_h\|_{1,*;V_h} }{\|u_h\|_{1,*} }
              \sup_{v_h \in \mV_h } \frac{b_h(u_h,v_h)}{\|u_h\|_{1,*;V_h} \|v_h\|_{2,*}}
              \ge \Lambda_h \sup_{v_h \in \mV_h } \frac{b_h(u_h,v_h)}{\|u_h\|_{1,*;V_h} \|v_h\|_{2,*}},
            \end{align*}
            whence
            $
             c_* \ge \Lambda_h \cbh \geq \Lambda_h \cbm
            $ by Theorem \ref{thm:inf-sup-stability} ($V_h$-dependent inf-sup stability). 
            We next write again the above equality
            \begin{equation*}
                \sup_{v_h\in\mathbb{V}_h} \frac{b_h(u_h,v_h)}{\|u_h\|_{1,*} \|v_h\|_{2,*}}
              = \frac{\|u_h\|_{1,*;V_h} }{\|u_h\|_{1,*} }
              \sup_{v_h \in \mV_h } \frac{b_h(u_h,v_h)}{\|u_h\|_{1,*;V_h} \|v_h\|_{2,*}}
              \le  \frac{\|u_h\|_{1,*;V_h} }{\|u_h\|_{1,*} }
              \sup_{u_h\in\mathbb{U}_h}\sup_{v_h \in \mV_h } \frac{b_h(u_h,v_h)}{\|u_h\|_{1,*;V_h} \|v_h\|_{2,*}},
            \end{equation*}
        and deduce by Theorem \ref{thm:inf-sup-stability} that
        $
          c_* \le \Lambda_h \Cbh \leq \Lambda_h \Cbp,
        $
        as asserted.
        \end{proof}

        We now argue as in \cite[Proposition 3.8]{tantardini_l2-projection_2016} to find a
        characterization of the critical $\Lambda_h$ in terms of negative-index norms in $V_h$.

        \begin{lemma}[characterization of $\Lambda_h$]\label{eq:character-Lambdah}
            The quantity $\Lambda_h$ defined in \eqref{eq:def-Lambda} satisfies
            \begin{equation}\label{eq:caracter-Lambdah}
              \Lambda_h = \inf_{v_h \in V_h} \frac{\|v_h\|_{V^{-1}_h}}{\|v_h\|_{H^{-1}_*}} \le 1.
            \end{equation}
        \end{lemma}
        \begin{proof}
            We observe that Lemma \ref{lemma:relationship_bw_the_dual_norms} (relation between the dual norms) yields
            \[
               \lambda_h := \inf_{v_h\in V_h} \frac{\|v_h\|_{V^{-1}_h}}{\|v_h\|_{H^{-1}_*}} \le 1.
            \]
            Therefore, definitions \eqref{eq:1*-Vh} and \eqref{eq:1*} give $\Lambda_h \ge \lambda_h$.
            To show equality, we construct a sequence $v_h = \phi(t,n) w_h \in \mathbb{V}_h$ with
            $w_h\in V_h$ and
            \[
               \phi(t,n) := \sin\big( 2\pi n T^{-1} t\big).
            \]
            We see that 
            \begin{align*}
               \frac{\|v_h\|_{1,*;V_h}}{\|v_h\|_{1,*}} &= 
               \frac{\int_I \|\partial_t v_h\|_{V^{-1}_h} + |v_h|_{H^1_*}^2}{\int_I \|\partial_t v_h\|_{H^{-1}_*} + |v_h|_{H^1_*}^2} 
               \\ &
               = \frac{\|w_h\|_{V^{-1}_h}^2 \frac{4\pi^2n^2}{T^2} \int_I \phi^2 dt + |w_h|_{H^1_*}^2 \int_I \phi^2 dt}
               {\|w_h\|_{H^{-1}_*}^2 \frac{4\pi^2n^2}{T^2} \int_I \phi^2 dt + |w_h|_{H^1_*}^2 \int_I \phi^2 dt}
               \\&
               = \frac{\|w_h\|_{V^{-1}_h}^2 + \frac{T^2}{4\pi^2 n^2} |w_h|_{H^1_*}^2}{\|w_h\|_{H^{-1}_*}^2 + \frac{T^2}{4\pi^2 n^2} |w_h|_{H^1_*}^2}
               \longrightarrow \frac{\|w_h\|_{V^{-1}_h}^2}{\|w_h\|_{H^{-1}_*}^2},
            \end{align*}
            as $n\to\infty$. We now choose $w_h\in V_h$ that realizes the infimum $\lambda_h$ of the right-hand side to deduce the desired equality.
        \end{proof}

        The constant $\Lambda_h$ is related to the operator norms $\|P_h\|_{\cL(H^1_\G)}, \|P_h\|_{\cL(H^1_*)}$ and the inverse parameter $\Cinvh$. We explore this next.

        \begin{lemma}[relation between $\Lambda_h$ and $P_h$]\label{lem:Lambdah-Ph}
            The following inequalities hold for any fixed mesh $\cT_h$ satisfying Assumption \ref{assump:res-of-geom} (resolution of the geometry):
            \begin{equation}\label{eq:Lambdah-Ph}
               \|P_h\|_{\cL(H^1_\G)} \le \frac{1}{\Lambda_h} \le \|P_h\|_{\cL(H^1_*)} + \Cinvh,
            \end{equation}
            where $\Cinvh$ is given by Definition \ref{def:Cinvh} (inverse parameter).
        \end{lemma}
        \begin{proof}
            We start by rewriting Lemma \ref{lemma:relationship_bw_the_dual_norms} (relation between the dual norms) as follows:
            \[
                \frac{1}{\Lambda_h} = \sup_{v_h\in V_h} \frac{\|v_h\|_{H^{-1}_*}}{\|v_h\|_{V^{-1}_h}}
                \le \|P_h\|_{\cL(H^1_*)} + \Cinvh.
            \]
            This is the rightmost inequality in \eqref{eq:Lambdah-Ph}. To derive the remaining
            inequality, we observe that
            \[
                 \frac{1}{\Lambda_h} 
                 \ge \sup_{\ell\in H^{-1}_\G} \frac{\|P_h\ell\|_{H^{-1}_*}}{\|P_h \ell\|_{V^{-1}_h}}
                 \ge \sup_{\ell\in H^{-1}_\G} \frac{\|P_h\ell\|_{H^{-1}_\G}}{\|P_h \ell\|_{V^{-1}_h}},
            \]
            because of \eqref{eq:stabilized_norms}.
            Moreover, by virtue of \eqref{eq:discrete_dual_norm} and \eqref{eq:intro:l2proj}, we see that
            \begin{equation*}
                \|P_h \ell\|_{V^{-1}_h} = \sup_{v_h\in V_h} \frac{(P_h\ell,v_h)_*}{\|v_h\|_{H^1_*}}
                = \sup_{v_h\in V_h} \frac{\langle\ell,v_h\rangle_\G}{\|v_h\|_{H^1_*}} 
               \le \|\ell\|_{H^{-1}_\G},
            \end{equation*}
            which in turn implies
            \[
                \frac{1}{\Lambda_h} \ge \sup_{\ell\in H^{-1}_\G} \frac{\|P_h\ell\|_{H^{-1}_\G}}{\|\ell\|_{H^{-1}_\G}} = \|P_h\|_{\cL(H^{-1}_\G)}.
            \]
            We finally invoke \eqref{eq:Ph-H1-H-1} to infer that
            \[
                \frac{1}{\Lambda_h} \ge \|P_h\|_{\cL(H^{-1}_\G)} = \|P_h\|_{\cL(H^1_\G)}
            \]
            This is the desired rightmost inequality in \eqref{eq:Lambdah-Ph}.
        \end{proof}
        
        We are now ready to establish the relation between the $H^1$-stability of $P_h$, the inverse parameter $\Cinvh$,
        and the discrete inf-sup constant $c_*$ in \eqref{eq:inf-sup-*}.
        
        \begin{theorem}[robust inf-sup constant]\label{thm:Ph-c*}
            The following inequalities are valid for any fixed mesh $\cT_h$ satisfying Assumption \ref{assump:res-of-geom} (resolution of the geometry):
            \begin{equation}\label{eq:Pd-c*}
                \frac{\cbm}{\| P_h \|_{\cL(H^1_*)} + \Cinvh} \le c_* \le \frac{\Cbp}{\| P_h \|_{\cL(H^1_\G)}},
            \end{equation}
            where $\cbm, \Cbp$ are from Proposition \ref{proposition:inf-sup_stability}.
        \end{theorem}
        \begin{proof}
            Simply concatenate \eqref{eq:cd-c*} with \eqref{eq:Lambdah-Ph}.
        \end{proof}

        If the sequence of triangulations $\{\cT_h\}_h$ is well-behaved so that the LHS of \eqref{eq:Pd-c*} is bounded away from zero uniformly, then the semidiscrete parabolic problem is inf-sup stable. In Section~ \ref{section:stabilized_projections} we will show that this is indeed the case if $\{\cT_h\}_h$ is shape-regular and quasi-uniform. The question of whether the LHS of \eqref{eq:Pd-c*} is bounded away from zero uniformly under weaker assumptions on the mesh is completely open.

        On the other hand, if the sequence of meshes is badly behaved in the sense that $\|P_h\|_{\cL(H^1_\G)}$ blows up as $h \to 0$, then \eqref{eq:Pd-c*} implies that the semidiscrete parabolic problem is not uniformly inf-sup stable. 

        \begin{remark}[comparison to \cite{tantardini_l2-projection_2016}]
            The analogue of \eqref{eq:Pd-c*} proved by Tantardini and Veeser \cite{tantardini_l2-projection_2016} for flat elements reads
            \[\frac{\cbm}{\| \cP_h \|_{\cL(H^1)}} \le c_* \le \frac{\Cbp}{\| \cP_h \|_{\cL(H^1)}},\] where $\cP_h$ is the bulk $L^2$-projection. The widening of the gap between the necessary and sufficient conditions for semidiscrete inf-sup stability is the consequence of the unfittedness of TraceFEM, and the presence of stabilization. It is not clear whether this gap can be closed, i.e. whether $H^1_\G$- or $H^1_*$-stability of $P_h$ alone is both necessary and sufficient for the discrete inf-sup stability.
        \end{remark}

        We are finally in a position to combine Theorems \ref{thm:inf-sup-stability} and \ref{thm:Ph-c*} to establish the main result of this section.

        \begin{theorem}[robust inf-sup stability]
            \label{thm:discrete-inf-sup-stability}
            Let $\T_h$ be a fixed mesh satisfying Assumption \ref{assump:res-of-geom} (resolution of the geometry).
            The bilinear form $b_h$ in \eqref{eq:discrete_problem} is continuous and inf-sup stable in the spaces $(\mathbb{U}_h,\|\cdot\|_{1,*})$ and $(\mathbb{V}_h,\|\cdot\|_{2,*})$ with constants $C_*, c_*$ defined in \eqref{eq:inf-sup-continuity} and \eqref{eq:inf-sup-*} given by
            \begin{equation}\label{eq:inf-sup-stability}
                C_* \le \sqrt{2}, \qquad c_* \ge c_*^- := \frac{1}{\sqrt8 (1+2T)(\|P_h\|_{\cL(H^1_*)} + \Cinvh)}.
            \end{equation}
            The asymptotic dependence on the final time $T = \sup_{t \in I} t$ is optimal. Moreover, if functions in the pair of spaces $(\mathbb{U}_h,\mathbb{V}_h)$ have vanishing meanvalue for every $t\in I$, then
            \begin{equation}\label{eq:improved-inf-sup}
                c_*^- = \frac{1}{\sqrt8 (\|P_h\|_{\cL(H^1_*)} + \Cinvh)}.
            \end{equation}
        \end{theorem}
        \begin{proof}
            The bound on $C_*$ is given in Lemma \ref{lem:C*} (robust continuity constant). For the bound on $c_*$, we concatenate Theorem \ref{thm:Ph-c*} (robust inf-sup constant) and Theorem \ref{thm:inf-sup-stability} ($V_h$-dependent inf-sup stability) to obtain \eqref{eq:inf-sup-stability}, namely
            \[
            c_* \geq \frac{c_b^-}{\|P_h\|_{\cL(H^1_*)} + \Cinvh} = \frac{1}{\sqrt8 (1+2T)(\|P_h\|_{\cL(H^1_*)} + \Cinvh)}.
            \]
            The estimate \eqref{eq:improved-inf-sup} for functions with vanishing meanvalue on $\G$ is a consequence of the improved lower bound $c_b^- = 8^{-1/2}$ for the inf-sup constant $c_b$ in Theorem \ref{thm:inf-sup-stability}.
       \end{proof}

%%%%%%%%%%%%%%%%%%%%%%%%%%%%%%%%%%%%%%%%%%%%%%%%%%%%%%%%%%%%%%%%%%%%%%%%%%%%%%%
\section{Consequences of uniform inf-sup stability} \label{section:conseq}
%%%%%%%%%%%%%%%%%%%%%%%%%%%%%%%%%%%%%%%%%%%%%%%%%%%%%%%%%%%%%%%%%%%%%%%%%%%%%%%
    Theorem \ref{thm:discrete-inf-sup-stability} (robust inf-sup stability) holds for every fixed partition $\cT_h$ satisfying Assumption \ref{assump:res-of-geom} (resolution of the geometry). For a sequence of partitions $\{\cT_h\}_h$ satisfying Assumption \ref{assump:res-of-geom}, we say that the discrete inf-sup stability holds \textit{uniformly} in $h$ if $c_*$ is finite and is bounded away from $0$ uniformly in $h$. Following Theorem \ref{thm:inf-sup-stability}, the next Assumption on sequences of meshes ensures that uniform inf-sup stability holds. 

    \begin{assump}[mesh stability] \label{assump:mesh-stab}
        For a sequence of bulk meshes $\{\cT_h\}_h$ fullfilling Assumption \ref{assump:res-of-geom} (resolution of the geometry), $\|P_h\|_{\cL(H^1_*)}, \Cinvh$ are bounded uniformly in $h$.
    \end{assump}

    The purpose of this section is to establish the following properties that hold for sequences of meshes satisfying Assumption \ref{assump:mesh-stab} (mesh stability).

    \begin{itemize}
        \item uniform well-posedness of the semidiscrete problem (Section \ref{sec:well-posedness}),
        \item discrete maximal parabolic regularity (Section \ref{sec:max-parab-reg}),
        \item quasi-best approximation (Section \ref{section:quasi-best-approx}),
        \item convergence to minimal regularity solutions under an additional approximability assumption on $\{\cT_h\}_h$ (Section \ref{sec:convergence}).
    \end{itemize}
    
    We want to emphasize that the sequence of bulk meshes $\{\cT_h\}_h$ is not assumed to be quasi-uniform or shape-regular in this section. In Section \ref{section:stabilized_projections} (Stabilized $L^2$-projection) we shall prove that Assumption \ref{assump:mesh-stab} is satisfied if $\{\cT_h\}_h$ is shape-regular and quasi-uniform (see Lemma \ref{lemma:refined_inverse_estimates} and Proposition \ref{theorem:h1*_stability_of_stab_l2_proj}).

    %------------------------------------------------------------------------------    
    \subsection{Uniform well-posedness of semidiscrete problem}\label{sec:well-posedness}
    %------------------------------------------------------------------------------  
    
        Uniform well-posedness of \eqref{eq:discrete_problem} is a consequence of the Banach-Ne\v cas Theorem \cite[Theorem 1]{NochettoSiebertVeeser:09}, \cite{ErnGuermond_2004}, which in turn hinges on two crucial properties. The first one is the uniform inf-sup property of Theorem \ref{thm:discrete-inf-sup-stability}. The second one is a nondegeneracy condition of the adjoint problem. This is the purpose of our next result which follows an argument from \cite{tantardini_l2-projection_2016}.

        \begin{lemma}[adjoint]\label{lem:adjoint}
            Let $\cT_h$ satisfy Assumption \ref{assump:res-of-geom} (resolution of the geometry). The kernel of the adjoint problem is trivial, namely if $v_h\in V_h$ satisfies
          \begin{equation}\label{eq:nondegeneracy}
            b_h(u_h, v_h) = 0 \quad \forall u_h \in \mU_h \quad\implies\quad v_h = 0.
          \end{equation}
        \end{lemma}
        \begin{proof}
            We first observe that $v_h\in L^2(I;V_h)$ in \eqref{eq:nondegeneracy} does not have any a priori regularity in time. However, \eqref{eq:nondegeneracy} implies for all $\widetilde u_h \in C_0^\infty(I; V_h)$
            \[
               \left|\int_I \langle \pa_t \widetilde u_h, v_{1,h} \rangle_* \right| = \left|\int_I a_*(\widetilde u_h, v_{1,h}) \right| \leq |\widetilde u_h|_{L^2(I; H^1_*)} |v_{1,h}|_{L^2(I; H^1_*)},
            \]
            or equivalently that $v_{1,h}$ admits a weak derivative in the dual space of $L^2(I; V_h)$, which is isomorphic to $L^2(I; V_h^{-1})$. Consequently, $v_{h,1} \in H^1(I;V_h^{-1})\subset C^{1/2}(I;V_h^{-1})$ and integration by parts in time is valid against any $u_h\in\mathbb{U}_h$. Since $b_h(u_h,v_h)=0$, this gives rise to
            \[ 
             (u_h(0), v_{0,h} - v_{1,h}(0))_* + (u_h(T), v_{1,h}(T))_* - \int_I \langle u_h, \pa_t v_{1,h} \rangle_* + \int_I a_*(u_h, v_{1,h}) = 0.
            \]
            We first choose $u_h\in\mathbb{U}_h$ so that $u_h(0)=u_h(T)=0$ to deduce $\int_I - \langle u_h, \pa_t v_{1,h} \rangle_* + a_*(u_h, v_{1,h}) = 0$, and next $u_h(0)$ and $u_h(T)$ arbitrary to obtain $v_{1,h}(0) = v_{0,h}$ and $v_{1,h}(T) = 0$. Taking now $u_h=v_{1,h}$ results in
            \[
              0 = - \int_I \langle v_{1,h}, \pa_t v_{1,h} \rangle_* + \int_I a_*(v_{1,h}, v_{1,h}) = \frac12 \|v_{0,h}\|_{L^2_*}^2 + \int_I |v_{1,h}|_{H^1_*}^2,
            \]
             whence $v_{0,h} = 0, v_{1,h} = C \in \mathbb{R}$. To show that $C=0$, we finally choose $u_h$ to be the characteristic function of $(0,t)$ to deduce the following property for the meanvalue $\bar v_{1,h}(t)$ of $v_{1,h}(t)$
             \[
               \int_0^t \langle 1, \pa_t v_{1,h} \rangle = 0 
               \quad\implies\quad 
               \bar v_{1,h}(t) = \bar v_{0,h} \quad \forall t \in I.
             \]
             Since $C=v_{1,h}=\bar v_{1,h}(t)=0$, we infer that $v_h=0$ as asserted.
        \end{proof}

        \begin{proposition}[well-posedness]\label{prop:well-posedness}
            Let $\{\T_h\}_h$ satisfy Assumption \ref{assump:mesh-stab} (mesh stability). There exists a unique solution $u_h\in\mathbb{U}_h$ of the semidiscrete problem \eqref{eq:discrete_problem} which depends continuously on general data $\ell=(u_0,f)\in\mathbb{D}$ with uniform constant $c_*^{-1}$ given by \eqref{eq:inf-sup-stability}
            \[
              \|u_h\|_{1,*} \le \frac{1}{c_*} \, \|(u_0,f)\|_{3,*}
              = \frac{1}{c_*} \, \sup_{v_h\in V_h} \frac{\ell(v_h)}{\|v_h\|_{2,*}}.
            \]
            Instead, if $u_0$ and $f(\cdot,t)$ have vanishing meanvalue for a.e. $t\in I$, then $c_*$ is given by \eqref{eq:improved-inf-sup}.
        \end{proposition}
        \begin{proof}
            We simply apply the Banach-Ne\v cas Theorem \cite[Theorem 1]{NochettoSiebertVeeser:09}
            \cite{ern_finite_2021} in conjunction with Theorem \ref{thm:discrete-inf-sup-stability} (robust inf-sup stability) and Lemma \ref{lem:adjoint} (adjoint).
        \end{proof}
    
%------------------------------------------------------------------------------------------
    \subsection{Discrete Maximal Parabolic Regularity} \label{sec:max-parab-reg}
%------------------------------------------------------------------------------------------
        We now want to point out an interesting observation which is a consequence of Lemma \ref{lemma:relationship_bw_the_dual_norms} (relation between the dual norms), Theorem \ref{thm:inf-sup-stability} ($V_h$-dependent inf-sup stability), and Theorem \ref{thm:discrete-inf-sup-stability} (robust inf-sup stability). First, let us define the discrete Laplace-Beltrami operator.
        \begin{definition}[discrete Laplace-Beltrami operator]
            Given $v_h \in V_h$, find $\Delta_h v_h \in V_h$ s.t.
            \begin{equation} \label{eq:discr_lapl}
                \langle \Delta_h v_h, w_h \rangle_* = a_*(v_h, w_h), \quad \forall w_h \in V_h.
            \end{equation}
        \end{definition}
        \begin{lemma}[stability of $\Delta_h$] \label{lem:discr_lapl_stability}
            Given $v_h \in V_h$, for $\Delta_h v_h \in V_h$ solving \eqref{eq:discr_lapl}, it holds
            \begin{equation} \label{eq:discr_lapl_stability}
                \|\Delta_h v_h\|_{H^{-1}_*} \leq (\|P_h\|_{\mathcal{L}(H^1_*)} + \Cinvh) |v_h|_{H^1_*}.
            \end{equation}
        \end{lemma}
        
        \begin{proof}
            Assertion \eqref{eq:discr_lapl_stability} follows by first getting a bound on $\|\Delta_h v_h\|_{V_h^{-1}}$ using Cauchy-Schwarz and Definition \ref{def:discrete_dual_norm} (discrete dual norm), then invoking Lemma \ref{lemma:relationship_bw_the_dual_norms} (relation between the dual norms).
        \end{proof}
        \begin{lemma}[control of $\Delta_h u_h$] \label{lem:control_of_deltah_uh}
            If $u_h \in \mU_h$ solves \eqref{eq:discrete_problem}, then it holds
            \begin{equation}\label{eq:discr_lapl_uh}
                \|\Delta_h u_h\|_{L^2_t H^{-1}_*} \leq (1/c_*^-) (\|f\|_{L^2_t H^{-1}_\G} + \|u_0\|_{L^2_\G}),
            \end{equation}
            where $c_*^-$ is the lower bound on $c_*$ from Theorem \ref{thm:discrete-inf-sup-stability} (robust inf-sup stability).
        \end{lemma}
        \begin{proof}
            Since $u_h \in H^1_t V_h$, let $\Delta_h u_h(t)$ be the discrete Laplace-Beltrami operator of $u_h(t)$ for all $t \in I$. We first square, then integrate both sides of Lemma \ref{lem:discr_lapl_stability} (stability of $\Delta_h$) over $I$ to obtain
            \[\int_I \|\Delta_h u_h\|_{H^{-1}_*}^2 \leq (\|P_h\|_{\mathcal{L}(H^1_*)} + \Cinvh)^2 \int_I |u_h|_{H^1_*}^2.\]
            Then assertion \eqref{eq:discr_lapl_uh} follows from Theorem \ref{thm:inf-sup-stability} ($V_h$-dependent inf-sup stability), which gives 
            \[\int_I |u_h|_{H^1_*}^2 \leq c_b^{-2} (\|f\|_{L^2_t H^{-1}_\G}^2 + \|u_0\|_{L^2_\G}^2),\] 
            and recognizing that $c_b / (\|P_h\|_{\mathcal{L}(H^1_*)} + \Cinvh)$ is a lower bound on $c_*$ from Theorem \ref{thm:discrete-inf-sup-stability} (robust inf-sup stability).
        \end{proof}
        We are now in the position to state the main result of this subsection.
        \begin{proposition}[discrete maximal parabolic regularity]\label{prop:max-regularity}
            Let $\{\T_h\}_h$ satisfy Assumption \ref{assump:mesh-stab} (mesh stability). If $u_h \in \mU_h$ solves \eqref{eq:discrete_problem}, then it holds:
            \begin{equation} \label{eq:dmpr}
                \|\Delta_h u_h\|_{L^2_t H^{-1}_\G} + \|\pa_t u_h\|_{L^2_t H^{-1}_\G} \leq C_{\mathrm{MPR}} (\|f\|_{L^2_t H^{-1}_\G} + \|u_0\|_{L^2_\G}).
            \end{equation}
            Here $C_{\mathrm{MPR}}$ is a multiple of $1/c_*^-$, where $c_*^-$ is the lower bound on $c_*$ from Theorem \ref{thm:discrete-inf-sup-stability}.
        \end{proposition}
        \begin{proof}
            We combine the results of Lemma \ref{lem:control_of_deltah_uh} (control of $\Delta_h u_h$) and Theorem \ref{thm:discrete-inf-sup-stability} (robust inf-sup stability).
        \end{proof}

    %----------------------------------------------------------------------------------   
    \subsection{Quasi-best approximation} \label{section:quasi-best-approx}
    %----------------------------------------------------------------------------------
    
        Our next goal is to prove that Theorem \ref{thm:discrete-inf-sup-stability} (robust inf-sup stability) implies a quasi-best approximation estimate for the parabolic error functional $\bE[u, z_h]$ defined in \eqref{eq:intro:energy_error_functional}: for any $u \in \mathbb{U}, z_h \in \mathbb{U}_h$
        \begin{equation}\label{eq:parabolic-error}
            \bE[u, z_h]^2 = \bE_{L^2_*}[u(0), z_h(0)]^2 + \int_I \bE_{H^{-1}_*}[\pa_t u, \pa_t z_h]^2 + \bE_{H^1_*}[u, z_h]^2.
        \end{equation}
        If $u\in\mathbb{U}$ and $u_h\in\mathbb{U}_h$ solve \eqref{eq:cont_problem} and  \eqref{eq:discrete_problem}, then they satisfy the following {\it Galerkin orthogonality}
        \begin{equation} \label{eq:parabolic_quasi_orth}
            0 = b(u,v_h) - b_h(u_h,v_h) = ((u - u_h)(0), v_{0,h})_\#
          + \int_I  \langle \pa_t (u - u_h), v_{1,h} \rangle_\#
          + a_\#(u - u_h, v_{1,h}),
        \end{equation}
        for all $v_h = (v_{0,h}, v_{1,h}) \in \mathbb{V}_h = V_h \times L^2_t V_h$ with
        \begin{align*}
            ((u - u_h)(0), v_{0,h})_\# &:= ((u - u_h)(0), v_{0,h})_\G - s_0(u_h(0), v_{0,h})_\G
            \\
            \langle \pa_t (u - u_h), v_{1,h} \rangle_\# &:= \langle \pa_t (u - u_h), v_{1,h} \rangle_\G - s_0(\pa_t u_h, v_{1,h})
            \\
            a_\#(u - u_h, v_{1,h}) &:= a(u - u_h, v_{1,h}) - s_1(u_h, v_{1,h}).
        \end{align*}
        The following quasi-optimality estimate justifies the parabolic error notion in \eqref{eq:parabolic-error}.
      
        \begin{theorem}[parabolic quasi-best approximation] \label{theorem:quasi-best-approx}
            Let $\{\T_h\}_h$ satisfy Assumption \ref{assump:mesh-stab} (mesh stability). Then the solutions $u\in\mathbb{U}$ of \eqref{eq:cont_problem} and $u_h\in\mathbb{U}_h$ of \eqref{eq:discrete_problem} satisfy
            \begin{equation}\label{eq:quasi-best-approx}
                \bE[u, u_h] \leq \left( 1 + \frac{C_*}{c_*} \right) \inf_{z_h \in \mathbb{U}_h} \bE[u, z_h] ,
            \end{equation}
            where $C_*$ and $c_*$ are the continuity and inf-sup constants of $b_h$ and satisfy
            \eqref{eq:inf-sup-stability}.
        \end{theorem}

        \begin{proof}
            We invoke Theorem \ref{thm:discrete-inf-sup-stability} in conjunction with
            \eqref{eq:inf-sup-*} for an arbitrary $z_h\in \mathbb{U}_h$, namely
            \[
               c_* \|u_h - z_h\|_{1,*} \leq \sup_{v_h \in \mathbb{V}_h} \frac{b_h(u_h - z_h, v_h)}{\|v_h\|_{2,*}},
            \]
            and use Galerkin orthogonality \eqref{eq:parabolic_quasi_orth} to rewrite the numerator as follows:
            \begin{equation}
                \begin{split}
                    b_h(u_h - z_h, v_h) & = ((u_h - z_h)(0), v_{0,h})_* + \int_I \langle \pa_t (u_h - z_h), v_{1,h} \rangle_* + a_*(u_h - z_h, v_{1,h}) \\
                    & = ((u - z_h)(0), v_{0,h})_\# + \int_I \langle \pa_t (u - z_h), v_{1,h} \rangle_\# + 
                    a_\#(u - z_h, v_{1,h}).
                \end{split}
            \end{equation}
            Applying this, the continuity constant of $b$ in \eqref{eq:inf-sup-*}, and the Cauchy-Schwarz inequality yields
            \begin{equation*}
                b_h(u_h - z_h, v_h) \leq C_* \, \bE[u, z_h] \, \|v_h\|_{2,*}
                \quad\Rightarrow\quad
                \|u_h - z_h\|_{1,*} \le \frac{C_*}{c_*} \, \bE[u, z_h].
            \end{equation*}
            Finally, \eqref{eq:quasi-best-approx} follows from the triangle inequality
            $\bE[u, u_h]  \leq \bE[u, z_h] + \|u_h - z_h\|_{1,*}$.
        \end{proof}

        The proof of Theorem \ref{theorem:quasi-best-approx} follows the original argument by Babu\v ska \cite{Babuska:71}, 
        \cite[Theorem 5]{NochettoSiebertVeeser:09}. Since $P_h$ is not an idempotent operator from
        $L^2_\G$ into $V_h$, it does not seem possible to remove the constant $1$ in \eqref{eq:quasi-best-approx} as proposed by Xu and Zikatanov \cite{XuZikatanov:03}. We refer to \cite{tantardini_l2-projection_2016} for the flat case. The estimate \eqref{eq:quasi-best-approx} is {\it symmetric} in the sense that the same error notion appears on its left and right-hand sides.

    % ------------------------------------------------------------------------------    
    \subsection{Convergence} \label{sec:convergence}
    % ------------------------------------------------------------------------------
        The following weak convergence result is new in the literature of TraceFEM. The reason is that it relies on a uniform bound of $\|u_h\|_{H^1_t H^{-1}_*}$, which in turn is a by-product of our uniform robust inf-sup stability. This result is valid under {\it minimal regularity} of data. To prove convergence, we shall need the following additional assumption on sequences of meshes that ensures approximability in $L^2_*$ and $H^1_*$.

        \begin{assump}[approximability in $L^2_*$ and $H^1_*$] \label{assump:approx-in-l2-and-h1}
            For a sequence of bulk meshes $\{\cT_h\}_h$ fulfilling Assumption \ref{assump:res-of-geom} (resolution of the geometry), it holds that $\bE_{L^2_*}[w, P_h w] \to 0$ for all $w \in L^2_\G$, and $\bE_{H^1_*}[v, P_h v] \to 0$ for all $v \in H^1_\G$.
        \end{assump}

        In Section \ref{section:stabilized_projections} we prove that this assumption is satisfied for families of shape-regular and quasi-uniform bulk partitions (refer to Lemma \ref{lem:quasi-best-H1*}).
        
        \begin{proposition}[weak convergence under minimal regularity] \label{prop:convergence}
            Let $\{\cT_h\}_h$ satisfy Assumptions~\ref{assump:mesh-stab} (mesh stability) and \ref{assump:approx-in-l2-and-h1} (approximability in $L^2_*$ and $H^1_*$). Then the trace of the solution $u_h\in\mathbb{U}_h$ of \eqref{eq:discrete_problem} converges weakly in $L^2_t H^1_\G$
            and $H^1_t H^{-1}_\G$, and strongly in $L^2_t L^2_\G$, to a function $u \in C^0([0,T], L^2_\G)$ which is a weak solution of \eqref{eq:cont_problem}, namely it satisfies
            \[
              b(u,v) = \ell(v) \qquad\forall v \in \mathbb{V}.
            \]
        \end{proposition}
        \begin{proof}
            The argument below is classical but we include it for completeness; we refer to \cite[Section 7.1, Theorem 3]{evans_partial_2010} for the flat case. The proof is brief but we split it in three steps.
        
            \medskip\noindent
            1. {\it Compactness}. Theorem \ref{thm:discrete-inf-sup-stability} (robust inf-sup stability) implies the uniform bound
            \[
               \|u_h\|_{1} \le \|u_h\|_{1,*} \le \frac{1}{c_*} \|(u_0,f)\|_{3,*} 
               \le \frac{1}{c_*} \|(u_0,f)\|_3,
            \]
            where $\|u_h\|_1$ is defined in \eqref{eq:coninuous-norms}. Invoking weak compactness there exist $u \in L^2_t H^1_\G \cap H^1_t H^{-1}_\G$ so that
            \[
               u_h|_\G \rightharpoonup u \quad \mathrm{in} \quad L^2_t H^1_\G,
               \qquad
                \pa_t u_h|_\G \rightharpoonup \pa_t u \quad \mathrm{in} \quad L^2_t H^{-1}_\G,
            \]
            and, since $L^2_t H^1_\G \cap H^1_t H^{-1}_\G$ is compactly embedded in $L^2_t L^2_\G$,
            by Aubin-Lions \cite{ErnGuermond_2004},
            \[
              u_h|_\G \to u \quad \mathrm{in} \quad L^2_t L^2_{\G},
            \]
            for subsequences (not relabeled). Moreover, Lions-Magenes \cite{ErnGuermond_2004} implies 
            $u \in C^0([0,T], L^2_\G)$ and for all $v\in C^1([0,T],H^1_\G)$ we have
            \begin{equation}\label{eq:integration-by-parts}
              ( u(T),v(T) )_\G - ( u(0),v(0) )_\G 
              = \int_I \langle \partial_t u, v \rangle_\G + \langle u, \partial_t v \rangle_\G.
            \end{equation}
        
            \medskip\noindent
            2. {\it Weak solution}. We assert that the following weak equation is valid for all 
            $v\in C^1([0,T],H^1_\G)$
            \begin{equation}\label{eq:weak-form-u}
              \int_I \langle \partial_t u, v\rangle_\G + a(u,v) - \langle f, v\rangle_\G = 0.
            \end{equation}
            To show this, we consider the discrete test function $w_h=(0,P_h v)\in\mathbb{V}_h$ and use that $u_h\in\mathbb{U}_h$ is a semidiscrete solution, namely $b_h(u_h, w_h) - \ell(w_h) = 0$.
            Therefore, for $w=(0,v)\in\mathbb{V}$ we see that
            \begin{align*}
                b(u,w)-\ell(w) & = b(u,w) - b_h(u_h,w_h) - \ell(w-w_h) =  \mathbb{I}_h +  \mathbb{II}_h
            \end{align*}  
            with
            \begin{align*}
                \mathbb{I}_h &:= \int_I \langle \partial_t u - \partial_t u_h, v\rangle_\G + a(u-u_h,v),
                \\
                 \mathbb{II}_h &:= \int_I a(u_h,v - P_h v) + s_1(u_h,v^e-P_hv) - \langle f, v-P_h v\rangle_\G
            \end{align*}
            because $(\partial_t u_h, P_h v)_* - (\partial_t u_h,v)_\G = 0$ by definition of $P_h$. Step 1 implies $\mathbb{I}_h\to 0$ whereas
            \[
              \big|  \mathbb{II}_h \big| \le \sqrt2 \Big( \int_I |u_h|_{H^1_*}^2 + \|f\|_{H^{-1}_\G}^2 \Big)^{1/2} 
              \bE_{H^1_*}[v,P_h v] \lesssim \|(u_0,f)\|_3 \, \bE_{H^1_*}[v,P_h v] \to 0,
            \]
            the latter because of Assumption~\ref{assump:approx-in-l2-and-h1} (approximability in $L^2_*$ and $H^1_*$). This proves the assertion.
        
            \medskip\noindent
            3. {\it Initial condition}. It remains to prove that $u(0)=u_0$. We integrate \eqref{eq:weak-form-u} using \eqref{eq:integration-by-parts} for $v(T) = 0$
            \[
              -( u(0), v(0) )_\G + \int_I - \langle u, \partial_t v\rangle_\G +
              a(u,v) - \langle f, v \rangle_\G = 0.
            \]
            A similar calculation for $u_h$, and passage to the limit as in Step 2, yields
            \begin{align*}
              0 & = -( u_h(0), v(0) )_\G + \int_I - \langle u_h, \partial_t v\rangle_\G +
              a(u_h,P_h v)_* - \langle f, P_h v \rangle_\G 
              \\
              & = -( \lim_{\cT_h} u_h(0), v(0) )_\G + \int_I - \langle u, \partial_t v\rangle_\G +
              a(u,v) - \langle f, v \rangle_\G.
            \end{align*}
            Hence, comparing these two expressions gives
            \[
              (u(0)-\lim_{\cT_h} u_h(0), v(0) )_\G = 0 
              \quad\Rightarrow\quad u(0) = \lim_{\cT_h} u_h(0).
            \]
            Moreover, $u_h(0) = P_h u_0 \to u_0$ in $L^2_\G$ strongly as per Assumption \ref{assump:approx-in-l2-and-h1}. Therefore, $u(0)=u_0$ and $u$ solves \eqref{eq:cont_problem}. Since \eqref{eq:cont_problem} admits a unique weak solution, the
            entire sequence $u_h$ converges to $u$, thereby concluding the proof.
        \end{proof} 

        We conclude this section by stating the strong counterpart of Proposition \ref{prop:convergence}, which requires the following assumption on $\{\cT_h\}_h$.

        \begin{assump}[parabolic approximability] \label{assump:approx-in-U}
            For a sequence of bulk meshes $\{\cT_h\}_h$ fulfilling Assumption \ref{assump:res-of-geom} (resolution of the geometry), it holds that $\inf_{z_h \in \mU_h} \bE[u, z_h] \to 0$ for all $u \in \mU$, where $\bE$ is the parabolic error functional defined in \eqref{eq:parabolic-error}, $\mU$ is the solution space defined in \eqref{eq:coninuous-norms}, and $\mU_h$ is the semidiscrete solution space defined in \eqref{eq:semidiscrete-spaces}.
        \end{assump}
    
        In Section~\ref{section:stabilized_projections} we shall prove that this assumption is satisfied for families of shape-regular and quasi-uniform bulk partitions upon taking $z_h:=P_hu$ (see Lemmas \ref{lemma:Ph_L2_estimates}, \ref{lem:quasi-best-H1*} and \ref{lemma:stabilized_l2_projection_h'_error_est}). We are now in a position to state the result.

        \begin{proposition}[strong convergence under minimal regularity]
            \label{prop:strong-convergence}
            Let $\{\cT_h\}_h$ satisfy Assumptions~\ref{assump:mesh-stab} (mesh stability) and \ref{assump:approx-in-U} (parabolic approximability). Then the solution $u_h\in\mathbb{U}_h$ of \eqref{eq:discrete_problem} converges strongly to the weak solution $u$ of \eqref{eq:cont_problem} as $h\to0$, namely
            \begin{equation} \label{eq:strong-conv}
                \bE[u, u_h] \to 0.
            \end{equation}
        \end{proposition}
        \begin{proof}
            Proposition \ref{prop:convergence} (weak convergence under minimal regularity) guarantees existence of the weak solution $u$ of \eqref{eq:cont_problem}. The assertion \eqref{eq:strong-conv} follows from Theorem \ref{theorem:quasi-best-approx} (parabolic quasi-best approximation) and Assumption \ref{assump:approx-in-U} (parabolic approximability).
        \end{proof}

%%%%%%%%%%%%%%%%%%%%%%%%%%%%%%%%%%%%%%%%%%%%%%%%%%%%%%%%%%%%%%%%%%%%%%%%%%%%%%%%
\section{Stabilized \texorpdfstring{$L^2$}{L2}-projection} \sectionmark{Stabilized $L^2$-projection} \label{section:stabilized_projections}
%%%%%%%%%%%%%%%%%%%%%%%%%%%%%%%%%%%%%%%%%%%%%%%%%%%%%%%%%%%%%%%%%%%%%%%%%%%%%%%%
    A crucial part of our analysis relies on the stabilized $L^2$-projection $P_h$ defined in \eqref{eq:intro:l2proj}. In this section, we study properties of $P_h$, which are of paramount and self-standing interest. Our results appear novel and are helpful in analyzing any scheme based on stabilized TraceFEM for low-regularity problems. This study reveals the crucial role of the stabilization scalings set forth in \eqref{def:stabilization}.
    
    We will prove $H^1_*$-stability of $P_h$ and uniform boundedness of $\Cinvh$ from Definition \ref{def:Cinvh} provided Assumption \ref{assumption:shape-reg-quasi-unif} below is valid. Moreover, we will also show that Assumption~\ref{assumption:shape-reg-quasi-unif} implies Assumptions~\ref{assump:mesh-stab} (mesh stability), \ref{assump:approx-in-l2-and-h1} (approximability in $L^2_*$ and $H^1_*$) and \ref{assump:approx-in-U} (parabolic approximability), thereby leading to discrete uniform inf-sup stability and convergence without rates.
    
    \begin{assump}[shape-regularity and quasi-uniformity] \label{assumption:shape-reg-quasi-unif}
            The sequence of meshes $\{\cT_h\}_h$ fulfilling Assumption \ref{assump:res-of-geom} (resolution of the geometry) is shape-regular and quasi-uniform.
        \end{assump}

    We will derive optimal order-regularity error estimates in Section \ref{section:error_est}.
    We point out that the stabilized elliptic projection $R_h$, while interesting in its own right, does not enter our analysis, but can be used to derive a $L^\infty_t L^2_*$-error estimate (assuming additional regularity of the solution).
        
    %----------------------------
    \subsection{Interpolation and inverse estimates on trace spaces}\label{S:prelim}
    %---------------------------- 
       
        The following lemmas are instrumental in proving properties of a quasi-interpolation operator $I_h$ that is utilized later for the study of $P_h$. Since they are standard in the TraceFEM literature, we present them without proofs. 
        
        \begin{lemma}[local trace inequality {\cite[Remark 4.2]{reusken_analysis_2015}}] \label{prop:trace_inequality} For any $T\in \cT_h$, we have
            \begin{equation} \label{eq:trace_inequality}
                \|v\|_{L^2(\G \cap T)}^2 \lesssim h_T^{-1} \|v\|_{T}^2 + h_T \|\nabla v\|_{T}^2 \quad \forall v \in H^1(T).
            \end{equation}
        \end{lemma}

        \begin{lemma}[inverse bulk estimate {\cite[\S 4.5,  Theorem 4.5.11]{brenner_mathematical_2008}}] \label{prop:inverse_estimate}
           For any $T\in \cT_h$ and any $0 \leq k \leq l$, we have
            \begin{equation} \label{eq:inverse_estimate}
                \|v_h\|_{H^l(T)} \lesssim h_T^{k-l} \|v_h\|_{H^k(T)}\, \quad  \forall v_h \in V_h
            \end{equation}
        \end{lemma}
    
        \begin{lemma}[norm equivalence {\cite[Lemma 7.5]{grande_analysis_2018}}] \label{prop:norm_equiv}
            The sequence of meshes $\{\cT_h\}_h$ and the surface $\Gamma$ are such that the following estimate holds
            \begin{equation} \label{eq:norm_equiv}
                 \|v\|_{L^2(\Omega_h)}^2 \lesssim h \|v\|_{L^2_*}^2 = h \big(\|v\|_{L^2_\G}^2 + h \|\nabla_\bn v\|_{L^2(\Omega_h)}^2 \big)  \quad \forall v\in H^1(\Omega_h).
            \end{equation}
        \end{lemma}
        
        \begin{lemma}[global normal extension {\cite[Lemma 3.1]{reusken_analysis_2015}}] \label{prop:normal_extension_property}
            For all $v \in H^1_\G$
            and all $k \in \mathbb{N}_{\geq 1}$, 
            the following estimate holds
            \begin{equation} \label{eq:normal_extension_property}
                |v^e|_{H^k(\Omega_h)}^2 \lesssim h |v|_{H^k_\G}^2 \quad \forall v\in H^k_\G.
            \end{equation}
        \end{lemma}
        
        In this subsection we quantify the errors incurred by a bulk interpolant $I_h v^e$ of a normally extended function $v^e$ in the stabilized norms \eqref{eq:stabilized_norms}. Although the estimates for $v\in H^2(\G)$ are available in the literature, we need to cover the case of a minimally regular $v\in H^1(\G)$ so we repeat rather standard arguments for completeness.

        We will rely on the bulk Scott--Zhang quasi-interpolation operator \cite{brenner_mathematical_2008,scott_finite_1990} $I_h: \Omega_h \to V_h$, which has the following bulk properties provided Assumption \ref{assumption:shape-reg-quasi-unif} holds
        and $\sigma \in \{0, 1\}$:
        \begin{equation} \label{eq:scott_zhang_h1bulk_stability}
            |I_h v|_{H^\sigma(\Omega_h)} \lesssim |v|_{H^\sigma(\Omega_h)}, \quad\forall v\in H^\sigma(\Omega_h),
        \end{equation}
        \begin{equation} \label{eq:scott_zhang_hkbulk_stability}
            \sum_{T \in \cT_h} \|I_h v\|_{H^2(T)}^2 \lesssim |v|_{H^2(\Omega_h)}^2, \quad\forall v\in H^2(\Omega_h).
        \end{equation}
        and for $v \in H^{1+\sigma}(\Omega_h)$
        \begin{equation} \label{eq:scott_zhang_l2bulk_error}
            \|v - I_h v \|_{L^2(\Omega_h)} \lesssim h^{1+\sigma} | v |_{H^{1+\sigma}(\Omega_h)},
        \end{equation}
        \begin{equation} \label{eq:scott_zhang_h1bulk_error}
            \|v - I_h v \|_{H^1(\Omega_h)} \lesssim h | v |_{H^2(\Omega_h)}.
        \end{equation}
        Using $I_h v^e$ we may approximate a surface function $v\in H^1(\G)$ by constructing the interpolant $I_h v^e$ of its normal extension $v^e$ and restricting it back to $\G$. 
        This procedure is stable in $H^1_\G$.

        \begin{lemma}[$H^1_\G$-stability]\label{lem:H1G-stability}
            There exists a constant $C_{\mathrm{stab}}$ such that
            \begin{equation}\label{eq:interp_h1_gamma_stability}
                |I_h v^e|_{H^1_*} \le C_{\mathrm{stab}} |v|_{H^1_\G} \quad\forall v\in H^1_\G.
            \end{equation}
        \end{lemma}
        \begin{proof}
            We use Lemma \ref{prop:trace_inequality} (local trace inequality) and Lemma \ref{prop:inverse_estimate} (inverse bulk estimate) in conjunction with \eqref{eq:scott_zhang_h1bulk_stability} and Lemma 
            \ref{prop:normal_extension_property} (global normal extension) for $k=1$ to get
            \[
               |I_h v^e|_{H^1_*}^2 = |I_h v|_{H^1_\G}^2 + s_1(I_h v^e,I_h v^e)
               \lesssim h^{-1} | I_h v^e |_{H^1(\Omega_h)}^2 
               \lesssim h^{-1} | v^e |_{H^1(\Omega_h)}^2 \lesssim | v |_{H^1_\G}^2,
            \]
            as asserted.
        \end{proof}
        
     In addition to the approximation properties in surface norms, the forthcoming lemma states  that $I_h v^e$ also gives optimal approximability in terms of the error functionals $\bE_{L^2_*}$ and $\bE_{H^1_*}$ of \eqref{eq:intro:error_fnals}.
        
        \begin{lemma}[interpolation error estimates]\label{lemma:interpolation} 
            There exist constants $C_0,C_1$ such that for $\sigma \in \{0, 1\}$
            \begin{align}\label{eq:interp_est_l2*}
                \bE_{L^2_*}[v, I_h v^e] = \| v^e-I_h v^e\|_{L^2_*}^2 &\le C_0 h^{1+\sigma} |u|_{H^{1+\sigma}_{\G}}
                \quad\forall v \in H^{1+\sigma}_{\G},
            \\
            \label{eq:interp_est_h1*}
                \bE_{H^1_*}[v, I_h v^e] = | v^e-I_h v^e |_{H^1_*}^2 &\le C_1 h |v|_{H^2_\G} \quad\forall v \in H^2_\G.
            \end{align}
        \end{lemma}
        \begin{proof}
            We start with \eqref{eq:interp_est_l2*}. To estimate $\|u - I_h u^e\|_{L^2_\G}^2$, we apply Lemma \ref{prop:trace_inequality} (local trace inequality) followed by the error estimates
            \eqref{eq:scott_zhang_l2bulk_error}, \eqref{eq:scott_zhang_h1bulk_error}
            and by Lemma~\ref{prop:normal_extension_property} (global normal extension)
            \begin{equation*}
                \begin{split}
                    \|u - I_h u^e\|_{L^2_\G}^2 & \lesssim h^{-1} \|u^e - I_h u^e\|_{\Omega_h}^2 + h \|\nabla (u^e - I_h u^e)\|_{\Omega_h}^2 \lesssim h^{1+2\sigma} | u^e |_{H^{1+\sigma}(\Omega_h)}^2 \lesssim h^{2(1+\sigma)} |u|_{H^{1+\sigma}_\G}^2.
                \end{split} 
            \end{equation*}
            For the stabilization term $s_{0}(I_h u^e, I_h u^e)$, we use $s_{0}(u^e, \cdot) = 0$,
            and the fact that the normal derivative is bounded by the full gradient. Then \eqref{eq:scott_zhang_h1bulk_error} and Lemma \ref{prop:normal_extension_property} (global normal extension) yield
            \begin{equation*}
                \begin{split}
                    s_0(I_h v^e,I_h v^e) = s_{0}(v^e - I_h v^e, v^e - I_h v^e) & \leq h \|\nabla (v^e - I_h v^e)\|_{\Omega_h}^2 \lesssim h h^{2\sigma} | v^e |_{H^{1+\sigma}(\Omega_h)}^2 \lesssim h^{2(1+\sigma)} | v |_{H^{1+\sigma}_\G}^2 .
                \end{split}
            \end{equation*}
            We now turn to \eqref{eq:interp_est_h1*}. Since $v \in H^2_\G$, we apply Lemma \ref{prop:trace_inequality} (local trace inequality) to $\nablaG (v - I_h v^e)$, argue with 
            $|v-I_h v^e|_{H^1_\G}$ and we did with $\|v - I_h v^e\|_{L^2_\G}$ but this time utilizing the stability of $I_h$ in the broken $H^2(\Omega_h)$-norm \eqref{eq:scott_zhang_hkbulk_stability}. This, together with \eqref{eq:normal_extension_property}, yields $|v-I_h v^e|_{H^1_\G}^2 \lesssim h^2 |v|_{H^2_\G}^2$. Finally, the fact that $s_1(I_h v^e,I_h v^e) \lesssim h^{-2} s_0(I_h v^e,I_h v^e) \lesssim h^2 |v|_{H^2_\G}^2$ concludes the proof.
        \end{proof}
        We conclude this subsection by proving that the constant $\Cinvh$ in Definition \ref{def:Cinvh} (inverse parameter) is bounded uniformly in $h$ provided that $\{\cT_h\}_h$ satisfies Assumption \ref{assumption:shape-reg-quasi-unif} (shape-regularity and quasi-uniformity). In particular, this gives a uniform bound on $\Cinvh$ (to be called $C_{\mathrm{inv},2}$) - one of the two quantities appearing in the sufficient condition for discrete inf-sup stability \eqref{eq:intro:main_result}.
        
        \begin{lemma}[inverse estimates] \label{lemma:refined_inverse_estimates}
            Let $\{\cT_h\}_{h}$ satisfy Assumption \ref{assumption:shape-reg-quasi-unif}. For $w_h \in V_h$, we have
            \begin{equation} \label{eq:refined_inverse}
               h \|w_h\|_{H^1_*} \leq C_{\mathrm{inv},1} \|w_h\|_{L^2_*}, \quad h \|w_h\|_{L^2_*} \leq C_{\mathrm{inv},2} \|w_h\|_{V^{-1}_h}.
            \end{equation}
        \end{lemma}
\begin{proof}
            We first apply the definition \eqref{eq:stabilized_norms} of the $H^1_*$ norm
            \[
                \|w_h\|_{H^1_*}^2  = \|w_h\|_{L^2_\G}^2 + |v|_{H^1_\G}^2 + s_1(w_h, w_h)
            \]
            along with Lemma \ref{prop:trace_inequality} (local trace inequality) to both $w_h$ and $\nabla w_h$ to obtain
            \[
            \|w_h\|_{H^1_*}^2\lesssim h^{-1} \|w_h\|_{L^2(\Omega_h)}^2 + h |w_h|_{H^1(\Omega_h)}^2 + h^{-1}|w_h|_{H^1(\Omega_h)}^2 + h \|\nabla_h \nabla w_h\|_{L^2(\Omega_h)}^2 + s_1(w_h, w_h).
            \]
            Here, the Hessian $\nabla_h\nabla w_h$ should be understood in a broken sense. Since $s_1(w_h,w_h) \lesssim h^{-2} s_0(w_h, w_h)$ according to \eqref{eq:intro:stabilization}, the resulting inequality reads
            \[
                \|w_h\|_{H^1_*}^2 \lesssim h^{-1} \|w_h\|_{L^2(\Omega_h)}^2  + h^{-1}|w_h|_{H^1(\Omega_h)}^2 + h \|\nabla_h \nabla w_h\|_{L^2(\Omega_h)}^2 + h^{-2} s_0(w_h, w_h).
            \]
            Applying first Lemma \ref{prop:inverse_estimate} (inverse bulk estimate) and next Lemma \ref{prop:norm_equiv} (norm equivalence) yields
            \begin{align*}
                \|w_h\|_{H^1_*}^2 \lesssim h^{-3} \|w_h\|_{L^2(\Omega_h)}^2 + h^{-2} s_0(w_h, w_h)
                \lesssim h^{-2} \|w_h\|_{L^2_*}^2.
            \end{align*}
            This gives the first estimate in \eqref{eq:refined_inverse}. To prove the second estimate, we use the definition \eqref{eq:discrete_dual_norm} of the $V^{-1}_h$ norm
            along with the preceding estimate to arrive at
            \begin{equation*}
                \|w_h\|_{L^2_*}^2 = (w_h, w_h)_{*} \leq \|w_h\|_{V^{-1}_h} \|w_h\|_{H^1_*} \lesssim \|w_h\|_{V^{-1}_h} h^{-1} \|w_h\|_{L^2_*}.
            \end{equation*}
            where $(\cdot, \cdot)_*$ is defined in \eqref{eq:discrete_inner_prod}.
            This gives $h \|w_h\|_{L^2_*} \lesssim \|w_h\|_{V^{-1}_h}$ as desired.
        \end{proof}

%-------------------------------------------------------------------------------   
 \subsection{Stabilized \texorpdfstring{$L^2$}{L2}-projection: Basic properties} \label{subsection:stabilized_l2_projection}
%------------------------------------------------------------------------------- 
        We now study the stabilized $L^2$-projection $P_h v$ of a function $v$ defined on the surface $\G$ to the bulk discrete space $V_h$. For convenience, we repeat the definition of $P_h$ appearing earlier in \eqref{eq:intro:l2proj}: given $v \in H^{-1}_\G$, find $P_h v \in V_h$ such that
        \begin{equation}\label{eq:stabilized_l2_projection}
            (P_h v, w_h)_* = (P_h v, w_h)_\G + s_{0} (P_h u, w_h) = \langle v, w_h \rangle_\G, \quad \forall w_h \in V_h.
        \end{equation}
        Additionally, if $v\in L^2_\G$, the duality pairing on the right-hand side of \eqref{eq:stabilized_l2_projection} should be understood as the $L^2_\G$ inner product, i.e.\ $\langle v, w_h \rangle_\G = (v, w_h)_\G$ for $w_h \in V_h.$

        We emphasise that the stabilization $s_0$ has to be present in \eqref{eq:stabilized_l2_projection} for the well-posedness of the $L^2$-projection $P_h$. We also observe that strictly speaking $P_h$ is not a projection: if $v_h\in V_h$ then $P_h v_h\ne v_h$ unless $s_0(v_h,v_h)=0$. If $v \in L^2(\G)$, then the normal extension $v^e$ is well defined and $s_0(v^e,v^e)=0$. We thus obtain the {\it Galerkin orthogonality} property
        \begin{equation}\label{eq:galerkin-ortho}
        \begin{aligned}
         0 & = (P_h v - v, w_h)_\G + s_0 (P_h v, w_h) \\
         & = (P_h v - v, w_h)_\G + s_0 (P_h v - v^e, w_h) = (P_h v - v^e, w_h)_*.
         \quad\forall w_h\in V_h.
         \end{aligned}
        \end{equation}
        Finally, $P_h$ satisfies a {\it $L^2_*$-stability property}. In fact, for all $v\in L^2_\G$
        \[
        \|P_h v\|_{L^2_*}^2 = (P_hv,P_h v)_* = (v, P_h v)_\G \le \|v\|_{L^2_\G} \|P_h v\|_{L^2_\G},
        \]
        whence
        \begin{equation}\label{eq:L^2-stability}
        \|P_h v \|_{L^2_*} \le \|v\|_{L^2_\G} \quad\forall v\in L^2_\G.
        \end{equation}

        \begin{lemma}[best approximation property of $P_h$]\label{lem:best-approx-Ph}
            If $v \in L^2(\G)$ and $P_h v \in V_h$ is its stabilized $L^2$-projection given in \eqref{eq:stabilized_l2_projection}, then
            \begin{equation} \label{eq:l2proj_best_approx}
                \bE_{L^2_*}[v, P_h v] \leq \inf_{v_h \in V_h} \bE_{L^2_*}[v, v_h].
            \end{equation}
        \end{lemma}
        \begin{proof}
            In view of the orthogonality property \eqref{eq:galerkin-ortho}, applying Cauchy-Schwarz yields
            \[\begin{split}
                \bE_{L^2_*}[v, P_h v]^2 & = (v - P_h v, v - P_h v)_\G + s_0(P_h v, P_h v) \\
                & = (v - P_h v, v - v_h)_\G + s_0(P_h v, v_h) \leq \bE_{L^2_*}[v, P_h v] \, \bE_{L^2_*}[v, v_h].
            \end{split}\]
            Upon dividing both sides by $\bE_{L^2_*}[v, P_h v]$, and computing $\inf_{0 \neq v_h \in V_h}$ gives the result.
        \end{proof}
        
        \begin{lemma}[$L^2_*$ error estimates for $P_h$] \label{lemma:Ph_L2_estimates}
            Let $\{\T_h\}_h$ satisfy Assumption~\ref{assumption:shape-reg-quasi-unif}. There exists a constant $C_0$ depending only the shape regularity of $\{\T_h\}_h$ such that for $\sigma=0,1$ we have
            \begin{equation}\label{eq:stabilized_l2_projection_estimate_L2}
                \bE_{L^2_*}[v, P_h v] \le C_0 \, h^{1+\sigma} |v|_{H^{1+\sigma}_\G}.
            \end{equation}
            In addition, if $v\in L^2_\G$ then $\bE_{L^2_*}[v, P_h v] \to 0$ as $h\to 0$.
        \end{lemma}
        \begin{proof}
            We use \eqref{eq:l2proj_best_approx} in conjunction with the interpolation estimate \eqref{eq:interp_est_l2*} to get
            \begin{equation}
                \begin{split}
                    \bE_{L^2_*}[v, P_h v] 
                    \leq \bE_{L^2_*}[v, I_h v^e] 
                    \le C_0 \, h^{1+\sigma} |v|_{H^{1+\sigma}_{\G}}.
                \end{split}
            \end{equation}
            The remaining statement results from the density of $H^1_\G$ in $L^2_\G$, \eqref{eq:L^2-stability} and \eqref{eq:stabilized_l2_projection_estimate_L2} for $\sigma=0$.
        \end{proof}

        After these elementary properties of $P_h$, we next examine the stability properties of $P_h$ in $H^1_*$ and $H^{-1}_*$ along with several norms of $P_h$ and their relation. This study is instrumental to relate $P_h$ to the parabolic inf-sup property in Section \ref{section:inf-sup}, and mimics that of Tandardini and Veeser in \cite{tantardini_l2-projection_2016} for standard semidiscrete FEM for parabolic problems in flat domains.

%------------------------------------------------------------------------------
        \subsection{\texorpdfstring{$H^1_*$}{H1*}-stability of \texorpdfstring{$P_h$}{Ph}}
%------------------------------------------------------------------------------ 

            Stability in $H^1$ of the $L^2$-projection into finite element spaces is an outstanding issue for flat domains. The proof for quasi-uniform meshes is elementary and hinges on the inverse inequality. However, for graded meshes there are several results with restrictions on mesh grading 
            \cite{crouzeix_stability_1987, carstensen_merging_2001, karkulik_2d_2013, bank_h1_2014, gaspoz_optimal_2016, diening_sobolev_2021};
            in fact, mesh grading cannot be arbitrary 
            \cite{bank_h1_2014}. The most practical result to date in two dimensions establishes stability of the $L^2$-projection on meshes generated by newest vertex bisection 
            \cite{gaspoz_optimal_2016, diening_sobolev_2021}.

            The situation for TraceFEM is completely open because the stabilized $L^2$-projection is a new concept. Below we prove $H^1_*$-stability of $P_h$ by mimicking the corresponding proof for flat domains provided Assumption~\ref{assumption:shape-reg-quasi-unif} (shape regularity and quasi-uniformity) is valid. The analysis of graded meshes is quite intricate and we do not explore it.
            
            \begin{proposition}[$H^1_*$-stability of $P_h$] \label{theorem:h1*_stability_of_stab_l2_proj}
                Let $\{\T_h\}_h$ satisfy Assumption~\ref{assumption:shape-reg-quasi-unif}. There exists a constant $\Cone\ge1$, only depending on shape regularity and quasi-uniformity of $\{\T_h\}_h$, such that
                \begin{equation}\label{eq:stabilized_l2_projection_h1*_stability}
                    \|P_h v\|_{H^1_*} \le \Cone \|v\|_{H^1_\G} \quad \forall v\in H^1_\G.
                \end{equation}
            \end{proposition}
            
            \begin{proof} Applying Lemma \ref{lemma:refined_inverse_estimates} (inverse estimates) and Lemma \ref{lemma:interpolation} (interpolation error estimates) yields
                \[
                    |P_h u|_{H^1_*} \leq |P_h u - I_h u^e|_{H^1_*} + |I_h u^e|_{H^1_*} \le C_{\mathrm{inv},1} h^{-1} \|P_h u - I_h u^e\|_{L^2_*} + C_1 |u|_{H^1_\G}.
                \]
             We next recall \eqref{eq:interp_est_l2*} and \eqref{eq:stabilized_l2_projection_estimate_L2} to derive 
             \[
               \|P_h u - I_h u^e\|_{L^2_*} \le \bE_{L^2_*}[u, P_h u] + \bE_{L^2_*}[u, I_h u^e]
               \le 2C_0 h |u|_{H^1_\G}.
             \]
             We finally use \eqref{eq:L^2-stability} to take care of the $L^2$ part of the norm, and
             realize that the constant $\Cone\ge1$ depends on $C_{\mathrm{inv},1}, C_1$ and $C_0$ as asserted.
            \end{proof}

            We now endow the operator $P_h:H^1_\G \to V_h$ with two distinct norms depending on the norms used in the target space $V_h$. According to \eqref{eq:intro:operator_norms}, we recall that
            \begin{equation} \label{eq:operator-norm-H1}
            \|P_h\|_{\cL(H^1_\G)} = \sup_{v\in H^1_\G} \frac{\|P_h v\|_{H^1_\G}}{\|v\|_{H^1_\G}},
            \quad
            \|P_h\|_{\cL(H^1_*)} = \sup_{v\in H^1_\G} \frac{\|P_h v\|_{H^1_*}}{\|v\|_{H^1_\G}}.
            \end{equation}
            It is clear that
            \begin{equation}\label{eq:Ph-H1G}
                \|P_h\|_{\cL(H^1_\G)} \le \|P_h\|_{\cL(H^1_*)},
            \end{equation}
            and that Proposition \ref{theorem:h1*_stability_of_stab_l2_proj} can be reformulated as follows:
            \begin{equation}\label{eq:Ph-H1*}
                \|P_h\|_{\cL(H^1_*)} \le \Cone.
            \end{equation}

            \begin{lemma}[$H^1$-norms of $P_h$]\label{L:norms-Ph}
                Let $\{\cT_h\}_h$ satisfy Assumption~\ref{assumption:shape-reg-quasi-unif}. If $C_0$ is the constant in \eqref{eq:interp_est_l2*}, which only depends on shape
                regularity of $\{\T_h\}_h$, then
                \begin{equation}\label{eq:norms-Ph}
                  \|P_h\|_{\cL(H^1_\G)} \le \|P_h\|_{\cL(H^1_*)} \le \|P_h\|_{\cL(H^1_\G)} + C_0.
                \end{equation}
            \end{lemma}
            \begin{proof}
               We have $\|P_h v\|_{H^1_\G} \le \|P_h\|_{\cL(H^1_\G)} \|v\|_{H^1_\G}$ and 
               \[
               s_1(P_h v,P_h v) = h^{-2} s_0(P_h v,P_h v) \le h^{-2} \bE_{L^2_*} [v,P_h v]^2
               \le C_0^2 |v|_{H^1_\G}^2.
               \]
               This yields $\|P_h v\|_{H^1_*}^2 = \|P_h v\|_{H^1_\G}^2 + s_1(P_h v,P_h v) \le \big(\|P_h\|_{\cL(H^1_\G)}^2 + C_0^2\big) \|v|_{H^1_\G}^2$, hence \eqref{eq:norms-Ph}.
            \end{proof}
            
%----------------------------------------------------------------------------             
        \subsection{\texorpdfstring{$H^{-1}_*$}{H-1*}-stability of \texorpdfstring{$P_h$}{Ph}}
 %----------------------------------------------------------------------------  
            The following is the dual counterpart of Proposition \ref{theorem:h1*_stability_of_stab_l2_proj}.
            
            \begin{lemma}[$H^{-1}_*$-stability of $P_h$]\label{lemma:h'_stab_of_stab_l2_proj}
                Let $\{\cT_h\}_h$ satisfy Assumption~\ref{assumption:shape-reg-quasi-unif}. For $\ell \in H^{-1}_\G$, we have
                \begin{equation} \label{eq:h'_stab_of_stab_l2_proj}
                    \|P_h \ell\|_{H^{-1}_*} \le \Cneg \|\ell\|_{H^{-1}_\G}
                \end{equation}
                where $\Cneg=\Cone + C_{\mathrm{inv},1}$ and $\Cone, C_{\mathrm{inv},1}$ are the constants in Proposition \ref{theorem:h1*_stability_of_stab_l2_proj} and Lemma \ref{lemma:refined_inverse_estimates} (inverse estimates).
               \begin{proof}
                   Given $v \in H^1_\G$, we use the symmetry relation \eqref{eq:l2_proj_symmetry} followed by \eqref{eq:stabilized_l2_projection_h1*_stability} to arrive at
                   \[
                     \big| \langle P_h\ell,v \rangle_\G\big| 
                     = \big| \langle \ell, P_h v \rangle_\G\big|
                     \le \|\ell\|_{H^{-1}_\G} \|P_h v\|_{H^1_\G} 
                     \le \Cone \|\ell\|_{H^{-1}_\G} \|v\|_{H^1_\G},
                   \]
                   whence
                   \[
                     \|P_h \ell\|_{H^{-1}_\G} 
                     = \sup_{v\in H^1_\G} \frac{\big| \langle P_h\ell,v \rangle_\G\big|}{\|v\|_{H^1_\G}}
                     \le \Cone \| \ell \|_{H^{-1}_\G}.
                   \]
                   On the other hand, using the definition \eqref{eq:stabilized_l2_projection} of $P_h\ell$ with $w_h=P_h\ell$ in conjunction with Lemma \ref{lemma:refined_inverse_estimates} (inverse estimates) yields
                   \[
                     \| P_h\ell\|_{L^2_*}^2 = \langle \ell, P_h\ell\rangle_\G \le
                     \|\ell\|_{H^{-1}_\G} \|P_h \ell\|_{H^1_\G} 
                     \le \frac{C_{\textrm{inv},1}}{h}{\|\ell\|_{H^{-1}_\G} \|P_h \ell\|_{L^2_*}},
                   \]
                   whence
                   \[
                      h \| P_h\ell \|_{L^2_*} \le C_{\textrm{inv},1} \|\ell\|_{H^{-1}_\G}.
                   \]
                   Consequently,
                   \[
                     s_{-1}(P_h\ell,P_h\ell) = h^2 s_0(P_h\ell,P_h\ell) 
                     \le h^2 \|P_h\ell\|_{L^2_*}^2 \le C_{\mathrm{inv},1}^2 \|\ell\|_{H^{-1}_\G}^2.
                   \]
                   Employing the definition of $\|P_h \ell\|_{H^{-1}_*}$ concludes the proof.
                \end{proof}
            \end{lemma}

            In addition to the definition \eqref{eq:Ph-op-norm-h-1} of $\|P_h\|_{\cL(H^{-1}_\G)}$, we can also endow $P_h: H^{-1}_* \to V_h$ with an operator norm that uses a stronger norm in the target space $V_h$, namely
            \begin{equation}
                \|P_h\|_{\cL(H^{-1}_*)} := \sup_{\ell\in H^{-1}_\G}
                     \frac{\|P_h\ell\|_{H^{-1}_*}}{\|\ell\|_{H^{-1}_\G}},
            \end{equation}
            where, according to \eqref{eq:stabilized_norms},
            \[
              \|v_h\|_{H^{-1}_*}^2 = \|v_h\|_{H^{-1}_\G}^2 + s_{-1}(v_h,v_h)
              \quad\forall v_h \in V_h.
            \]
            \begin{lemma}[$H^{-1}$ -norms of $P_h$]\label{L:norm-Ph-1}
              Let $\{\cT_h\}_h$ satisfy Assumption~\ref{assumption:shape-reg-quasi-unif}. We have
               \begin{equation}\label{eq:norm-Ph-1}
                \|P_h\|_{\cL(H^{-1}_\G)} \le
                \|P_h\|_{\cL(H^{-1}_*)} \le \|P_h\|_{\cL(H^{-1}_\G)} + C_{\mathrm{inv},1}.
            \end{equation}
            \end{lemma} 
            \begin{proof}
            The leftmost inequality is obvious. The proof of Lemma \ref{lemma:h'_stab_of_stab_l2_proj} shows that
            $
              s_{-1} (P_h\ell,P_h\ell) \le C_{\mathrm{inv},1}^2 \|\ell\|_{H^{-1}_\G}^2.
            $
            This, in conjunction with $\|P_h\ell\|_{H^{-1}_\G} \le \|P_h\|_{\cL(H^{-1}_\G)} \|\ell\|_{H^{-1}_\G}$, implies the rightmost inequality on \eqref{eq:norm-Ph-1}.
            \end{proof}
    %---------------------------------------------
    \subsection{Error estimates in \texorpdfstring{$H^1_*$}{H1*} and \texorpdfstring{$H^{-1}_*$}{H-1*}} \label{sec:error-est-L2}
    %---------------------------------------------

        We conclude with error estimates for the stabilized $L^2$-projection $P_h$ in $H^1_*$ and $H^{-1}_*$ that complement Lemma \ref{lemma:Ph_L2_estimates} ($L^2_*$ error estimates for $P_h$).
        
        \begin{lemma}[quasi-best approximation in $H^1_*$]\label{lem:quasi-best-H1*}
            If Proposition \ref{theorem:h1*_stability_of_stab_l2_proj} is valid, then there exists a
            constant $C_2$ depending on $C_{\mathrm{inv},1}$ in \eqref{eq:refined_inverse} and $\Cone$ in \eqref{eq:stabilized_l2_projection_h1*_stability} such that
            \begin{equation}\label{eq:quasi-best-H1*}
              \bE_{H^1_*} [u,P_hu] \le C_2 \inf_{v_h\in V_h} \bE_{H^1_*} [u,v_h].
            \end{equation}
            Moreover, if $u\in H^1_\G$ then $\bE_{H^1_*}[u, P_h u]\to0$ as $h\to0$.
        \end{lemma}
        \begin{proof}
            Given $v_h\in V_h$ note that $P_h v_h \ne v_h$ but the difference of these two functions satisfies
            \[
              (v_h-P_hv_h, w_h)_\G + s_0(v_h-P_h v_h, w_h) = s_0(v_h,w_h)
              \quad\Rightarrow\quad 
              \|v_h-P_hv_h\|^2_{L^2_*} \le s_0(v_h,v_h) = h^2 s_1(v_h,v_h),
            \]
            where we chose $w_h = v_h - P_h v_h$, used Cauchy-Schwarz, and the fact that $s_0(w_h,w_h)\le \|w_h\|_{L^2_*}^2$ for any $w_h\in V_h$. We now recall the
            definition $\bE_{H^1_*} [u,P_hu]^2 = \|u-P_h u\|_{H^1_\G}^2 + s_1(P_h u,P_h u)$, write
            \[
              u-P_h u = (u-v_h) + (v_h-P_h v_h) + P_h (v_h-u),
              \quad
              P_h u = v_h + (P_h v_h - v_h ) + P_h (u-v_h),
            \]
            and apply the triangle inequality to obtain
            \begin{equation} \label{eq:Ph-quasi-best-H1*-triangle}
                \bE_{H^1_*} [u,P_hu] \le \bE_{H^1_*} [u,v_h] + \|v_h-P_h v_h\|_{H^1_*}
             + \| P_h (u-v_h) \|_{H^1_*}
            \end{equation}
            By virtue of Proposition \ref{theorem:h1*_stability_of_stab_l2_proj} ($H^1_*$ stability of $P_h$) we deduce 
            \[
               \| P_h (u-v_h) \|_{H^1_*}\le \Cone 
               \|u-v_h\|_{H^1_\G}\le \Cone \bE_{H^1_*} [u,v_h]
            \]
            On the other hand, we apply Lemma \ref{lemma:refined_inverse_estimates} (refined inverse estimate) to the second term of \eqref{eq:Ph-quasi-best-H1*-triangle}
            \[
              \|v_h- P_h v_h\|_{H^1_*} \le C_{\mathrm{inv,1}} h^{-1} \|v_h- P_h v_h\|_{L^2_*} 
              \le C_{\mathrm{inv,1}} \sqrt{s_1(v_h,v_h)} \le C_{\mathrm{inv,1}} \bE_{H^1_*} [u,v_h].
            \]
            This proves \eqref{eq:quasi-best-H1*} with $C_2=1+\Cone+ C_{\mathrm{inv,1}}$. The remaining statement follows from the density of $H^2_\G$ in $H^1_\G$, \eqref{eq:quasi-best-H1*} and \eqref{eq:interp_est_h1*}. In fact, let $u_\varepsilon\in H^2_\G$ be so that $\|u-u_\varepsilon\|_{H^1_\G} \le \varepsilon$
            and notice that
            \[
              \bE_{H^1_*} [u,P_hu] \le C_2 \bE_{H^1_*} [u, I_h u_\varepsilon]
              \le C_2 \|u-u_\varepsilon\|_{H^1_\G} +
              C_2 \bE_{H^1_*} [u_\varepsilon, I_h u_\varepsilon] \le C_2\varepsilon + C_1 C_2 h |u_\epsilon|_{H^2_\G} \le 2 C_2 \epsilon
            \]
            provided $C_1 h |u_\epsilon|_{H^2_\G} \le 1$ for $h$ sufficiently small. Since $\varepsilon$ is arbitrary the assertion follows.
        \end{proof}
        We point out that property \eqref{eq:quasi-best-H1*} would be trivial if $P_h$ were a projection into $V_h$, and so invariant on $V_h$, but this is not the case. Moreover, the key inequality
        $\|v_h-P_hv_h\|^2_{L^2_*} \le s_0(v_h,v_h)$ for any $v_h\in V_h$ quantifies the deviation from 
        $v_h - P_h v_h = 0$ in terms of the stabilization term $s_0(v_h,v_h)$. 

        \begin{corollary}[$H^1_*$ error estimate for $P_h$] \label{cor:Ph-h1-error-est}
            Let $\{\cT_h\}_h$ satisfy Assumption~\ref{assumption:shape-reg-quasi-unif}. There exists a constant $C_1$ depending only the shape regularity of $\{\T_h\}_h$ such that for $\sigma = 0, 1$ we have
            \begin{equation}\label{eq:stabilized_l2_projection_estimate_h1}
                 \bE_{H^1_*}[u, P_h u] \le C_1 h^\sigma |u|_{H^{1+\sigma}_\G}.
            \end{equation}
        \end{corollary}
        \begin{proof}
            Concatenate Lemmas \ref{lem:quasi-best-H1*} (quasi-best approximation in $H^1_*$) and \ref{lemma:interpolation} (interpolation error estimates).
        \end{proof}

        We complement Lemma \ref{lem:quasi-best-H1*} with the analogous result in $H^{-1}_*$.
        \begin{lemma}[quasi-best approximation in $H^{-1}_*$] \label{lem:quasi-best-H-1*}
            If Proposition \ref{theorem:h1*_stability_of_stab_l2_proj}, Lemma \ref{lemma:h'_stab_of_stab_l2_proj}, and Lemma \ref{lemma:refined_inverse_estimates} are valid, then there exists a constant $C_3$ depending on $C_{\mathrm{inv},2}$ in \eqref{eq:refined_inverse}, $\Cone$ in \eqref{eq:stabilized_l2_projection_h1*_stability}, and $C_{\mathrm{stab,-1}}$ in \eqref{eq:h'_stab_of_stab_l2_proj} such that for every $\ell \in H^{-1}_\G$
            \begin{equation} \label{eq:quasi-best-H-1*}
                \bE_{H^{-1}_*}[\ell, P_h \ell] \leq C_3 \inf_{v_h \in V_h} \bE_{H^{-1}_*}[\ell, v_h].
            \end{equation}
        \end{lemma}
        \begin{proof}
            As in the proof of Lemma \ref{lem:quasi-best-H1*} (quasi-best approximation in $H^1_*$), we let $v_h \in V_h$ and split the error $\ell - P_h \ell$ as follows:
            \[\ell - P_h \ell = (\ell - v_h) + (v_h - P_h v_h) + P_h (v_h - \ell),
              \quad
              P_h \ell = v_h + (P_h v_h - v_h ) + P_h (\ell - v_h).\]
            Triangle inequality readily yields
            \begin{equation*} 
                \bE_{H^{-1}_*}[\ell, P_h \ell] \leq \bE_{H^{-1}_*}[\ell, v_h] + \|v_h - P_h v_h\|_{H^{-1}_*} + \|P_h(\ell - v_h)\|_{H^{-1}_*}.
            \end{equation*}
            Lemma \ref{lemma:h'_stab_of_stab_l2_proj} ($H^{-1}_*$-stability of $P_h$) allows us to control the last term as follows:
            \[\|P_h(\ell - v_h)\|_{H^{-1}_*} \leq C_{\mathrm{stab,-1}} \|\ell - v_h\|_{H^{-1}_\G} \leq C_{\mathrm{stab,-1}} \bE_{H^{-1}_*}[\ell, v_h].\]
            Since the middle term $\|v_h - P_h v_h\|_{H^{-1}_*}$ is discrete, estimating it in the discrete dual norm $V_h^{-1}$ suffices according to Lemma \ref{lemma:relationship_bw_the_dual_norms} (relation between the dual norms). To that end, we use the definitions of $\|\cdot\|_{V_h^{-1}}$ and $P_h$, and the duality of scaling $s_0(\cdot, \cdot) \leq s_{-1}(\cdot, \cdot)^{1/2} s_1(\cdot, \cdot)^{1/2}$:
            \[\begin{split}
                \|v_h - P_h v_h\|_{V_h^{-1}} & = \sup_{w_h \in V_h} \frac{|(v_h - P_h v_h, w_h)_*|}{\|w_h\|_{H^1_*}} = \sup_{w_h \in V_h} \frac{|(v_h, w_h)_* - (v_h, w_h)_\G|}{\|w_h\|_{H^1_*}} \\
                & = \sup_{w_h \in V_h} \frac{|s_0(v_h, w_h)|}{\|w_h\|_{H^1_*}} \leq s_{-1}(v_h, v_h)^{1/2} \leq \bE_{H^{-1}_*}[\ell, v_h].
            \end{split}\]
            Lemmas \ref{lemma:relationship_bw_the_dual_norms} and \ref{lemma:refined_inverse_estimates} (inverse estimates) together with Proposition \ref{theorem:h1*_stability_of_stab_l2_proj} ($H^1_*$-stability of $P_h$) yield
            \[\|v_h - P_h v_h\|_{H^{-1}_*} \leq (C_{\mathrm{stab,1}} + C_{\mathrm{inv,2}}) \bE_{H^{-1}_*}[\ell, v_h].\]
            We now combine the above estimates to arrive at
            \[\bE_{H^{-1}_*}[\ell, P_h \ell] \leq (1 + C_{\mathrm{stab,1}} + C_{\mathrm{inv,2}} + C_{\mathrm{stab,-1}}) \bE_{H^{-1}_*}[\ell, v_h].\]
            Taking the infimum over $0 \neq v_h \in V_h$ completes the proof of assertion \eqref{eq:quasi-best-H-1*}.
        \end{proof}
        Even though the previous lemma establishes the quasi-best approximation in $H^{-1}_*$ for $P_h$, it is not immediately obvious that optimal error estimates follow. For instance, using the Scott-Zhang operator of a constant normal extension does not make sense on $H^{-1}_\G$. It turns out that we can still get error estimates for $P_h$ in $H^{-1}_*$ by duality. The following lemma will be essential for quantifying the error $\pa_t u - \pa_t u_h$ in Section \ref{section:error_est}.
        \begin{lemma}[$H^{-1}_*$ error estimate for $P_h$]\label{lemma:stabilized_l2_projection_h'_error_est}
            Let $\{\cT_h\}_h$ satisfy Assumption~\ref{assumption:shape-reg-quasi-unif}.  For $u \in H^{1+\sigma}_\G, \sigma \in \{-1,0,1\}$, we have
            \begin{equation} \label{eq:stabilized_l2_projection_h'_error_est}
                \bE_{H^{-1}_*}[u, P_h u] \leq C_{-1} h^{2+\sigma} |u|_{H^{1+\sigma}_\G}.
            \end{equation}
            Moreover, if $u \in H^{-1}_\G$ then $\bE_{H^{-1}_*}[u, P_h u] \to 0$ as $h \to 0$.
        \end{lemma}
        \begin{proof}
            We proceed by duality. We intend to estimate $(u - P_h u, \phi)_\G$ for $\phi \in H^1_\G$. From \eqref{eq:stabilized_l2_projection} (stabilized $L^2$-projection), we have for all $ w_h \in V_h$
            \[
              (P_h u, w_h)_\G + s_0(P_h u, w_h) = (u, w_h)_\G \implies (u - P_h u, w_h)_\G - s_0(P_h u, w_h) = 0.
              \]
            Picking $w_h = I_h \phi^e$, we obtain
            \[(u - P_h u, I_h \phi^e)_\G - s_0(P_h u, I_h \phi^e) = 0,\]
            and subtracting this from $(u - P_h u, \phi)_\G$ we arrive at
            \[
                (u - P_h u, \phi)_\G 
                = (u - P_h u, \phi - I_h \phi^e)_\G + s_0(P_h u, I_h \phi^e).
            \]
            Next, we make use of Cauchy-Schwarz twice to get
            \[\begin{split}
                |(u - P_h u, \phi)_\G| & \leq \|u - P_h u\|_{L^2_\G} \|\phi - I_h \phi^e\|_{L^2_\G} \\
                & + s_{0}(P_h u, P_h u)^{1/2} s_{0}(I_h \phi^e, I_h \phi^e)^{1/2}
                \leq \bE_{L^2_*}[u, P_h u] \bE_{L^2_*}[\phi, I_h \phi^e].
            \end{split}\]
            Applying Lemma \ref{lemma:interpolation} (interpolation error estimates) on the second term, we obtain
            \[|(u - P_h u, \phi)_\G| \leq C_0 h \bE_{L^2_*}[u, P_h u] \|\phi\|_{H^1_\G}.\]
            Dividing both sides by $\|\phi\|_{H^1_\G}$ and taking a supremum over $\phi \in H^1_\G$ yields
            \[\|u - P_h u\|_{H^{-1}_\G} \leq C_0 h \bE_{L^2_*}[u, P_h u].\]
            To bound the stabilization term $s_{-1}(P_hu,P_hu)$ note that
            \[s_{-1}(P_h u, P_h u) \le h^2 s_0(P_h u, P_h u) \leq h^2 \bE_{L^2_*}[u, P_h u]^2.\]
            Altogether,
            \[\bE_{H^{-1}_*}[u, P_h u] \leq (1 + C_0^2)^{1/2} h \bE_{L^2_*}[u, P_h u].\]
            Concatenating the above with Lemma \ref{lemma:Ph_L2_estimates} ($L^2_*$ error estimates for $P_h$) gives assertion \eqref{eq:stabilized_l2_projection_h'_error_est} with $C_{-1} = C_0 (1 + C_0^2)^{1/2}$. For the remaining assertion, we use the density of $L^2_\G$ in $H^{-1}_\G$, \eqref{eq:h'_stab_of_stab_l2_proj} and \eqref{eq:stabilized_l2_projection_h'_error_est} as follows. Let $u_\varepsilon \in L^2_\G$ be such that $\|u - u_\varepsilon\|_{H^{-1}_\G} \leq \varepsilon$ and use triangle inequality to get
            \[
            \bE_{H^{-1}_*}[u, P_h u] \leq \|u - u_\varepsilon\|_{H^{-1}_\G} + 
            \|P_h(u - u_\varepsilon)\|_{H^{-1}_*} + \bE_{H^{-1}_*}[u_\varepsilon, P_h u_\varepsilon] \lesssim \varepsilon + h \|u_\varepsilon\|_{L^2_\G} \lesssim \varepsilon,
            \]
            provided $h \|u_\varepsilon\|_{L^2_\G} \leq \varepsilon$ for $h$ sufficiently small. Since $\varepsilon$ is arbitrary, the assertion follows.
        \end{proof}

%%%%%%%%%%%%%%%%%%%%%%%%%%%%%%%%%%%%%%%%%%%%%%%%%%%%%%%%%%%%%%%%%%%%%%%%%%%%%%%
\section{Optimal error estimates} \label{section:error_est}
%%%%%%%%%%%%%%%%%%%%%%%%%%%%%%%%%%%%%%%%%%%%%%%%%%%%%%%%%%%%%%%%%%%%%%%%%%%%%%%
    Throughout this section we let Assumption \ref{assumption:shape-reg-quasi-unif} (shape-regularity and quasi-uniformity) hold. We derive optimal order-regularity error estimates, first in the energy norm and later in $L^2_t L^2_*$

%------------------------------------------------------------------------------    
    \subsection{Estimate of the parabolic error functional} \label{section:error_est:l2h1}
%------------------------------------------------------------------------------
    An immediate consequence of the parabolic quasi-best approximation and error estimates for $P_h$ is the following {\it optimal order-regularity} error estimate which is novel in the literature of TraceFEM.
    \begin{corollary}[optimal parabolic error estimate] \label{corollary:error_estimate}
       Let $\{\cT_h\}_h$ satisfy Assumption~\ref{assumption:shape-reg-quasi-unif}. Let $u \in \mathbb{U}$ and $u_h \in \mathbb{U}_h$ be the solutions of \eqref{eq:cont_problem} and
        \eqref{eq:discrete_problem}. Then there exist constants $C_4$ and $C_5$ depending on shape regularity
        of $\{\T_h\}_h$ and continuity and inf-sup constants $C_b,c_*$ of $b_h$, such that
            \begin{equation}\label{eq:optimal-error-est}
                \bE[u, u_h]^2 \le C_4 h^2 \left(|u(0)|_{H^1_\G}^2 + \int_I\|\pa_t u\|_{L^2_\G}^2 + |u|_{H^2_\G}^2
                \right) \le C_5 h^2 \left( |u_0|_{H^1_\G}^2 + \int_I \|f\|_{L^2_\G}^2 \right).
            \end{equation}
        \end{corollary}
        \begin{proof}
            We take $z_h(t) = P_h u(t) \in \mathbb{U}_h$ in the symmetric bound \eqref{eq:quasi-best-approx}
            and invoke Lemma \ref{cor:Ph-h1-error-est} ($H^1_*$ error estimates for $P_h$) and Lemma \ref{lemma:stabilized_l2_projection_h'_error_est} ($H^{-1}_*$ error estimate for $P_h$) to obtain
            \begin{equation*}
                \bE_{L^2_*}[u_0, P_h u_0]^2 \lesssim h^2 |u_0|_{H^1_\G}^2, \quad \int_I \bE_{H^1_*}[u, P_h u]^2 \lesssim h^2 \int_I |u|_{H^2_\G}^2, \quad \int_I \bE_{H^{-1}_*}[\pa_t u, P_h \pa_t u]^2 \lesssim h^2 \int_I \|\pa_t u\|_{L^2_\G}^2. 
            \end{equation*}
            This proves the first estimate in \eqref{eq:optimal-error-est}. The second one is a consequence of the a priori bounds
            \[
               |u(T)|_{H^1_\G}^2 + \int_I \|\partial_t u\|_{L^2_\G}^2 \le |u_0|_{H^1_\G}^2 
               + \int_I \|f\|_{L^2_\G}^2,
               \quad
               \int_I \|\Delta_\G u\|_{L^2_\G}^2 \le 2 \int_I \|\partial_t u\|_{L^2_\G}^2 + \|f\|_{L^2_\G}^2.
            \]
            for strong solutions $u\in\mathbb{U}$ of \eqref{eq:cont_problem}.
            Hence $\int_I \|u\|_{H^2_\G}^2 \lesssim  |u_0|_{H^1_\G}^2 + \int_I \|f\|_{L^2_\G}^2,$ because of
            $H^2$-regularity theory for the Laplace-Beltami operator on $C^2$ surfaces 
            \cite[Lemma 3]{bonito_finite_2020}.
        \end{proof}

%--------------------------------------------------------------------------------------    
    \subsection{Error estimate in \texorpdfstring{$L^2_t L^2_*$}{L2t L2G}} \label{section:error_est:l2l2}
%--------------------------------------------------------------------------------------
        The goal of this subsection is to establish an error estimate in $L^2_t L^2_*$ which requires only {\it minimal regularity} $u\in\mathbb{U}$. To that end, we shall use a duality argument similar to that of Chrysafinos and Hou \cite{k_chrysafinos_error_2003}. Note that our result does not rely on the discrete inf-sup stability analysis of Section \ref{section:inf-sup} 
        and is of independent interest. 

        We start with some definitions and notation that are necessary to state three technical lemmas.
        We employ these lemmas to prove Theorem \ref{theorem:l2l2_error_estimate} (a priori error estimate in $L^2_t L^2_*$), and present the proof of these lemmas right after that of Theorem \ref{theorem:l2l2_error_estimate}.
        Let $L^2_{\G,0}$ denote the space of $L^2_\G$ functions with vanishing mean on $\G$. We also define $H^1_{\G,0} := H^1_\G \cap L^2_{\G,0}$, and $V_{h,0} \subset V_h$ such that $V_{h,0}|_\G \subset H^1_\G \cap L^2_{\G,0}$. Note that $V_{h,0}$ is a purely theoretical tool which does not enter our algorithm.  
        
        Now, consider the following auxiliary elliptic problems.
        Given $\chi \in L^2_{\G,0}$, find $\phi \in H^1_{\G,0}$ such that
        \begin{equation}\label{eq:elliptic}
            a(\phi, \psi) = (\chi, \psi)_\G \quad \forall \psi \in H^1_{\G,0}.
        \end{equation}
        This is the standard adjoint equation $-\Delta_\G\phi=\chi$ that will be used with $\chi=u-u_h$ in the duality argument. The discrete counterpart of \eqref{eq:elliptic} is a bit unusual due to the presence of stabilization terms. It is convenient to define $\phi_h \in V_{h,0}$ satisfying
        \[
           a_*(\phi_h, \psi_h) = (u, \psi_h)_\G - (u_h,\psi_h)_* = (u-u_h,\psi_h)_\G - s_0(u_h,\psi_h).
        \]
        Therefore, the discrete version of \eqref{eq:elliptic} reads: given $\chi \in L^2_{\G,0}$ and $g_h \in V_h$, find $\phi_h \in V_{h,0}$ such that
        \begin{equation}\label{eq:discrete_elliptic_perturbed}
            a_*(\phi_h, \psi_h) = (\chi, \psi_h)_\G - s_0(g_h, \psi_h) \quad \forall \psi_h \in V_{h,0}.
        \end{equation}
        Note that the last term satisfies the compatibility condition $s_0(g_h, 1) = 0$, whence there is 
        a unique solution of \eqref{eq:discrete_elliptic_perturbed}.
        Taking $\psi = \psi_h|_\G \in H^1_{\G,0}$ in \eqref{eq:elliptic} and subtracting from \eqref{eq:discrete_elliptic_perturbed} readily yields the following Galerkin quasi-orthogonality property, that uses $s_0(\phi^e, \cdot) = 0$ for the normal extension $\phi^e$:
        \begin{equation} \label{eq:elliptic_quasi_orth}
             a_*(\phi_h - \phi^e, \psi_h) + s_0 (g_h, \psi_h) = 0 \quad \forall \psi_h \in V_{h,0}.   
        \end{equation}
        Lastly, we shall assume that $\G$ is of class $C^2$ and invoke elliptic $H^2$-regularity \cite[Lemma 3]{bonito_finite_2020}
        \begin{equation}\label{eq:H2-regularity}
            \|\phi\|_{H^2_\G} \le C_{\mathrm{reg}} \|\chi\|_{L^2_\G}.
        \end{equation}

        \begin{lemma}[vanishing mean] \label{lemma:error_has_vanishing_mean}
            If $u\in\mathbb{U}$ solve \eqref{eq:cont_problem} and $u_h\in\mathbb{U}_h$ solve
            \eqref{eq:discrete_problem}, then the error $u-u_h$ has vanishing mean over $\G$ for all 
            $t \in I$, i.e. $(u - u_h|_\G) \in C_t^0 L^2_{\G,0}$.
        \end{lemma}
        
        \begin{lemma}[error estimate for the dual problem] \label{lemma:error_est_perturbed_semidiscrete_elliptic}
            Let $\{\T_h\}_h$ satisfy Assumption \ref{assumption:shape-reg-quasi-unif}. Let $\chi \in L^2_{\G,0}$ and  $g_h \in V_h$. If $\phi \in H^1_{\G,0}$ solves \eqref{eq:elliptic} and $\phi_h \in V_{h,0}$ solves \eqref{eq:discrete_elliptic_perturbed},
            then there exists a constant $C_6$ such that
            \begin{equation}\label{eq:error_estimate_perturbed_elliptic_problem}
                \bE_{H^1_*}[\phi, \phi_h]^2 \le C_6 \, h^2 \big(\|\chi\|_{L^2_\G}^2 + s_0(g_h, g_h) \big).
            \end{equation}
        \end{lemma}
        
        \begin{lemma}[identity for time derivative] \label{lemma:technical_identity_l2l2est}
            Let $\chi \in \mathbb{U} \cap C_t^0 L^2_{\G,0}$ and $g_h \in \mU_h$. If $\phi_h(t) \in V_{h,0}$ solves \eqref{eq:discrete_elliptic_perturbed} for all $t \in I$, then the following 
            identity holds for a.e. $t \in I$
            \begin{equation}\label{eq:tech_est_duality_arg}
                 \frac{1}{2} \frac{d}{dt} a_*(\phi_h, \phi_h) = \langle \pa_t \chi, \phi_h \rangle_\G - s_0(\pa_t g_h, \phi_h).
            \end{equation}
        \end{lemma}
 
        \begin{theorem}[optimal error estimate in $L^2_t L^2_*$] \label{theorem:l2l2_error_estimate}
            Let $\{\T_h\}_h$ satisfy Assumption~\ref{assumption:shape-reg-quasi-unif}. Suppose that $u$ solves \eqref{eq:cont_problem} and $u_h$ solves \eqref{eq:discrete_problem}. There exists a constant $C_7$ such that the following error estimate holds
            \begin{equation}\label{eq:a_priori_error_est_l2l2}
                \int_I \bE_{L^2_*}[u, u_h]^2 \le C_7
                h^2 \int_I \bE_{H^1_*}[u, u_h]^2.
            \end{equation}
        \end{theorem}
        \begin{proof}
            Let $\epsilon := u - u_h|_\G \in\mathbb{U}$ be the error. Lemma \ref{lemma:error_has_vanishing_mean} (vanishing mean) implies that 
            $\epsilon \in  C_t^0 L^2_{\G, 0}$. Let $\phi(t)\in H^1_{\G,0}$ be the solution of \eqref{eq:elliptic} with $\chi = \epsilon(t)$ for all $t \in I$
            \begin{equation} \label{eq:elliptic_with_error}
                a(\phi(t), \psi) = (\epsilon(t), \psi)_\G \quad \forall \psi \in H^1_{\G, 0}.
            \end{equation}
            Likewise, let $\phi_h(t)\in V_{h,0}$ be the solution of \eqref{eq:discrete_elliptic_perturbed} with $\chi = \epsilon(t)$ and $g_h = u_h(t)$ for all $t \in I$
            \begin{equation} \label{eq:discrete_elliptic_with_error}
                a_*(\phi_h, \psi_h) = (\epsilon(t), \psi_h)_\G - s_0(u_h(t), \psi_h) \quad \forall \psi_h \in V_{h,0}.
            \end{equation}
            We first test \eqref{eq:elliptic_with_error} with $\psi = \epsilon(t)\in H^1_{\G,0}$, and rearrange the right-hand side to get
            \[
            \begin{split}
                \|\epsilon\|_{L^2_\G}^2 & = a(\phi, \epsilon) = a(\phi - \phi_h, \epsilon) + a(\phi_h, \epsilon).
            \end{split}
            \]
            Next, we integrate both sides over $I$, and resort to parabolic Galerkin quasi-orthogonality \eqref{eq:parabolic_quasi_orth} tested with $v_h = (0, \phi_h) \in \mV_h$ to rewrite the second term as follows:
            \[
                \int_I a(\phi_h, \epsilon) = \int_I s_{1} (u_h, \phi_h) + s_{0}(\partial_t u_h, \phi_h)  - \langle \partial_t \epsilon, \phi_h \rangle_\G.
            \]
            The next step is to invoke Lemma \ref{lemma:technical_identity_l2l2est} (identity for the time derivative) to manipulate the last two terms with time derivative
            \[ 
              \int_I \|\epsilon\|_{L^2_\G}^2 = \int_I a(\phi - \phi_h, \epsilon) + s_{1} (u_h, \phi_h) - \frac{1}{2} \frac{d}{dt} a_*(\phi_h, \phi_h). 
            \]
            The last term can be further estimated using the definition  \eqref{eq:discrete_elliptic_with_error} of $\phi_h$ with $\psi_h = \phi_h$, followed by the choice $u_h(0) = P_h u_0$ and the definition of the stabilized $L^2$-projection \eqref{eq:stabilized_l2_projection} as follows:
            \[
            \begin{split}
                \int_I \frac{d}{dt} a_*(\phi_h, \phi_h) & = a_*(\phi_h, \phi_h)(T)
                - a_*(\phi_h, \phi_h)(0) \geq - a_*(\phi_h, \phi_h)(0) \\
                & = - [(\epsilon, \phi_h) - s_{0}(u_h, \phi_h)](0) = (P_h u_0 - u_0, \phi_h) + s_{0}(P_h u_0, \phi_h) = 0;
            \end{split}
            \]
            in retrospect, this justifies the definition \eqref{eq:discrete_elliptic_with_error}. Therefore, this calculation yields
            \[
            \int_I \|\epsilon\|_{L^2_\G}^2 \leq \int_I a(\phi - \phi_h, \epsilon) + s_{1} (\phi_h, u_h).
            \]
            Let us focus on the integrand on the right-hand side and use $\epsilon = u-u_h$ from now on
            to get
            \[
              a(\phi - \phi_h, u - u_h) + s_{1} (\phi_h, u_h) \leq \bE_{H^1_*}[\phi, \phi_h] \, \bE_{H^1_*}[u, u_h]
            \]
            upon using the Cauchy-Schwarz inequality. The first term on the right-hand side can be further estimated via Lemma \ref{lemma:error_est_perturbed_semidiscrete_elliptic} (error estimate for the dual problem) with $\chi = u - u_h|_\G, g_h = u_h$
            \[
                \bE_{H^1_*} [\phi, \phi_h]^2 \le C_6 h^2 \big( \|u - u_h\|_{L^2_\G}^2 + s_0(u_h, u_h) \big)
            \]
            whence
            \[\begin{split}
                a(\phi - \phi_h, u - u_h) + s_{1} (u_h, \phi_h) & \leq C h \, \bE_{H^1_*}[u, u_h] \big(\|u - u_h\|_{L^2_\G}^2 + h^2 s_{1}(u_h, u_h) \big)^{1/2} \\
                & \leq C h \, \bE_{H^1_*}[u, u_h] \big(\|u - u_h\|_{L^2_\G}^2 + h^2 \,\bE_{H^1_*}[u, u_h]^2 \big)^{1/2} \\
                & \leq C h \, \bE_{H^1_*}[u, u_h] \, \|u - u_h\|_{L^2_\G} + C h^2 \, \bE_{H^1_*}[u, u_h]^2.
            \end{split} \]
            Altogether, we obtain
            \[
            \begin{split}
                \int_I \|u - u_h\|_{L^2_\G}^2 & \leq C h \int_I \bE_{H^1_*}[u, u_h] \|u - u_h\|_{L^2_\G} + C h^2 \int_I \bE_{H^1_*}[u, u_h]^2,
            \end{split}
            \]
            which, after applying Young's inequality, in turn implies
            \[
              \int_I \|u - u_h\|_{L^2_\G}^2 \leq C(C+2) h^2 \int_I \bE_{H^1_*}[u, u_h]^2. 
            \]
            Lastly, the identity $s_0(u_h, u_h) = h^2 s_1(u_h, u_h) \leq \bE_{H^1_*}[u, u_h]^2$ integrated over $I$ completes the proof of assertion \eqref{eq:a_priori_error_est_l2l2}.
        \end{proof}

        We shall now present the proofs of the technical Lemmas \ref{lemma:error_has_vanishing_mean}, \ref{lemma:error_est_perturbed_semidiscrete_elliptic} and \ref{lemma:technical_identity_l2l2est}
        announced earlier.
%-------------------------------------------------------------------------------------
        \begin{proof}[Proof of Lemma \ref{lemma:error_has_vanishing_mean} (vanishing mean)]
%-------------------------------------------------------------------------------------        
            Let $\epsilon = u - u_h|_\G \in \mU$. Given $t\in (0,T]$ let $0< \delta < t$ and consider 
            the continuous piecewise constant function in time
            \[
               \chi_\delta(s) = 1\quad 0 \le s \le t-\delta;
               \qquad
               \chi_\delta(s) = \delta^{-1} (t-s) \quad t-\delta < s < t;
               \qquad
               \chi_\delta(s) = 0 \quad t \le s \le T.
            \]
            Since $\chi_\delta$ is constant in space, we can take $v_h=(0,\chi_\delta) \in \mathbb{V}_h$ as a test function in \eqref{eq:parabolic_quasi_orth} to obtain
            \[
              \int_I \langle \pa_t \epsilon, \chi_\delta \rangle_\G = 0 
              \quad\Rightarrow\quad 
              \frac{1}{\delta} \int_{t-\delta}^t \epsilon = ( \epsilon(0) ,1 )_\G 
              = ( u_0 - P_h u_0, 1)_\G = 0,
            \]
            according to the definition \eqref{eq:stabilized_l2_projection} of $P_h$.
            Since $\epsilon \in C_t^0 L^2_\G$ by Lions-Magenes Lemma \cite{ErnGuermond_2004}, we may take
            the limit $\delta\to 0$ to derive the desired property $(\epsilon(t),1)_\G=0$.
        \end{proof}

%------------------------------------------------------------------------------------- 
        \begin{proof}[Proof of Lemma \ref{lemma:error_est_perturbed_semidiscrete_elliptic} (error estimate for the dual problem)]
 %-------------------------------------------------------------------------------------

            We exploit the continuity and coercivity of the bilinear form $a_*$ on $V_{h,0}$ endowed
            with the norm $|\cdot|_{H^1_*}$. We first estimate $P_h\phi-\phi_h\in V_h$ using
            Galerkin quasi-orthogonality \eqref{eq:elliptic_quasi_orth} with $\psi_h = P_h \phi - \phi_h$
            \begin{align*}
                |P_h \phi - \phi_h|_{H^1_*}^2 & = a_*(P_h \phi - \phi_h, P_h \phi - \phi_h) + a_*(\phi_h - \phi^e, P_h \phi - \phi_h) + s_0(g_h, P_h \phi - \phi_h) \\
                & = a_*(P_h \phi - \phi^e, P_h \phi - \phi_h) + s_0 (g_h, P_h \phi - \phi_h),
            \end{align*}
            We next apply the Cauchy-Schwarz inequality to arrive at
            \[
            |P_h \phi - \phi_h|_{H^1_*}^2 \leq \sqrt2 \left(\bE_{H^1_*}[\phi, P_h \phi]^2 + h^2 s_0(g_h, g_h)\right)^{1/2} |P_h \phi - \phi_h|_{H^1_*},
            \]
            whence
            \[
             |P_h \phi - \phi_h|_{H^1_*}^2 \leq 2 \bE_{H^1_*}[\phi, P_h \phi]^2 + 2 h^2 s_0(g_h, g_h).
            \]
            Since $\phi-\phi_h = (\phi - P_h \phi) + (P_h \phi - \phi_h)$, and $|\phi-P_h\phi|_{H^1_*} \le
            \bE_{H^1_*}[\phi, P_h \phi]$, 
            we readily see that
            \[
              \bE_{H^1_*}[\phi, \phi_h]^2 \leq 6 \bE_{H^1_*}[\phi, P_h \phi]^2 + 2 h^2 s_0 (g_h, g_h).
            \]
            Finally, Lemma \ref{cor:Ph-h1-error-est} ($H^1_*$ error estimate for $P_h$) in 
            conjunction with \eqref{eq:H2-regularity} gives
            \[
               \bE_{H^1_*}[\phi, P_h \phi] \le C_1 h |\phi|_{H^2_\G} \le C_1 C_{\mathrm{reg}}  h 
               \|\chi\|_{L^2_\G}.
            \]
            This completes the proof of assertion \eqref{eq:error_estimate_perturbed_elliptic_problem}.
        \end{proof}
%-------------------------------------------------------------------------------------
        \begin{proof}[Proof of Lemma \ref{lemma:technical_identity_l2l2est} (identity for time
        derivative)]
%-------------------------------------------------------------------------------------

            We observe that \eqref{eq:discrete_elliptic_perturbed} is an elliptic problem with time-dependent right-hand side $\chi \in \mathbb{U} \cap C_t^0 L^2_{\G,0}$ and $g_h \in \mU_h$,
            which admits a unique solution $\phi_h \in H^1_t V_h$. 
            We formally differentiate \eqref{eq:discrete_elliptic_perturbed} in time using the fact that
            $\phi_h, g_h \in H^1_t V_h$ to get
            \begin{equation}\label{eq:adjoint_elliptic_semidiscrete_ddt}
                a_*(\pa_t \phi_h, \psi_h) = \frac{d}{dt} (\chi, \psi_h)_\G - s_{0}(\pa_t g_h, \psi_h), \quad \forall \psi_h \in V_{h,0}, \quad \text{for a.e. $t \in I$}.
            \end{equation}
            On the other hand, the Lions-Magenes lemma \cite{ErnGuermond_2004} guarantees that 
            $\chi \in \mathbb{U}$ is absolutely continuous in time and we can differentiate within the
            duality bracket, thereby getting for a.e. $t\in I$
            \begin{equation*} 
                \begin{split}
                    \frac{d}{dt} (\chi, \psi_h)_\G = \langle \pa_t \chi, \psi_h \rangle_\G \quad \forall \psi_h \in V_{h,0}.
                \end{split}
            \end{equation*}
            We test \eqref{eq:adjoint_elliptic_semidiscrete_ddt} with $\psi_h = \phi_h$ and integrate by parts, because $\phi_h$ is discrete, to arrive at
            \begin{equation*}
                \frac{1}{2} \frac{d}{dt} a_*(\phi_h, \phi_h) = \langle \pa_t \chi, \phi_h \rangle_\G + s_{0}(\pa_t g_h, \phi_h) \quad \text{for a.e. $t \in I$.}
            \end{equation*}
            This is the desired relation \eqref{eq:tech_est_duality_arg}.
        \end{proof}
        
        We conclude this subsection with an optimal $L^2_t L^2_*$ with minimal data regularity.
        \begin{corollary}[$L^2$ error estimate with minimal data regularity]
            Let $(u_0,f)\in\mathbb{D}$ and let $u \in \mU$ solve \eqref{eq:cont_problem} and $u_h \in \mU_h$ solve \eqref{eq:discrete_problem}. Then there exists a constant $C_7$ such that
            \begin{equation}
                \Big(\int_I \bE_{L^2_*}[u, u_h]^2 \Big)^{1/2} \le C_7 h \Vert (u_0, f)\Vert_3.
            \end{equation}
        \end{corollary}
        \begin{proof}
            In view of Theorem \ref{theorem:l2l2_error_estimate} (optimal error estimate in $L^2_t L^2_*$), triangle inequality, Proposition \ref{proposition:inf-sup_stability} (well-posedness of the continuous problem), and Theorem \ref{thm:inf-sup-stability} ($V_h$-dependent inf-sup stability):
            \[
            \begin{split}
                \int_I \bE_{L^2_*}[u,u_h]^2 & \le C_7 h^2 \int_I \bE_{H^1_*}[u,u_h]^2 \leq 2 C_7 h^2 \int_I |u|_{H^1_\G}^2 + |u_h|_{H^1_*}^2 \\
                & \leq 2 C_7 h^2 (\|u\|_1^2 + \|u_h\|_{1,h;V_h}^2) \leq 4 C_7 h^2 (c_b^-)^{-2} \|(u_0, f)\|_3^2,
            \end{split}
            \]
            where $\cbm$ is defined in \eqref{eq:Cb+cb-}. This concludes the proof.
        \end{proof}

%%%%%%%%%%%%%%%%%%%%%%%%%%%%%%%%%%%%%%%%%%%%%%%%%%%%%%%%%%%%%%%%%%%%%%%%%%%%%%%%%%%%%%%
\section{Algebraic stability} \sectionmark{Algebraic stability} \label{section:algebraic_stability} 
%%%%%%%%%%%%%%%%%%%%%%%%%%%%%%%%%%%%%%%%%%%%%%%%%%%%%%%%%%%%%%%%%%%%%%%%%%%%%%%%%%%%%%%

It turns out that adding the stabilizing term $s_0(\partial_t u_h,v_h)$ in \eqref{eq:intro:our_method_1}
to time derivative, namely
\begin{equation}\label{eq:stab-time}
  (\partial_t u_h, v_h)_* = (\partial_t u_h, v_h)_\G + s_0(\partial_t u_h, v_h),
\end{equation}
has not only a dramatic analytic effect, as described earlier in the paper, but also an algebraic
effect that we explore now. Discretizing \eqref{eq:intro:our_method_1} by any practical implicit scheme, such as BE, CN, TR-BDF2, BDF$k, k=1,..,6$, dG(0), leads to an elliptic equation of the form
\begin{equation}\label{eq:total_stab_bilin_form}
        \begin{split}
            u_h\in V_h: \quad
            B_*(u_h, v_h) := \frac{C}{\Delta t} (u_h, v_h)_* + a_*(u_h, v_h) 
            = (g,v_h) \quad \forall v_h \in V_h,
        \end{split}
\end{equation}
for a suitable constant $C>0$ where $\Delta t$ is the time step; hereafter $u_h$ denotes the new 
time iterate and the forcing $g$ depends on the scheme and contains the previous time iterates. The issue
at stake is the condition number of the corresponding matrix $\bB_*$ associated with $B_*$ in \eqref{eq:total_stab_bilin_form}, which reads
\begin{equation} \label{eq:matrices}
        \bB_* = \frac{1}{\Delta t} \big(\bM + h \bS\big) + \left(\bA + \frac{1}{\Delta t} \bS \right)
\end{equation}
provided we take $C=1$ for simplicity, where
\[
   \bM = (u_h, v_h)_{\G}, 
   \quad
   \bA = (\nabla_{\G} u_h, \nabla_{\G} v_h)_{\G},
   \quad
   \bS = (\nabla_{\bn} u_h, \nabla_{\bn} v_h)_{\Omega_h}
\]
are the usual mass and stiffness matrices as well as the stabilization matrix. We emphasize the
presence of $\frac{h}{\Delta t}\bS$ in \eqref{eq:matrices} that corresponds to the newly
added stabilization term $s_0$ in \eqref{eq:stab-time}. 

Time discretizations of TraceFEM so far deal with $(\partial_t u_h, v_h)_\G$ rather than $(\partial_t u_h, v_h)_*$ and lead to a matrix $\bB$ of the form
\begin{equation}\label{eq:old-B}
    \bB = \frac{1}{\Delta t} \bM + \left(\bA + \frac{1}{\Delta t} \bS \right).
\end{equation}
instead of $\bB_*$. The condition number estimate of $\bB$ from \cite{lehrenfeld_stabilized_2018} under the Assumption \ref{assumption:shape-reg-quasi-unif} is
sharp and reads
\begin{equation}\label{eq:cond-B}
   \kappa(\bB) \lesssim \frac{\Delta t}{h^2} + \frac{h^2}{\Delta t}.
\end{equation}
We see that $\kappa(\bB)$ is sensitive to small $\Delta t$ relative to $h^2$, which is expected for 
high-order time stepping schemes as well as adaptive space-time schemes. The sharpness of 
\eqref{eq:cond-B} can be verified by a simple numerical experiment. We consider $\Gamma$ to be the unit sphere and run a simulation with fixed mesh size $h=1/32$, $\mathbb{P}_1$ bulk finite elements, BDF1 discretization in time, and a high-order approximation of $\G$ to make
geometric error negligible. Fig.~\ref{fig:mot_ex} documents the growth of $\kappa(\bB)$ as a function of
$\Delta t$ as well as the rather insensitive dependence on $\Delta t$ of $\kappa(\bB_*)$,
that stabilizes asymptotically at a value of about $166$.
    \begin{figure}[ht]
        \centering
        \begin{tikzpicture}
            \newcommand\Nplot{10}
            \newcommand\samplesPlot{100}
            \begin{axis}[
                width=17cm,
                height=4cm,
                xmode=log,
                ymode=log,
                xlabel=$\Delta t$,
                ylabel=$\kappa$,
                legend style={nodes={scale=0.9, transform shape}},
                grid,
                mark options={scale=2}]
                \addplot[thick, solid, color=red, mark=square, mark options={thick, solid, fill=black}] table [y=cond, x=dt]{data/mot_ex_ref_5.dat};
                \addlegendentry{$\kappa(\bB)$}
                \addplot[thick, solid, color=blue, mark=o, mark options={thick, solid, fill=black}] table [y=cond, x=dt]{data/mot_ex_tot_5.dat};
                \addlegendentry{$\kappa(\bB_*)$}
                
                \addplot[thick, dashed, color=black,
                    domain=1e-15:1e-10, 
                    samples=\samplesPlot
                ]{1e-4*\Nplot*1/x};
                \addlegendentry{$1/\Delta t$}
            \end{axis}
        \end{tikzpicture}
        \caption{Asymptotic behavior as $\Delta t \to 0$ of the condition numbers $\kappa(\bB_*), \kappa(\bB)$ for a fixed $h=1/32$, , where $\bB_*$ is given in \eqref{eq:matrices} and $\bB$ in \eqref{eq:cond-B}.
        Note that $\kappa(\bB)$ grows as $(\Delta t)^{-1}$ whereas $\kappa(\bB_*)$ is about constant 
        with respect to $\Delta t$.}
        \label{fig:mot_ex}
    \end{figure}
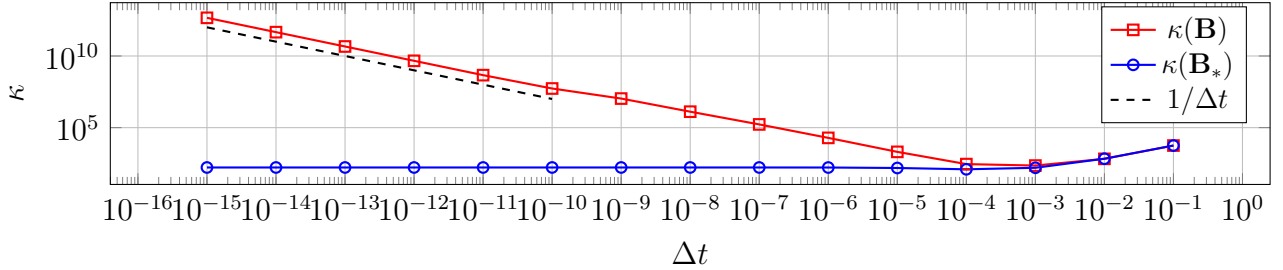
    We assess the behavior of $\kappa(\bB_*)$ as follows.
    
    \begin{lemma}[condition number of $\bB$] \label{lemma:conditioning}
        Let $\{\T_h\}_h$ be a quasi-uniform shape-regular mesh with mesh size $h$. The matrix $\bB_*$ in
        \eqref{eq:matrices} has largest and smallest eigenvalues and condition number satisfying
        \begin{equation}\label{eq:total_stab_lamb}
            \lambda_{\normalfont\textsc{max}}(\mathbf{B}_*) \lesssim 1 + \frac{h^2}{\Delta t}, \qquad \lambda_{\normalfont\textsc{min}}(\mathbf{B}_*) \gtrsim \frac{h^2}{\Delta t}, \qquad \kappa(\bB_*) \lesssim 1 + \frac{\Delta t}{h^2}.
        \end{equation}
    \end{lemma}
    \begin{proof}
        We first note the discrete norm equivalence
        \begin{equation}\label{eq:norm-equiv}
            \|v_h\|_{L^2(\Omega_h)}^2 \simeq h \|v_h\|_{L^2_*}^2 \simeq h^3 \|\bx\|_2^2 \quad\forall v_h\in V_h.
        \end{equation}
        where the vector $\bx$ denotes the nodal values of $v_h$.
        The first equivalence is due to Lemma \ref{prop:norm_equiv} (norm equivalence) which yields $\lesssim$, and Lemmas \ref{prop:trace_inequality} (trace inequality) and \ref{prop:inverse_estimate} (inverse estimate) which together give $\gtrsim$. The second
        equivalence is standard in dimension $3$; see \cite{lehrenfeld_stabilized_2018}.

        To prove the first part of \eqref{eq:total_stab_lamb}, we use \eqref{eq:norm-equiv} and Lemma \ref{lemma:refined_inverse_estimates} (refined inverse estimates)
        \[\begin{split}
            \textbf{x}^T \mathbf{B}_* \textbf{x} & = B_*(v_h, v_h) = \frac{1}{\Delta t} \|v_h\|_{L^2_*}^2 + |v_h|_{H^1_*}^2 \\
            & \lesssim \left( \frac{1}{\Delta t} + \frac{1}{h^2} \right) \|v_h\|_{L^2_*}^2 \simeq \left( \frac{1}{h\Delta t} + \frac{1}{h^3} \right) \|v_h\|_{L^2(\Omega_h)}^2
            \simeq \left( \frac{h^2}{\Delta t} + 1 \right) \|\textbf{x}\|_2^2.
        \end{split}\]
        This gives an upper bound on the largest eigenvalue $ \lambda_{\normalfont\textsc{max}}(\mathbf{B}_*)$. For the lower bound on the smallest eigenvalue $\lambda_{\normalfont\textsc{min}}(\mathbf{B}_*)$, we use \eqref{eq:norm-equiv} and the definition \eqref{eq:total_stab_bilin_form}
        of $B_*(u_h, v_h) = \bx^T \bB_* \bx$
        \begin{equation}\label{eq:cond_num_smallest_eigval}
            \begin{split}
                h^3 \|\textbf{x}\|_2^2 & \simeq \|v_h\|_{L^2(\Omega_h)}^2 \simeq h \|v_h\|_{L^2_*}^2 = \frac{1}{\Delta t} \|v_h\|_{L^2_*}^2 h \Delta t \leq h \Delta t B_*(v_h, v_h) = h \Delta t \, \bx^T \bB_* \bx,
            \end{split}    
        \end{equation}
        whence, $\bx^T \bB_* \bx \gtrsim \frac{h^2}{\Delta t} \|\textbf{x}\|_2^2$. The estimate on  $\kappa(\bB_*)= \frac{\lambda_{\normalfont\textsc{max}}(\mathbf{B}_*)}{\lambda_{\normalfont\textsc{min}}(\mathbf{B}_*)}$ readily follows.
    \end{proof}

    The improvement of the condition number in \eqref{eq:total_stab_lamb} relative to \eqref{eq:cond-B} comes
    from the lower bound of $\lambda_{\normalfont\textsc{min}}(\mathbf{B}_*)$. If we were to connect
    $\|v_h\|_{L^2(\Omega_h)}^2$ with $B(v_h, v_h)$, we could no longer rely on the augmented mass term
    $\|v_h\|_{L^2_*}^2 = (v_h,v_h)_*$ but rather on $(v_h,v_h)_\G$, which could be much smaller. The 
    difference is precisely $s_0(v_h,v_h)$, a minor but crucial correction missing in 
    \cite{lehrenfeld_stabilized_2018}. Otherwise, one could resort to
    the term $s_1(v_h,v_h)$ within $a_*(v_h,v_h)$ in \eqref{eq:total_stab_bilin_form} and use the relation
    $s_0(v_h,v_h) = h^2 s_1(v_h,v_h)$. This would lead to $\lambda_{\normalfont\textsc{min}}(\mathbf{B})\gtrsim 1$ for $\Delta t < h^2$, and so to \eqref{eq:cond-B}.
    
    Moreover, the augmented mass matrix $\mathbf{P}_* = \bM + h \bS$ in \eqref{eq:matrices}, corresponding to the stabilized $L^2_*$-projection $P_h$, has a condition
    number $\kappa(\mathbf{P}_*)\lesssim 1$. The proof is similar to that of Lemma \ref{lemma:conditioning}. Since $\kappa(\mathbf{P}_*)$ is uniformly bounded, setting the initial condition $u_h(0) = P_h u_0$ is a cheap operation.

%%%%%%%%%%%%%%%%%%%%%%%%%%%%%%%%%%%%%%%%%%%%%%%%%%%%%%%%%%%%%%%%%%%%%%%%%%%%%%%%%%%%
\section{Conclusions} \sectionmark{Conclusion} \label{section:conclusion}
%%%%%%%%%%%%%%%%%%%%%%%%%%%%%%%%%%%%%%%%%%%%%%%%%%%%%%%%%%%%%%%%%%%%%%%%%%%%%%%%%%%%
    The TraceFEM is a relatively new approach for surface PDEs. It hinges on two variational crimes:
    (a) it uses an ambient or bulk FE space for problems on manifolds $\G\subset\mathbb{R}^n$ of codimension 1; (b) it poses the discrete problem on an approximate surface $\G_h$ of $\G$ that is unfitted w.r.t the bulk mesh $\cT_h$. We focus on the surface heat equation as a prototype of linear parabolic PDE and on (a) because this is the most critical aspect of TraceFEM relative to parametric FEM. We thus assume no geometric error $\G=\G_h$, exact integration on $\G$, and no time discretization. For the resulting semidiscrete TraceFEM we reveal analytic and algebraic issues that hinder its performance. We fix them by adding a suitable normal derivative volume stabilization to the time derivative in Section \ref{sec:introduction} and studying a corresponding stabilized $L^2$ projection $P_h$ in Section \ref{section:stabilized_projections}. We develop a novel inf-sup theory in Section \ref{section:inf-sup} that establishes necessary and sufficient conditions in terms of $H^1$-stability of $P_h$ and an inverse inequality constant $C_{\text{inv},h}$ that accounts for the lack of conformity of TraceFEM. We then derive several consequences in Sections \ref{section:conseq} and \ref{section:error_est} provided the $H^1$-norm of $P_h$ and $C_{\text{inv},h}$ are uniformly bounded, properties shown to be valid in Section \ref{section:stabilized_projections} for a shape-regular and quasi-uniform sequence $\{\cT_h\}_h$. The by-products of the inf-sup theory are:
    \begin{itemize}
        \item 
        Uniform well-posedness of the semidiscrete problem;
        \item 
        Discrete maximal parabolic regularity;
        
        \item 
        Parabolic quasi-best approximation;
        
        \item 
        Convergence to minimal regularity solutions;
        
        \item 
        Optimal order-regularity parabolic and $L^2 L^2$ error estimates.
    \end{itemize}
    We also prove in Section \ref{section:algebraic_stability}  that the time derivative stabilization restores optimal condition number.

%%%%%%%%%%%%%%%%%%%%%%%%%%%%%%%%%%%%%%%%%%%%%%%%%%%%%%%%%%%%%%%%%%%%%%%%%%%%%%%%%%

\printbibliography 

%%%%%%%%%%%%%%%%%%%%%%%%%%%%%%%%%%%%%%%%%%%%%%%%%%%%%%%%%%%%%%%%%%%%%%%%%%%%%%%%%%
\end{document}